\documentclass[12pt]{amsart}

\usepackage[english]{babel}
\usepackage{amsmath}
\usepackage{amssymb}
\usepackage{amsfonts}
\usepackage{amsthm}
\usepackage{geometry}
\usepackage{listings}
\usepackage{xcolor}
\usepackage{algorithm}
\usepackage{algpseudocode}

\newcommand{\Sq}{\text{Sq}}

\usepackage{tikz}

\usetikzlibrary{positioning, arrows.meta, decorations.pathreplacing, shapes.geometric, calc}
\usepackage{tikz-qtree}

\usepackage[many]{tcolorbox} 

\newtcolorbox{outputbox}[1][]{
  breakable, 
  colback=black!5!white, 
  colframe=black!75!black, 
  fonttitle=\bfseries,
  sharp corners,
  listing only, 
  listing options={style=tcblatex, basicstyle=\ttfamily\small}, 
  title=Algorithm Output, 
  #1
}

\usepackage{hyperref}

\usepackage{pdflscape}

\newcommand{\F}{\mathbb{F}}

\newtheorem{theorem}{Theorem}[section]

\newtheorem{corollary}[theorem]{Corollary}

\theoremstyle{definition}

\newtheorem{example}[theorem]{Example}
\newtheorem{remark}[theorem]{Remark}

\definecolor{codegreen}{rgb}{0,0.6,0}
\definecolor{codegray}{rgb}{0.5,0.5,0.5}
\definecolor{codepurple}{rgb}{0.58,0,0.82}
\definecolor{backcolour}{rgb}{0.95,0.95,0.92}

\lstdefinestyle{SAGEMATHstyle}{
    backgroundcolor=\color{backcolour},   
    commentstyle=\color{codegreen},
    keywordstyle=\color{blue},
    numberstyle=\tiny\color{codegray},
    stringstyle=\color{codepurple},
    basicstyle=\ttfamily\footnotesize,
    breakatwhitespace=false,         
    breaklines=true,                 
    captionpos=b,                    
    keepspaces=true,                 
    numbers=none,                     
    numbersep=5pt,                  
    showspaces=false,                
    showstringspaces=false,
    showtabs=false,                  
    tabsize=2
}
\def\DD{D\kern-.7em\raise0.4ex\hbox{\char '55}\kern.33em}

\title[A Matrix Criterion and Algorithmic Approach for the Peterson Hit Problem (I)]{A Matrix Criterion and Algorithmic Approach\\ for the Peterson Hit Problem: Part I}

\author{\DD\d{\u a}ng V\~o Ph\'uc$^{*}$}
\address{Department of AI, FPT University, Quy Nhon AI Campus\\
An Phu Thinh New Urban Area, Quy Nhon City, Binh Dinh, Vietnam}
\email{dangphuc150488@gmail.com}

\begin{document}

\maketitle

\begin{abstract}
The Peterson hit problem in algebraic topology  is to explicitly determine the dimension of the quotient space $Q\mathcal P_k = \mathbb F_2\otimes_{\mathcal A}\mathcal P_k$ in positive degrees, where $\mathcal{P}_k$ denotes the polynomial algebra in $k$ variables over the field $\mathbb{F}_2$, considered as an unstable module over the Steenrod algebra $\mathcal{A}$. Current approaches to this problem still rely heavily on manual computations, which are highly prone to errors due to the intricate nature of the underlying calculations. To date, no efficient algorithm implemented in any computer algebra system has been made publicly available to tackle this problem in a systematic manner.

Motivated by the above context, in this work, which is considered as Part I of our project, we first establish a criterion based entirely on linear algebra for determining whether a given homogeneous polynomial is ``hit''. Accordingly, we describe the dimensions of the hit spaces. This leads to a practical and reliable computational method for determining the dimension of $Q\mathcal{P}_k$ for arbitrary $k$ and any positive degrees, with the support of a computer algebra system. We then give a concrete implementation of the obtained results as novel algorithms in \textsc{SageMath}. As an application, our algorithm demonstrates that the manually computed result presented in the recent work of Sum and Tai~\cite{Sum4} for the dimension of \(Q\mathcal{P}_5\) in degree $2^{6}$ is not correct. Furthermore, our algorithm determines that \(\dim(Q\mathcal{P}_5)_{2^{7}} = 1985\), which falls within the range \(1984 \leq \dim(Q\mathcal{P}_5)_{2^{7}} \leq 1990\) as estimated in~\cite{Sum4}.

\end{abstract}

\bigskip


\medskip


\section{Introduction}

The interplay between the Steenrod algebra and polynomial algebras is a cornerstone of modern algebraic topology, providing powerful computational tools for understanding homotopy theory. A central, long-standing question in this area is the ``hit problem'', first posed in a general setting by Franklin Peterson \cite{Peterson}. The foundational context of this problem can be summarized briefly as follows.  Let $\mathcal{P}_k = \mathbb{F}_2[x_1, \dots, x_k]$ be the polynomial algebra in $k$ variables over the finite field $\mathbb{F}_2$, graded by degree. The mod-2 Steenrod algebra, denoted $\mathcal{A}$, is a graded associative algebra over $\mathbb{F}_2$ generated by the Steenrod squares $Sq^i$ for $i \ge 0.$ There is a well-defined action of $\mathcal{A}$ on $\mathcal{P}_k$, turning $\mathcal{P}_k$ into a graded left $\mathcal{A}$-module. The cohomology of the $k$-fold product of infinite-dimensional real projective space, $H^*((\mathbb{R}P^\infty)^{\times k}; \mathbb{F}_2)$, which is isomorphic to $\mathcal{P}_k$ as an $\mathcal{A}$-module. A polynomial $f \in \mathcal{P}_k$ is said to be ``hit'' if it lies in the image of the augmentation ideal of the Steenrod algebra, $\mathcal{A}^+ = \ker(\epsilon: \mathcal{A} \to \mathbb{F}_2)$, where $\epsilon$ is the augmentation map. In other words, $f$ is hit if it can be written as a sum of Steenrod square operations on polynomials of lower degree:
$$  
f \in \operatorname{Im}(\mathcal{A}^+) \iff f = \sum_{i > 0} \operatorname{Sq}^i(g_i) \quad \text{for some } g_i \in \mathcal{P}_k.
 $$
The set of all hit polynomials forms a submodule of $\mathcal{P}_k$, denoted $\mathcal{A}^+\mathcal{P}_k$. The \textit{Peterson hit problem} asks for a minimal set of generators for $\mathcal{P}_k$ as a module over the Steenrod algebra $\mathcal{A}$. This is equivalent to finding a vector space basis for the graded quotient space $Q\mathcal{P}_k:=\mathbb{F}_2 \otimes_{\mathcal{A}} \mathcal{P}_k$, which is isomorphic to $\mathcal{P}_k / \mathcal{A}^+\mathcal{P}_k$. The elements of this basis represent the ``unhit'' or ``essential'' classes, which cannot be generated from lower-degree elements by the action of the Steenrod algebra.

The motivation for studying the hit problem is multifaceted. Geometrically, its solution provides insight into the attaching maps of cells in CW-complexes at the prime 2. Algebraically, it is deeply connected to the study of invariant theory, particularly in determining minimal generating sets for rings of invariants of subgroups of $GL(k, \mathbb{F}_2)$ acting on $\mathcal{P}_k$. Furthermore, its dual formulation has profound applications in homotopy theory via the Singer algebraic transfer, which relates invariants in polynomial rings to the cohomology of the Steenrod algebra itself.

Despite its importance, the problem remains largely open. Complete solutions, in the form of explicit dimension formulas for $Q\mathcal{P}_k$ at all degrees, are known only for a small number of variables. The cases for $k \le 4$ have been resolved through the foundational work of several authors (see the book by Walker and Wood \cite{Walker-Wood} and references therein). For $k \ge 5$, a general solution has not been found. Progress has been made by focusing on specific ``generic'' degrees, where the structure of the problem simplifies enough to be tractable. However, the problem in its full generality remains a formidable challenge.

A critical sub-problem within the hit problem is the task of determining whether a specific, given polynomial is hit. This step is fundamental because a basis for the quotient space $Q\mathcal{P}_k$ can be constructed by identifying a set of admissible monomials whose classes are linearly independent.

The primary difficulty lies in proving that a polynomial is \emph{not} hit. To do so, one must demonstrate that \emph{no possible} linear combination of $Sq^i(g_j)$ can equal the given polynomial. This is a non-trivial task. One approach that has been known to address this challenge uses criteria based on the structural properties of the $\mathcal{A}$-module $\mathcal{P}_k$. For example, a classical result by Wood \cite{Wood} states that $Q\mathcal{P}_k$ is trivial in positive degrees $d$ if $\alpha(d + k) > k$, where $\alpha(m = d+k)$ denotes the number of ones in the binary expansion of the degree $m.$ Other powerful tools include the analysis of Kameko's squaring homomorphism \cite{Kameko}, which provides a way to relate hit spaces at different degrees. However, current approaches all face a major obstacle: they solve this problem for a specific number of variables at specific degrees, corresponding to substantial manual computational workload, and errors are likely to occur when overlooking the verification capabilities of computer algebra tools. Therefore, to address this, the present work contributes directly to the development of algorithmic and computational methods that streamline the process and enhance reliability. Its primary contribution, which specifically addresses the challenge of identifying hit polynomials is the construction of a \textit{novel algorithm in \textsc{SageMath}} based on the general theoretical results established in this paper. By providing these powerful computational tools, this work takes a concrete step forward in solving the hit problem for five variables and beyond. It is also worth noting that Janfada, in \cite{Janfada}, proposed a criterion for identifying hit polynomials. Nevertheless, his result applies specifically to the case of three variables at a fixed degree and does not constitute a general criterion for determining hit polynomials.  

Based on the criterion we established for identifying hit polynomials, we are able to describe the structure of the hit space \( \mathcal{A}^+\mathcal{P}_k \) for any number of variables \( k \) and in every degree \( d \). This structural insight enables a comprehensive understanding of how the action of the Steenrod algebra generates the hit elements in the polynomial algebra $\mathcal P_k.$

As a direct application of this result, \textit{we derive an expression for the dimension of the quotient space \( Q\mathcal{P}_k \) in all positive degrees and for arbitrary values of \( k \).} This formula provides a concrete and computable invariant that characterizes the size of the space of non-hit polynomials, thereby offering a complete answer to the dimensional aspect of the hit problem.

Our algorithm also indicates that the dimension of \(Q\mathcal{P}_5\) in degree $2^{6}$, as computed by hand in the recent work of Sum and Tai~\cite{Sum4}, is not accurate. Furthermore, our algorithm determines that \(\dim(Q\mathcal{P}_5)_{2^{7}} = 1985\), which falls within the range \(1984 \leq \dim(Q\mathcal{P}_5)_{2^{7}} \leq 1990\) as estimated in~\cite{Sum4} (see Remark \ref{nxc} in Section \ref{s3}).

To the best of our knowledge, no published work to date has addressed an explicit algorithm for determining hit polynomials. Moreover, it is widely recognized that manually solving the hit problem becomes virtually infeasible as the number of variables and the degree increase, due to the enormous computational complexity and the limitations of available computational tools, which make errors nearly unavoidable. This situation underscores the importance of algorithmic solutions, the primary focus of our research contribution described above.

\textit{Our work can be considered a meaningful contribution to computer science, as we have achieved important initial progress in addressing a well-known and difficult problem in algebraic topology through computational algorithms.}

\section{Main Theorems and Corollaries}

Let $\mathcal{P}_k^d$ be the $\mathbb{F}_2$-vector space of homogeneous polynomials of degree $d$ in $k$ variables. Recall that a polynomial $f \in \mathcal{P}_k^d$ is hit if and only if it lies in the image of the augmentation ideal $\mathcal{A}^+.$ This means that $f$ can be expressed as a finite sum of the form:
$$ f = \sum_{i>0} Sq^i(g_i),$$
where each $Sq^i$ is a Steenrod square operator and each $g_i \in \mathcal{P}_k^{d-i}$ is a polynomial of appropriate degree. The set of all hit polynomials of degree $d$ forms an $\mathbb{F}_2$-vector subspace of $\mathcal{P}_k^d$, which we denote by $(\mathcal{A}^+\mathcal{P}_k)_d$.

The difficulty of the problem is to determine, for any given polynomial $f$, whether it is hit or not. The following theorem provides a definitive and computable criterion.

\begin{theorem}\label{dlc1}

 Let $\mathcal{B}_d = \{m_1, m_2, \dots, m_N\}$ be the basis of all monomials of degree $d$. This basis induces an isomorphism $\phi: \mathcal{P}_k^d \to \mathbb{F}_2^N$, where a polynomial $f$ is mapped to its coordinate vector $[f] := \phi(f)$. Let $\mathcal{H}_d (\mathcal{P}_k)= \{ Sq^i(g_j) \mid i > 0, \text{ and } g_j \text{ is a monomial of degree } d-i \}$ be a generating set for the subspace of hit polynomials $(\mathcal{A}^+\mathcal{P}_k)_d$. 

Construct a matrix $M$ whose columns are the coordinate vectors $[h]$ for each generator $h \in \mathcal{H}_d(\mathcal{P}_k)$ with respect to the basis $\mathcal{B}_d,$ consequently, a polynomial $f \in \mathcal{P}_k^d$ is a hit polynomial if and only if the linear system of equations $M\mathbf{c} = [f]$ has a solution for the coefficient vector $\mathbf{c}$.
\end{theorem}

\begin{proof}

The proof relies on the isomorphism between the space of polynomials and its coordinate vector space. Let all notations be as defined in the statement of the theorem.

\paragraph{$(\Rightarrow)$ Forward Implication:} Assume $f \in \mathcal{P}_k^d$ is a hit polynomial.
By definition, $f$ belongs to the subspace $(\mathcal{A}^+\mathcal{P}_k)_d$. Since $\mathcal{H}_d(\mathcal{P}_k)$ is a generating set for this subspace, $f$ can be written as a linear combination of elements from $\mathcal{H}_d(\mathcal{P}_k).$ That is, there exist coefficients $c_j \in \mathbb{F}_2$ such that:
$$ f = \sum_{j} c_j h_j,$$
where $h_j \in \mathcal{H}_d(\mathcal{P}_k)$. The mapping $p \mapsto [p]$ is a vector space isomorphism. Applying this linear map to both sides of the equation, we get:
$$ [f] = \left[ \sum_{j} c_j h_j \right].$$
By the linearity of the map, this becomes:
$$ [f] = \sum_{j} c_j [h_j].$$
This equation is precisely the definition of matrix-vector multiplication, $M\mathbf{c}$, where the columns of $M$ are the vectors $[h_j]$ and $\mathbf{c}$ is the column vector of coefficients $c_j$. Thus, we have shown that $M\mathbf{c} = [f]$, which means the linear system has a solution.

\paragraph{$(\Leftarrow)$ Reverse Implication:} Assume the linear system $M\mathbf{c} = [f]$ has a solution for some coefficient vector $\mathbf{c} = (c_1, c_2, \dots)^T$.
By the definition of matrix-vector multiplication, the existence of a solution means that the vector $[f]$ can be expressed as a linear combination of the columns of $M$:
$$ [f] = \sum_{j} c_j [h_j],$$
where $[h_j]$ are the columns of $M$. Again, by the linearity of the isomorphism $\phi$, we can combine the terms on the right-hand side: $ [f] = \left[ \sum_{j} c_j h_j \right].$ Since $\phi$ is an isomorphism (and thus injective), if the coordinate vectors are equal, the polynomials themselves must be equal $ f = \sum_{j} c_j h_j.$ By construction, every generator $h_j$ is an element of the hit subspace $(\mathcal{A}^+\mathcal{P}_k)_d$. Since this is a vector subspace, it is closed under linear combinations. Therefore, $f$, being a linear combination of elements of the subspace, must also be an element of the subspace, $ f \in (\mathcal{A}^+\mathcal{P}_k)_d.$ This, by definition, means that $f$ is a hit polynomial.
\end{proof}

\begin{remark}
While Theorem \ref{dlc1} holds for any generating set $\mathcal{H}_d(\mathcal{P}_k),$ the practical efficiency of an algorithm based on this theorem depends critically on the choice of this set and the implementation of the matrix $M.$ Our optimized algorithm, presented in Section \ref{s31} below, makes specific strategic choices---namely, using a filtered generating set based on the operators $Sq^{2^{j}}$ and representing $M$ as a sparse matrix---to make the problem computationally tractable for larger degrees and variable counts. These choices are efficient implementations of the general framework established by the theorem.
\end{remark}

\medskip

\begin{theorem}\label{dlc2}
Let $\phi$ be a linear isomorphism from the polynomial space $\mathcal{P}^{d}_k$ to a standard vector space $V = \mathbb{F}_2^N$, where $N$ is the number of monomials of degree $d$. Let $\mathcal{H}_d(\mathcal{P}_k)= \{ Sq^{2^j}(m) \mid j \ge 0, \ m \text{ is a monomial}, \deg(m) = d - 2^j \}$ be a spanning set for $(\mathcal{A}^+\mathcal{P}_k)_d.$ Construct a matrix $M$ whose columns are the vectors $\phi(h)$ for all $h \in \mathcal{H}_d(\mathcal{P}_k)$. Consequently, 
\[
\dim ((\mathcal{A}^+\mathcal{P}_k)_d) = \operatorname{rank}(M).
\]
\end{theorem}

\begin{proof}

Since \(\phi\) is an isomorphism, it preserves the structure of vector spaces, including subspaces and their dimensions. Therefore, the hit subspace \( (\mathcal{A}^+\mathcal{P}_k)_d \) is isomorphic to the subspace \( \phi((\mathcal{A}^+\mathcal{P}_k)_d) \subset V \), and their dimensions are equal:
\[ 
\dim((\mathcal{A}^+\mathcal{P}_k)_d) = \dim(\phi((\mathcal{A}^+\mathcal{P}_k)_d)).
\]
From Theorem \ref{dlc1}, we know that a polynomial \( f \in \mathcal{P}_k^d \) is hit if and only if the linear system \( M\mathbf{c} = \phi(f) = [f] \) has a solution. Thus, the image \( \phi((\mathcal{A}^+\mathcal{P}_k)_d) \) is exactly the subspace spanned by the vectors \( \phi(h) \) for \( h \in \mathcal{H}_d(\mathcal{P}_k)\).

By construction, the column space of the matrix \(M\) is 
\[ 
\text{ColumnSpace}(M) = \text{span}_{\mathbb{F}_2}\{\phi(h) \mid h \in \mathcal{H}_d(\mathcal{P}_k)\} = \phi((\mathcal{A}^+\mathcal{P}_k)_d). 
\]
A fundamental theorem of linear algebra states that the dimension of the column space of any matrix is equal to its rank:
\[ 
\dim(\text{ColumnSpace}(M)) = \operatorname{rank}(M).
\]
Combining the equalities from the steps above, we establish the desired result:
\[ 
\dim((\mathcal{A}^+\mathcal{P}_k)_d) = \dim(\phi((\mathcal{A}^+\mathcal{P}_k)_d)) = \dim(\text{ColumnSpace}(M)) = \operatorname{rank}(M). 
\]
Thus, the dimension of the hit space at degree \(d\) is precisely equal to the rank of the matrix \(M\) whose columns are the vector representations of the spanning set \(\mathcal{H}_d(\mathcal{P}_k)\).

\end{proof}

Theorem \ref{dlc2} provides a direct and fully algorithmic method to compute a key structural invariant of the module $\mathcal{P}_k$ over the Steenrod algebra. Instead of reasoning with abstract polynomials, one can construct a concrete matrix and use standard computational algorithms (e.g., Gaussian elimination) to find its rank, thereby determining the exact ``size'' of the hit subspace at any given degree.

\medskip

From the above theorems, several important consequences can be derived, shedding light on the structure of the hit problem.

\begin{corollary}\label{hq1}
If a non-zero hit polynomial admits a representation as a sum of Steenrod square actions, this representation is generally not unique.
\end{corollary}

\begin{proof}

By Theorem \ref{dlc1}, if \( f \in \mathcal{P}_k^d \) is a hit polynomial, then there exists at least one solution \( \mathbf{c} \in \mathbb{F}_2^N \) such that \( M\mathbf{c}= [f] \). We now consider the uniqueness of this representation.

The solution $\mathbf{c}$ to the linear system $M\mathbf{c} = [f]$ is unique if and only if the null space (or kernel) of the matrix $M$ is trivial, i.e., it contains only the zero vector. This, in turn, is equivalent to the columns of $M$ being linearly independent.

However, the generators of the hit space, $\{Sq^i(g_j)\}$, are generally a linearly dependent set. The Steenrod algebra $\mathcal{A}$ possesses a rich algebraic structure, including the well-known Adem relations (e.g., $Sq^1Sq^2 = Sq^3$). These relations imply the existence of non-trivial linear combinations of the generators that sum to the zero polynomial.

Each such non-trivial relation corresponds to a non-zero vector $\mathbf{c'}$ in the null space of $M$, satisfying $M\mathbf{c'} = \mathbf{0}$. If $\mathbf{c}$ is a particular solution to $M\mathbf{c} = [f]$, then for any $\mathbf{c'}$ in the null space, the vector $\mathbf{c} + \mathbf{c'}$ is also a solution:
$$ M(\mathbf{c} + \mathbf{c'}) = M\mathbf{c} + M\mathbf{c'} = [f] + \mathbf{0} = [f] $$
Since the null space of $M$ is almost always non-trivial, there exist multiple, and often infinitely many, distinct coefficient vectors $\mathbf{c}$ that produce the same polynomial $f$. Each such vector corresponds to a different explicit representation of $f$.
\end{proof}

Corollary \ref{hq1} confirms that the primary goal in solving the hit problem for a given polynomial is to establish the \textit{existence} of at least one representation. The search for a unique or \textit{canonical} representation is a much harder problem and is generally not required to prove that a polynomial is hit.

\begin{corollary}\label{hq2}
We have
$$ 
\dim(Q\mathcal{P}_k)_d= \dim\left( \mathcal{P}_k^d \big/ (\mathcal{A}^+ \mathcal{P}_k)_d \right) = \binom{d+k-1}{k-1} - \operatorname{rank}(M).
 $$
\end{corollary}

\begin{proof}
This is a direct application of the dimension theorem for vector spaces (also known as the rank-nullity theorem in a more general form). For any subspace $W$ of a finite-dimensional vector space $V$, the dimension of the quotient space $V/W$ is given by:
$$ \dim(V/W) = \dim(V) - \dim(W) $$
In our case, we identify $V = \mathcal{P}_k^d$ and $W = (\mathcal{A}^+\mathcal{P}_k)_d$. Consequently, the proof of the theorem rests on the following key observations:
\begin{enumerate}
    \item The dimension of the total space, $\dim(\mathcal{P}_k^d)$, is the number of distinct monomials of degree $d$ in $k$ variables. Hence, $$\dim \mathcal{P}_k^d  = \binom{d+k-1}{k-1}.$$
Indeed, any such monomial can be written in the form $ x_1^{a_1} x_2^{a_2} \cdots x_k^{a_k},$ where the exponents $a_1, a_2, \dots, a_k$ are non-negative integers. The condition that the polynomial is homogeneous of degree $d$ translates to the condition that the sum of the exponents must be exactly $d$:
$$ a_1 + a_2 + \dots + a_k = d, \quad \text{where each } a_i \in \{0, 1, 2, \dots\}.$$
Therefore, the problem of finding the dimension of the vector space $\mathcal{P}_k^d$ is equivalent to the combinatorial problem of counting the number of distinct non-negative integer solutions to the equation above. We can solve this counting problem using a classic and intuitive combinatorial technique known as ``\textit{stars}'' and ``\textit{bars}''.

Let's imagine we have $d$ identical items, which we will represent as ``\textit{stars}'' ($\star$). Our goal is to partition these $d$ stars into $k$ distinct groups. The number of stars in group $i$ will correspond to the value of the exponent $a_i$. To separate a sequence of items into $k$ groups, we need exactly $k-1$ dividers, which we will represent as ``\textit{bars}'' ($|$).

Let's find the dimension of $\mathcal{P}_3^4$, which is the space of homogeneous polynomials of degree $d=4$ in $k=3$ variables ($x_1, x_2, x_3$). This corresponds to finding the number of non-negative integer solutions to the equation: $ a_1 + a_2 + a_3 = 4.$ In our analogy, we have $d=4$ stars ($\star\star\star\star$) and we need to place them into $k=3$ groups, which requires $k-1=2$ bars ($||$).

Every unique arrangement of these stars and bars corresponds to a unique solution, and thus a unique monomial:
\begin{itemize}
    \item The arrangement \texttt{$\star\star|\star|\star$} corresponds to:
    \begin{itemize}
        \item 2 stars in the first group ($a_1=2$).
        \item 1 star in the second group ($a_2=1$).
        \item 1 star in the third group ($a_3=1$).
    \end{itemize}
    This gives the solution $(2, 1, 1)$, which corresponds to the monomial $x_1^2 x_2x_3.$

    \item The arrangement \texttt{$|\star\star\star|\star$} corresponds to:
    \begin{itemize}
        \item 0 stars in the first group ($a_1=0$).
        \item 3 stars in the second group ($a_2=3$).
        \item 1 star in the third group ($a_3=1$).
    \end{itemize}
    This gives the solution $(0, 3, 1)$, corresponding to the monomial $x_2^3 x_3$. Note how an empty group naturally represents a zero exponent.

    \item The arrangement \texttt{$\star\star\star\star||$} corresponds to:
    \begin{itemize}
        \item 4 stars in the first group ($a_1=4$).
        \item 0 stars in the second group ($a_2=0$).
        \item 0 stars in the third group ($a_3=0$).
    \end{itemize}
    This gives the solution $(4, 0, 0)$, corresponding to the monomial $x_1^4$.
\end{itemize}

From the analogy, we can see that counting the number of solutions is equivalent to counting the number of distinct ways to arrange a sequence containing $d$ stars and $(k-1)$ bars. We have a total of $d + (k-1)$ positions in our sequence. The problem then becomes choosing which of these positions will be filled with bars (the rest will automatically be filled with stars).

The number of ways to choose $k-1$ positions for the bars from a total of $d+k-1$ available positions is given directly by the binomial coefficient $\binom{d+k-1}{k-1}.$ Equivalently, we could choose $d$ positions for the stars from the $d+k-1$ total positions, which gives $\binom{d+k-1}{d}$. Due to the symmetry of binomial coefficients, these two formulas are identical:
$$ \binom{d+k-1}{k-1} = \binom{d+k-1}{(d+k-1)-(k-1)} = \binom{d+k-1}{d}.$$

    \item From Theorem \ref{dlc2}, the dimension of the subspace, $\dim\left( (\mathcal{A}^+ \mathcal{P}_k)_d \right)$, is equal to $\operatorname{rank}(M)$.
\end{enumerate}
Substituting these into the formula yields the desired result.
\end{proof}

Corollary \ref{hq2} is arguably the most powerful consequence, as it directly addresses the central question of the Peterson hit problem: determining the dimension of the space of ``unhit'' elements. It transforms this abstract algebraic question into a concrete computational task: calculate the number of monomials, construct the matrix $M$, find its rank, and subtract. \textit{Therefore, the algorithm described in Theorem \ref{dlc2} thus becomes a complete tool for solving for this dimension.}

\section{Algorithmic implementation}\label{s3}

\subsection{Algorithm for determining hit polynomials}\label{s31}

\emph{}

\medskip

While Theorem~\ref{dlc1} provides a theoretical framework, a direct implementation is computationally infeasible due to the combinatorial explosion in the size of the resulting linear system. To overcome this challenge, the \textsc{SageMath} implementation presented in this section utilizes an \textbf{adaptive hybrid model}. The adaptive workflow, which intelligently chooses between fast serial execution for small problems and massively parallel processing for large ones, is detailed in the flowchart below. Its key technical aspects are explained, followed by the complete code in LISTING \ref{lst:SAGEMATH_code}.

\medskip

\noindent\rule{\textwidth}{0.4pt}
\begin{center}
\textbf{Algorithm 1: Determining if a polynomial is hit (Hybrid Approach)}
\end{center}
\noindent\rule{\textwidth}{0.4pt}
\noindent\textbf{Input:} A homogeneous polynomial $f \in \mathcal{P}_k^d$.\\
\textbf{Output:} \texttt{True} if $f$ is hit (with an explicit representation), \texttt{False} otherwise.
\vspace{0.5em}

\noindent\textit{$\blacktriangleright$ Phase 1: Setup}
\begin{enumerate}
    \item Define the polynomial ring $\mathcal{P}_k= \mathbb{F}_2[x_1, \dots, x_k]$.
    \item Generate the monomial basis $\mathcal{B}_d$ for the vector space $\mathcal{P}_k^d$.
    \item Create an isomorphism $\phi: \mathcal{P}_k^d \to \mathbb{F}_2^N$ by mapping each monomial to a standard basis vector.
\end{enumerate}

\noindent\textit{$\blacktriangleright$ Phase 2: Adaptive Generator Construction}
\begin{enumerate}
    \setcounter{enumi}{3}
    \item Initialize an empty list for tasks, $T$.
    \item Initialize an empty list for hit generators, $\mathcal{H}'_d(\mathcal{P}_k)$.
    \item \textit{-- Gather all potential generator tasks --}
    \item \textbf{for all} Steenrod operators $Sq^{2^j}$ where $d-2^j \ge 0$ \textbf{do}
    \begin{enumerate}
        \item Let $d' = d - 2^j$.
        \item \textbf{if} $\alpha(d'+k) \le k$ \textbf{then} \textit{$\blacktriangleright$ Apply theoretical filter}
        \begin{enumerate}
            \item \textbf{for all} monomials $g \in \mathcal{P}_k^{d'}$ \textbf{do}
            \item[] \quad Append task $(Sq^{2^j}, g)$ to $T$.
            \item \textbf{end for}
        \end{enumerate}
        \item \textbf{end if}
    \end{enumerate}
    \item \textbf{end for}
    \item \textit{-- Decide execution strategy based on the number of tasks --}
    \item \textbf{if} size of $T < \texttt{THRESHOLD}$ \textbf{then} \textit{$\blacktriangleright$ Execute in Serial Mode}
    \begin{enumerate}
        \item \textbf{for each} task $(Sq^k, g)$ in $T$ \textbf{do}
        \item[] \quad $h \gets Sq^k(g)$.
        \item[] \quad \textbf{if} $h \neq 0$ \textbf{then} Append the triple $(h, Sq^k, g)$ to $\mathcal{H}'_d(\mathcal{P}_k)$.
        \item[] \quad \textbf{end if}
        \item \textbf{end for}
    \end{enumerate}
    \item \textbf{else} \textit{$\blacktriangleright$ Execute in Parallel Mode}
    \begin{enumerate}
        \item Initialize a pool of parallel workers.
        \item Distribute tasks in $T$ among workers to compute a list of results $R$.
        \item $\mathcal{H}'_d(\mathcal{P}_k) \gets$ Filtered non-zero results from $R$.
    \end{enumerate}
    \item \textbf{end if}
\end{enumerate}

\noindent\textit{$\blacktriangleright$ Phase 3 $\&$ 4: Sparse System Construction and Solution}
\begin{enumerate}
    \setcounter{enumi}{10}
    \item Construct a sparse matrix $M$ where column $j$ is the coordinate vector $[h_j] = \phi(h_j)$ for each generator $(h_j, \dots, \dots) \in \mathcal{H}'_d(\mathcal{P}_k)$.
    \item Compute the target coordinate vector $[f] = \phi(f)$.
    \item \textbf{try} to solve the linear system $M\mathbf{c} = [f]$ for a coefficient vector $\mathbf{c}$.
\end{enumerate}

\noindent\textit{$\blacktriangleright$ Phase 5: Result Analysis}
\begin{enumerate}
    \setcounter{enumi}{13}
    \item \textbf{if} a solution $\mathbf{c}$ exists \textbf{then}
    \begin{enumerate}
        \item Construct the explicit representation $f = \sum_j c_j h_j$.
        \item \textbf{return} \texttt{True}, representation.
    \end{enumerate}
    \item \textbf{else}
    \begin{enumerate}
        \item \textbf{return} \texttt{False}.
    \end{enumerate}
    \item \textbf{end if}
\end{enumerate}

\medskip

The algorithm's efficiency stems from its adaptive use of an optimized generating subset, denoted $\mathcal{H}'_d(\mathcal{P}_k)$. Instead of a single fixed method, the algorithm first compiles a list of all potential generator tasks that satisfy the necessary conditions (i.e., using operators of the form $Sq^{2^j}$ and meeting the admissibility condition $\alpha(d' + k) \le k$). It then evaluates the scale of the problem by counting the total number of these tasks. For small-scale problems, where the overhead of parallelization would be detrimental, it automatically employs a fast \textbf{serial loop}. Conversely, for large-scale problems, it seamlessly transitions to a \textbf{parallel mode}, distributing the computationally intensive calculations across all available CPU cores to dramatically reduce execution time. This hybrid strategy ensures optimal performance across the entire spectrum of problem sizes.

\medskip

\noindent\rule{\textwidth}{0.4pt}
\begin{center}
\textbf{\textbf{Detailed SageMath code for the algorithm}}
\end{center}
\noindent\rule{\textwidth}{0.4pt}

\medskip

\begin{lstlisting}[language=Python, caption={\textbf{\textsc{SageMath}-based algorithm for detecting hit polynomials}}, label={lst:SAGEMATH_code}]
from sage.all import *
from collections import defaultdict
from multiprocessing import Pool
import os

def alpha(n):
    """Computes the sum of the bits in the binary expansion of an integer n."""
    return sum(int(bit) for bit in bin(n)[2:])

def get_sq_function(P):
    """Factory function that creates an optimized `sq` function with its own cache."""
    memo = {}
    def sq_recursive(k, mono):
        state = (k, mono)
        if state in memo:
            return memo[state]
        if k == 0:
            return mono
        if mono == 1:
            return P(1) if k == 0 else P(0)
        gens = P.gens()
        first_var_index = -1
        for i, v in enumerate(gens):
            if mono.degree(v) > 0:
                first_var_index = i
                break
        v = gens[first_var_index]
        e = mono.degree(v)
        rest_of_mono = mono // (v**e)
        total = P(0)
        for i in range(k + 1):
            if binomial(e, i) % 2 == 1:
                total += v**(e + i) * sq_recursive(k - i, rest_of_mono)
        memo[state] = total
        return total
    def sq(k, f):
        if f.is_zero() or k < 0:
            return P(0)
        if k == 0:
            return f
        total_sum = P(0)
        for mono_tuple, coeff in f.dict().items():
            if coeff % 2 == 1:
                monomial_object = P.monomial(*mono_tuple)
                total_sum += sq_recursive(k, monomial_object)
        return total_sum
    return sq

# ---- FUNCTIONS FOR PARALLEL PROCESSING (IMPROVED) ----
worker_P = None
worker_mono_map = None
worker_sq = None

def init_worker(ring, a_map):
    """Initializer for each worker process."""
    global worker_P, worker_mono_map, worker_sq
    worker_P = ring
    worker_mono_map = a_map
    worker_sq = get_sq_function(worker_P)

def worker_task(task):
    """
    Performs one unit of work.
    It now receives an exponent tuple which is more efficient than a string.
    """
    k_op, g_exponents = task
    g = worker_P.monomial(*g_exponents) # Recreate monomial from exponents
    hit_poly = worker_sq(k_op, g)
    if not hit_poly.is_zero():
        row_indices = []
        for m_tuple, c in hit_poly.dict().items():
            monomial_m = worker_P.monomial(*m_tuple)
            if monomial_m in worker_mono_map:
                row_indices.append(worker_mono_map[monomial_m])
        if row_indices:
            return (row_indices, k_op, g_exponents) # Return exponents
    return None

# ----------------- MAIN DRIVER FUNCTION (HYBRID APPROACH) -----------------
def show_hit_representation_hybrid(poly, PARALLEL_TASK_THRESHOLD=10000):
    """
    The main driver function using a hybrid approach.
    - Runs in SERIAL mode for small problems (faster due to no overhead).
    - Runs in PARALLEL mode for large problems to leverage multiple cores.
    """
    # --- Step 1: Common Initialization and Setup ---
    P = poly.parent()
    d = poly.degree()
    k_vars = P.ngens()
    if d == 0:
        print(f"Polynomial ({poly}) is of degree 0 and is not considered 'hit'.")
        return False
    print(f"Analyzing polynomial of degree {d} in {k_vars} variables...")
    all_monos_d = list(P.monomials_of_degree(d))
    V = VectorSpace(GF(2), len(all_monos_d))
    mono_map = {m: i for i, m in enumerate(all_monos_d)}
    print(f"Vector space created with dimension {len(all_monos_d)}.")

    # --- Step 2: Prepare tasks and decide execution mode ---
    tasks = []
    sq_powers = [2**i for i in range(d.nbits()) if 2**i > 0]
    for k_op in sq_powers:
        deg_g = d - k_op
        if deg_g < 0 or alpha(deg_g+ k_vars) > k_vars:
            continue
        for g in P.monomials_of_degree(deg_g):
            tasks.append((k_op, g))
    total_tasks = len(tasks)
    print(f"Total generator tasks to process: {total_tasks}")
    hit_space_info = []

    # --- CHOOSE EXECUTION MODE ---
    if total_tasks < PARALLEL_TASK_THRESHOLD:
        # --- SERIAL MODE (for small problems) ---
        print(f"Number of tasks ({total_tasks}) is below threshold ({PARALLEL_TASK_THRESHOLD}). Running in fast SERIAL mode.")
        sq = get_sq_function(P)
        for k_op, g in tasks:
            hit_poly = sq(k_op, g)
            if not hit_poly.is_zero():
                row_indices = []
                for m_tuple, c in hit_poly.dict().items():
                    monomial_m = P.monomial(*m_tuple)
                    if monomial_m in mono_map:
                        row_indices.append(mono_map[monomial_m])
                if row_indices:
                    hit_space_info.append((row_indices, k_op, g))
    else:
        # --- PARALLEL MODE (for large problems) ---
        print(f"Number of tasks ({total_tasks}) is above threshold. Running in PARALLEL mode.")
        num_workers = os.cpu_count()
        print(f"Starting parallel computation with {num_workers} worker processes...")
        exponent_tasks = [(k, g.exponents()) for k, g in tasks]
        with Pool(initializer=init_worker, initargs=(P, mono_map)) as pool:
            results = pool.map(worker_task, exponent_tasks)
        print("Parallel computation finished. Collecting results...")
        for res in results:
            if res is not None:
                row_indices, k_op, g_exponents = res
                g = P.monomial(*g_exponents) # Recreate monomial object
                hit_space_info.append((row_indices, k_op, g))

    # --- Step 3 onwards: Common logic for matrix construction and solving ---
    if not hit_space_info:
        print("-" * 20)
        print(f"FAILURE: Polynomial ({poly}) is not a 'hit'.")
        print("(The hit space generated is trivial, so no solution can exist).")
        return False
    print("Constructing the sparse matrix...")
    matrix_dict = {}
    for j, (row_indices, _, _) in enumerate(hit_space_info):
        for row_idx in row_indices:
            matrix_dict[(row_idx, j)] = 1
    nrows = len(all_monos_d)
    ncols = len(hit_space_info)
    M = Matrix(GF(2), nrows, ncols, matrix_dict, sparse=True)
    num_non_zero = len(M.nonzero_positions())
    print(f"Constructed a sparse matrix of size {M.nrows()}x{M.ncols()} with {num_non_zero} non-zero entries.")

    target_vector = V.zero_vector()
    for m_tuple, c in poly.dict().items():
        if c % 2 == 1:
            monomial_m = P.monomial(*m_tuple)
            if monomial_m in mono_map:
                target_vector[mono_map[monomial_m]] = 1

    print("Solving the linear system... (This may still take a long time for huge systems)")
    try:
        coeffs = M.solve_right(target_vector)
        grouped_terms = defaultdict(list)
        for i, c in enumerate(coeffs):
            if c == 1:
                _, k, g = hit_space_info[i]
                grouped_terms[k].append(g)
        representation_parts = []
        for k in sorted(grouped_terms.keys()):
            monomials_str = " + ".join(str(m) for m in grouped_terms[k])
            representation_parts.append(f"Sq^{k}({monomials_str})")
        representation_str = " + ".join(representation_parts)
        print("-" * 20)
        print(f"SUCCESS: Polynomial ({poly}) is a 'hit'.")
        print(f"Explicit representation:\n{poly} = {representation_str}")
        print("-" * 20)
        return True
    except ValueError:
        print("-" * 20)
        print(f"FAILURE: Polynomial ({poly}) is not a 'hit'.")
        print("(The linear system has no solution with this set of generators).")
        print("-" * 20)
        return False
\end{lstlisting}

\medskip

For readers to easily access this tool, we explain some key aspects of the \textbf{hybrid algorithm}, which intelligently switches between serial and parallel execution, as follows:

\begin{itemize}
    \item \textbf{The context ($\mathcal{P}_k^d$ and $\mathcal{B}_d$):} 
    The code first defines the polynomial ring \texttt{P} and then creates the basis 
    \texttt{all\_monos\_d} and the mapping \texttt{mono\_map}. These components 
    establish the vector space $\mathcal{P}_k^d$ and its isomorphism to 
    $\mathbb{F}_2^N$. This foundational setup is performed regardless of the 
    subsequent execution mode and remains fundamental to the linear algebra approach.

    \item \textbf{A Hybrid Strategy for the Generating Set ($\mathcal{H}'_d(\mathcal{P}_k)$):} 
    This is the most significant enhancement. While the core filtering strategy using 
    $\operatorname{Sq}^{2^j}$ (via \texttt{sq\_powers}) and the \texttt{alpha()} 
    function is retained, the execution model is now dynamic:
    \begin{enumerate}
        \item First, the algorithm gathers a complete list of all potential 
        generator \textit{tasks} that pass the initial filters.
        
        \item It then compares the total number of tasks against a configurable threshold, denoted by \texttt{PARALLEL\_TASK\_THRESHOLD}.
        
        \item \textbf{If the task count is low}, it executes in a fast \textbf{serial mode}, 
        iterating directly through the tasks. This avoids the overhead of 
        parallelization for small problems where serial execution is faster.
        
        \item \textbf{If the task count is high}, it automatically switches to a 
        highly scalable \textbf{parallel mode}. The list of tasks is distributed 
        among all available CPU cores using Python's \texttt{multiprocessing} 
        library, dramatically accelerating what is typically the primary bottleneck 
        of the computation.
    \end{enumerate}

    \item \textbf{Efficient Data Handling and Sparse Matrix ($M$):} 
    To handle potentially vast dimensions, the algorithm still populates a Python 
    dictionary, \texttt{matrix\_dict}, with only the non-zero entries. The results 
    from the computation (whether from the serial loop or the parallel workers) 
    populate this dictionary. To ensure robustness and efficiency in the parallel 
    path, monomial data is transferred as lightweight \textbf{exponent tuples} 
    rather than complex objects. The final matrix is then constructed in a single, 
    memory-efficient operation via \texttt{M = Matrix(GF(2), ..., matrix\_dict, sparse=True)}.

    \item \textbf{The system ($M\mathbf{c} = [f]$):} 
    This step remains unchanged. The input \texttt{poly} is converted into its 
    \texttt{target\_vector}, $[f]$. The command \texttt{M.solve\_right(target\_vector)} 
    is used to solve the linear system $M\mathbf{c} = [f]$. Because \texttt{M} is a 
    sparse matrix, \textsc{SageMath} automatically dispatches this task to a powerful 
    iterative solver (such as the Wiedemann algorithm) that is specifically 
    designed for large, sparse systems over finite fields.

    \item \textbf{The result:} 
    The success or failure of the \texttt{solve\_right} command remains the core 
    of the determination. A successful return provides the coefficient vector 
    \texttt{coeffs} used to build the explicit representation, while a 
    \texttt{ValueError} exception indicates that no solution exists with the given 
    set of generators, meaning the polynomial is not hit.
\end{itemize}

\medskip

\begin{example}
\begin{itemize}

\item[(i)] We illustrate the above algorithm with the polynomial: $$f = x_1^2x_2^5x_3^4x_4^3 + x_1^3x_2^3x_3^4x_4^4 + x_1^3x_2^2x_3^4x_4^5 + x_1^2x_2^3x_3^4x_4^5 + x_1^3x_2^4x_3x_4^6\in \mathcal P_4^{14}.$$ The corresponding \textsc{SageMath} code implementing the above-described algorithm for this polynomial is given below. (Immediately after ``\texttt{return False}'' in the last line of the algorithm above, we insert the following line of code.)

\medskip

\begin{lstlisting}[language=Python]
if __name__ == '__main__':
    
    # 1. Setup the Polynomial Ring for 4 variables
    P_user = PolynomialRing(GF(2), names=['x1', 'x2', 'x3', 'x4'])
    
    # 2. Define the specific polynomial string
    poly_user_str = "x1^2*x2^5*x3^4*x4^3 + x1^3*x2^3*x3^4*x4^4 + x1^3*x2^2*x3^4*x4^5 + x1^2*x2^3*x3^4*x4^5 + x1^3*x2^4*x3*x4^6"
    
    # 3. Create the polynomial object from the string
    poly_user = P_user(poly_user_str)
    
    print(f"Testing with polynomial: {poly_user}")
    
    # 4. Run the hybrid function
    # The function will automatically choose the best mode for this case.
    show_hit_representation_hybrid(poly_user)
\end{lstlisting}

\medskip

The output confirms that the polynomial $f$ is hit, and it admits an explicit representation of the form $f = \sum_{i > 0}Sq^{i}(f_i),$ as follows:

\begin{align*}
f&= x_1^2 x_2^5 x_3^4 x_4^3 + x_1^3 x_2^3 x_3^4 x_4^4 + x_1^3 x_2^2 x_3^4 x_4^5 + x_1^2 x_2^3 x_3^4 x_4^5 + x_1^3 x_2^4 x_3 x_4^6 \\
&= \operatorname{Sq}^1\big( x_1 x_2^2 x_3^2 x_4^8 + x_1 x_2^3 x_3 x_4^8 + x_1 x_2^2 x_3^4 x_4^6 + x_1 x_2^4 x_3^2 x_4^6 + x_1^3 x_2^3 x_3 x_4^6 \\
&\qquad + x_1 x_2^3 x_3^4 x_4^5 + x_1^3 x_2^3 x_3^2 x_4^5 + x_1 x_2^6 x_3 x_4^5 + x_1 x_2^4 x_3^4 x_4^4 + x_1^5 x_2^2 x_3^2 x_4^4 + x_1^3 x_2^3 x_3^4 x_4^3 \big) \\
&\quad + \operatorname{Sq}^2\big( x_1 x_2^2 x_3 x_4^8 + x_1 x_2^3 x_3^2 x_4^6 + x_1^2 x_2^3 x_3 x_4^6 + x_1 x_2^2 x_3^4 x_4^5 + x_1^2 x_2^3 x_3^2 x_4^5 \\
&\qquad + x_1^3 x_2^2 x_3^2 x_4^5 + x_1 x_2^5 x_3 x_4^5 + x_1^2 x_2^3 x_3^4 x_4^3 + x_1^3 x_2^2 x_3^4 x_4^3 + x_1^3 x_2^4 x_3^2 x_4^3 \big) \\
&\quad + \operatorname{Sq}^4\big( x_1^3 x_2^2 x_3^2 x_4^3 \big).
\end{align*}

\medskip

\item[(ii)] We now consider another example to illustrate the algorithm, applied to the polynomial $$g = x_1x_2^2x_3^2x_4 + x_1x_2^2x_3x_4^2 + x_1x_2x_3^2x_4^2\in \mathcal{P}_4^{6}.$$ For this polynomial, the algorithm readily produces the following output:
$$ g = Sq^{1}(x_1x_2x_3x_4^{2} +x_1x_2x_3^{2}x_4 + x_1x_2^{2}x_3x_4) + Sq^{2}(x_1x_2x_3x_4).$$
However, by a simple observation, we can see that
$$ g = Sq^{1}(x_1^{2}x_2x_3x_4) + Sq^{2}(x_1x_2x_3x_4).$$
Hence, $g$ is a hit polynomial with multiple decompositions, illustrating Corollary~\ref{hq1}.
\end{itemize}
{\bf Note.} Both illustrative examples provided above can be checked directly via the \textsc{SageMath} online interface at \url{https://sagecell.sagemath.org/}. 
\end{example}

\medskip

\subsection{Algorithm for computing hit space dimensions}\label{s32}

\emph{}

\medskip

Translating the matrix-based criterion from Theorem \ref{dlc2} into a practical and efficient computational method requires overcoming two primary challenges: the immense memory required for the hit matrix and the prohibitive CPU time needed to compute its entries. The algorithm described in LISTING \ref{lst:SAGEMATH_code_2} below therefore employs a dual-pronged optimization strategy to render the problem tractable even at high degrees. First, it leverages a sparse matrix representation to avoid prohibitive memory costs. Second, and most critically for performance, it implements an adaptive hybrid computation model that intelligently switches between a fast serial process for small-scale problems and massively parallel execution for large-scale ones. The key stages of this highly optimized process, from initial setup and adaptive computation to the final rank determination, are broken down in detail below.

\medskip

\begin{center}
\scalebox{0.6}{
\begin{tikzpicture}[
    transform shape,
    node distance=1.5cm and 1.5cm,
    phase/.style={
        rectangle, draw=orange!80, thick,
        minimum width=4cm, minimum height=1.1cm, text width=3.8cm,
        align=center, font=\sffamily\small\bfseries, fill=orange!10, rounded corners=3pt
    },
    detail/.style={
        rectangle, draw=green!70,
        minimum width=2.6cm, minimum height=0.9cm, text width=2.4cm,
        align=center, font=\sffamily\scriptsize, fill=green!10, rounded corners=2pt
    },
    decision/.style={
        diamond, aspect=1.5, draw=red!80, thick,
        minimum width=2.5cm, minimum height=1cm, text width=2.2cm,
        align=center, font=\sffamily\scriptsize\bfseries, fill=red!10
    },
    path_box/.style={
        rectangle, draw=cyan!70!black, thick,
        minimum width=2.8cm, minimum height=1.5cm, text width=2.6cm,
        align=center, font=\sffamily\footnotesize\itshape, fill=cyan!10, rounded corners=3pt
    },
    result/.style={
        rectangle, draw=purple!80,
        minimum width=4.5cm, minimum height=1cm, text width=4.2cm,
        align=center, font=\sffamily\footnotesize\bfseries, fill=purple!15, rounded corners=3pt
    },
    highlighted/.style={ 
        fill=yellow!25, draw=orange!80
    },
    main_flow/.style={->, thick, red!80, -{Stealth[scale=1.2]}},
    data_flow/.style={->, dashed, blue!80, -{Stealth}},
    dep_arrow/.style={->, dotted, green!60!black, -{Stealth}}, 
    internal_flow/.style={->, -{Stealth}}
]

\node[phase, highlighted, minimum width=4.5cm, text width=4.2cm] (main) at (0,13) {
  Hit Space Dimension\\Algorithm
};

\node[phase] (init) [below=1.5cm of main] {
  Phase 1: Initialization
};
\node[detail] (ring) [below=1cm of init, xshift=-4.5cm] {Ring $\mathcal{P}_k$};
\node[detail] (basis) [right=1.2cm of ring] {Basis $\mathcal{B}_d$};
\node[detail] (memo) [right=1.2cm of basis] {Memoized\\Sq Hom.};
\node[detail] (iso) [right=1.2cm of memo] {Isomorphism $\phi$};

\node[phase] (gen) [below=2.8cm of init] {
  Phase 2: Adaptive Generation
};
\node[detail] (tasks) [below=0.8cm of gen] {
  Gather All\\Generator Tasks
};
\node[decision] (decision) [below=1cm of tasks] {
  Tasks $<$\\THRESHOLD?
};
\node[path_box] (serial) [below left=0.8cm and 1cm of decision] {
  Serial Mode:\\Single-core loop
};
\node[path_box] (parallel) [below right=0.8cm and 1cm of decision] {
  Parallel Mode:\\Distribute to\\all CPU cores
};
\node[detail, highlighted] (hit_coll) [below=2.2cm of decision] {
  Hit Coordinate\\Collection $\mathcal{V}_H$
};

\node[phase, highlighted] (matrix) [below=1.5cm of hit_coll] {
  Phase 3: Sparse Matrix Const.
};

\node[phase, highlighted] (rank) [below=1.5cm of matrix] {
  Phase 4: Rank Computation
};
\node[detail] (iterative) [left=2.5cm of rank] {
  Iterative Solvers\\(Wiedemann)
};
\node[detail] (rank_calc) [right=2.5cm of rank] {
  $\text{rank}(M)$
};

\node[phase] (final) [below=1.8cm of rank] {
  Phase 5: Final Result
};
\node[result] (dim_hit) [below=1cm of final] {
  $\dim(\mathcal{A}^+\mathcal{P}_k)_d = \text{rank}(M)$
};

\node[draw=gray!60, fill=gray!5, text width=3cm, align=left, font=\sffamily\tiny,
      rounded corners=3pt, right=1.2cm of gen, yshift=-0.5cm] {
    \textbf{Legend:}\\[0.15cm]
    \textcolor{red!80}{\rule{0.8cm}{1.2pt}} Main Flow\\[0.1cm]
    \textcolor{blue!80}{\rule[0.2ex]{0.8cm}{0.4pt}} Data Flow (Dashed)\\[0.1cm]
    \textcolor{green!60!black}{\rule[0.2ex]{0.8cm}{0.4pt}} Dependencies (Dotted)\\[0.1cm]
    \textcolor{black}{\rule{0.8cm}{0.4pt}} Internal Steps
};


\draw[main_flow] (main) -- (init);
\draw[main_flow] (init) -- (gen);
\draw[main_flow] (gen) -- (decision);
\draw[main_flow] (hit_coll) -- (matrix);
\draw[main_flow] (matrix) -- (rank);
\draw[main_flow] (rank) -- (final);
\draw[main_flow] (final) -- (dim_hit);

\draw[internal_flow] (init.south) to[out=-135, in=90] (ring.north);
\draw[internal_flow] (init.south) to[out=-90, in=90] (basis.north);
\draw[internal_flow] (init.south) to[out=-45, in=90] (memo.north);
\draw[internal_flow] (init.south) to[out=-20, in=90] (iso.north);

\draw[dep_arrow] (ring.south) to[out=-45, in=135, looseness=0.8] (tasks.north);
\draw[dep_arrow] (basis.south) to[out=-45, in=100, looseness=0.8] (tasks.north);
\draw[dep_arrow] (memo.south) to[out=-90, in=45] (tasks.north);
\draw[internal_flow] (gen.south) to[bend right=20] (tasks.north);
\draw[internal_flow] (tasks.south) -- (decision.north);
\draw[internal_flow] (decision.west) -- node[above, font=\tiny] {Yes} (serial.north);
\draw[internal_flow] (decision.east) -- node[above, font=\tiny] {No} (parallel.north);

\draw[data_flow] (serial.south) to[out=-90, in=135] (hit_coll.north);
\draw[data_flow] (parallel.south) to[out=-90, in=45] (hit_coll.north);

\draw[data_flow] (matrix.west) to[out=180, in=90] (iterative.north);
\draw[internal_flow] (rank) -- (iterative);
\draw[internal_flow] (rank) -- (rank_calc);
\draw[data_flow] (iterative.south) to[out=-90, in=180, looseness=0.8] (rank_calc.west);
\draw[data_flow] (rank_calc.south) to[out=-90, in=90] (dim_hit.north);

\end{tikzpicture}
}
\end{center}

\medskip

\noindent\rule{\textwidth}{0.4pt}
\begin{center}
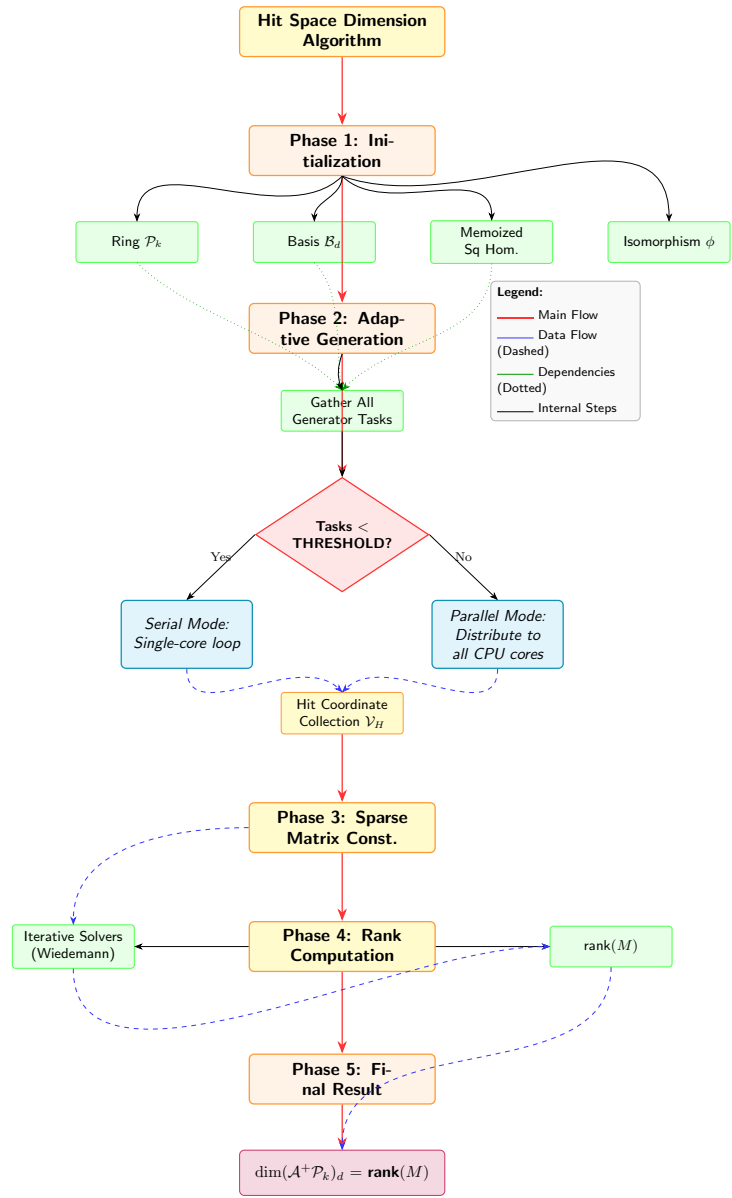
\textbf{Algorithm 2: Computing the dimension of the hit space (Hybrid Approach)}
\end{center}
\noindent\rule{\textwidth}{0.4pt}
\noindent\textbf{Input:} Degree $d$, number of variables $k$.\\
\textbf{Output:} The dimension of the hit space, $\dim((\mathcal{A}^+\mathcal{P}_k)_d)$.
\vspace{0.5em}

\noindent\textit{$\blacktriangleright$ Phase 1: Initialization}
\begin{enumerate}
    \item Define the polynomial ring $\mathcal{P}_k = \mathbb{F}_2[x_1, \dots, x_k]$.
    \item Generate the monomial basis $\mathcal{B}_d$ for $\mathcal{P}_k^d$. Let $N = |\mathcal{B}_d|$.
    \item Create the isomorphism $\phi: \mathcal{P}_k^d \to \mathbb{F}_2^N$.
\end{enumerate}

\noindent\textit{$\blacktriangleright$ Phase 2: Adaptive Generator Construction}
\begin{enumerate}
    \setcounter{enumi}{3}
    \item Initialize an empty list for tasks, $T$.
    \item Initialize an empty collection of hit vector data, $\mathcal{V}_H$.
    \item \textit{-- Gather all potential generator tasks --}
    \item \textbf{for all} Steenrod operators $Sq^{2^j}$ where $d-2^j \ge 0$ \textbf{do}
    \begin{enumerate}
        \item Let $d' = d-2^j$.
        \item \textbf{if} $\alpha(d'+k) \le k$ \textbf{then} \textit{$\blacktriangleright$ Apply theoretical filter}
        \begin{enumerate}
            \item \textbf{for all} monomials $g \in \mathcal{P}_k^{d'}$ \textbf{do}
            \item[] \quad Append task $(Sq^{2^j}, g)$ to $T$.
            \item \textbf{end for}
        \end{enumerate}
        \item \textbf{end if}
    \end{enumerate}
    \item \textbf{end for}
    \item \textit{-- Decide execution strategy based on the number of tasks --}
    \item \textbf{if} size of $T < \texttt{THRESHOLD}$ \textbf{then} \textit{$\blacktriangleright$ Execute in Serial Mode}
    \begin{enumerate}
        \item \textbf{for each} task $(Sq^k, g)$ in $T$ \textbf{do}
        \item[] \quad $h \gets Sq^k(g)$.
        \item[] \quad \textbf{if} $h \neq 0$ \textbf{then}
        \item[] \qquad Compute coordinate vector data $[h]_\text{data} = \phi(h)$.
        \item[] \qquad Add $[h]_\text{data}$ to the collection $\mathcal{V}_H$.
        \item[] \quad \textbf{end if}
        \item \textbf{end for}
    \end{enumerate}
    \item \textbf{else} \textit{$\blacktriangleright$ Execute in Parallel Mode}
    \begin{enumerate}
        \item Initialize a pool of parallel workers.
        \item Distribute tasks in $T$ among workers. Each worker computes $h=Sq^k(g)$ and returns its coordinate vector data $[h]_\text{data}$.
        \item $\mathcal{V}_H \gets$ Collect all non-null vector data from workers.
    \end{enumerate}
    \item \textbf{end if}
\end{enumerate}

\noindent\textit{$\blacktriangleright$ Phase 3: Sparse Matrix Construction and Rank Computation}
\begin{enumerate}
    \setcounter{enumi}{10}
    \item Construct a sparse matrix $M$ whose columns are formed from the vector data in $\mathcal{V}_H$.
    \item Compute the rank of $M$ using an efficient algorithm (e.g., Wiedemann).
    \item \textbf{return} $\text{rank}(M)$.
\end{enumerate}

\medskip

\noindent\rule{\textwidth}{0.4pt}
\begin{center}
\textbf{\textbf{Detailed SageMath code for the algorithm}}
\end{center}
\noindent\rule{\textwidth}{0.4pt}

\medskip

\begin{lstlisting}[language=Python, caption={\textbf{\textsc{SageMath} implementation of the hit space dimension}}, label={lst:SAGEMATH_code_2}]
from sage.all import *
from multiprocessing import Pool
import os
import time

def alpha(n):
    """Computes the sum of the bits in the binary expansion of an integer n."""
    return sum(int(bit) for bit in bin(n)[2:])

def get_sq_function(P, max_k):
    """Generate the Sq^k function using power series substitution."""
    sq_homs = {}
    sq_degrees_needed = [2**i for i in range((max_k + 1).nbits()) if 2**i <= max_k and 2**i > 0]
    for j in sq_degrees_needed:
        try:
            PS = PowerSeriesRing(P, 's', j + 1)
            s = PS.gen()
            sq_homs[j] = P.hom([x + s*x*x for x in P.gens()], codomain=PS)
        except Exception as e:
            pass
    def sq_final(k, f):
        if k == 0:
            return f
        if k not in sq_homs or f.is_zero() or k > f.degree():
            return P(0)
        coeffs = sq_homs[k](f).coefficients()
        return P(0) if k >= len(coeffs) else coeffs[k]
    return sq_final

# ---- Worker Functions ----
worker_P = None
worker_mono_map = None
worker_sq = None

def init_worker(ring, a_map, max_k):
    """Initializer function for each worker process."""
    global worker_P, worker_mono_map, worker_sq
    worker_P = ring
    worker_mono_map = a_map
    worker_sq = get_sq_function(worker_P, max_k)

def worker_task(task):
    """
    Performs one unit of work: computes Sq^k(g) and returns a list of 
    row indices for the corresponding matrix column.
    """
    k, g_exponents = task
    g = worker_P.monomial(*g_exponents)
    Sqg = worker_sq(k, g)
    if not Sqg.is_zero():
        row_indices = []
        for mt, c in Sqg.dict().items():
            m = worker_P.monomial(*mt)
            if m in worker_mono_map:
                row_indices.append(worker_mono_map[m])
        if row_indices:
            return row_indices
    return None



def calculate_hit_dimension_hybrid(degree, num_vars, PARALLEL_TASK_THRESHOLD=5000):
    """
    Calculates the dimension of the hit space using a hybrid (serial/parallel) method.
    """
    print(f"Starting calculation for degree {degree}, {num_vars} variables...")
    P = PolynomialRing(GF(2), num_vars, [f'x{i+1}' for i in range(num_vars)])
    monomials_d = list(P.monomials_of_degree(degree))
    if not monomials_d:
        print("No monomials of this degree.")
        return 0
    mono_map = {m: i for i, m in enumerate(monomials_d)}

    # --- Phase 1: Gather all tasks ---
    tasks = []
    sq_degrees = [2**i for i in range(degree.nbits()) if 2**i > 0]
    print("Preparing computation tasks...")
    for k in sq_degrees:
        deg_g = degree - k
        if deg_g < 0:
            continue
        if alpha(deg_g+ num_vars) > num_vars:
            continue
        for g in P.monomials_of_degree(deg_g):
            tasks.append((k, g.exponents()))
    total_tasks = len(tasks)
    print(f"Total tasks to process: {total_tasks}")

    # --- Phase 2: Select execution mode and compute ---
    valid_results = []
    if total_tasks < PARALLEL_TASK_THRESHOLD:
        print(f"Task count ({total_tasks}) is below threshold ({PARALLEL_TASK_THRESHOLD}). Running in SERIAL mode.")
        sq = get_sq_function(P, degree)
        for k, g_exponents in tasks:
            g = P.monomial(*g_exponents)
            Sqg = sq(k, g)
            if not Sqg.is_zero():
                row_indices = []
                for mt, c in Sqg.dict().items():
                    m = P.monomial(*mt)
                    if m in mono_map:
                        row_indices.append(mono_map[m])
                if row_indices:
                    valid_results.append(row_indices)
    else:
        num_workers = os.cpu_count()
        print(f"Task count ({total_tasks}) exceeds threshold. Running in PARALLEL mode with {num_workers} processes.")
        with Pool(initializer=init_worker, initargs=(P, mono_map, degree)) as pool:
            results = pool.map(worker_task, tasks)
        valid_results = [res for res in results if res is not None]

    if not valid_results:
        print("Hit space is zero-dimensional.")
        return 0

    # --- Phase 3: Construct sparse matrix and compute rank ---
    print("Constructing sparse matrix from results...")
    sparse_matrix = {}
    num_cols = len(valid_results)
    for col_index, row_indices in enumerate(valid_results):
        for row_idx in row_indices:
            sparse_matrix[(row_idx, col_index)] = 1
    num_rows = len(monomials_d)
    print(f"Constructed sparse matrix with {len(sparse_matrix)} non-zero entries (size {num_rows} x {num_cols})...")
    M = matrix(GF(2), num_rows, num_cols, sparse_matrix, sparse=True)
    print("Computing rank of the matrix...")
    rank = M.rank()
    print(f"Rank computation complete. Rank = {rank}")
    return rank
\end{lstlisting}

\subsection*{Algorithmic description of the sparse matrix technique}

The algorithm in LISTING \ref{lst:SAGEMATH_code_2} provides an optimized method for computing the dimension of the hit space by implementing Theorem \ref{dlc2}. Its efficiency for high degrees hinges on a \textit{dual strategy}: a sparse matrix representation to manage memory, combined with an \textit{adaptive hybrid computation model} to leverage modern multi-core processors. The process, grounded in the principles of linear algebra over $\F_2$, can be broken down into the following key steps:

\begin{itemize}

    \item \textbf{Motivation for sparsity and initial setup}: The number of rows and columns of the matrix $M$ grows polynomially with the degree $d$ and the number of variables $k$, quickly becoming too large to store in memory as a dense matrix. To circumvent this, the algorithm begins by establishing the computational context: defining the polynomial ring $\mathcal{P}_k = \F_2[x_1, \dots, x_k]$ and generating a basis of all monomials of the target degree $d$. A crucial step is creating a mapping, \texttt{mono\_map}, which assigns a unique integer index (a row index) to each monomial in this basis. This map serves as the bridge between abstract polynomials and concrete matrix row indices.

    \item \textbf{Adaptive generator computation and coordinate collection}: This is the core of the performance optimization. Instead of a single fixed computational path, the algorithm first compiles a complete list of all potential generator \textit{tasks}. This list is built using a highly effective \textit{two-level filtering strategy}:
    \begin{itemize}
        \item \textit{Operator Filtering:} First, the algorithm only considers Steenrod operators whose degrees are powers of two, $Sq^{2^j}$.
        \item \textit{Generator Degree Filtering:} Second, before considering the monomials $g$ of a potential degree $d' = d - 2^j$, the algorithm applies the theoretical filter checking if the condition $\alpha(d'+k) \le k$ is met. By skipping degrees $d'$ that do not satisfy this, a vast number of redundant computations are avoided.
    \end{itemize}
    With the full list of tasks prepared, the algorithm employs its core \textit{adaptive strategy}: it checks if the total number of tasks is below a certain \texttt{THRESHOLD}.
    \begin{itemize}
        \item If the problem is small, it runs in a fast \textit{serial mode}, executing each task sequentially. This avoids the overhead cost of parallelization where it is not beneficial.
        \item If the problem is large, it automatically switches to a \textit{parallel mode}, distributing the entire workload across all available CPU cores. Each worker process computes its assigned tasks independently, and the results are collected upon completion.
    \end{itemize}
    The outcome of either mode is a collection of coordinate data representing all non-zero hit vectors.

    \item \textbf{Sparse matrix representation}: After the collection of hit vector data is computed (either serially or in parallel), a dictionary, named \texttt{sparse\_matrix} in the code, is used to store the matrix in a sparse format. For each computed hit vector (which corresponds to a column, \texttt{col\_index}), the algorithm adds entries of the form \texttt{(row\_index, col\_index): 1} to the dictionary for each of its non-zero coordinates. This method ensures that only the locations of the '1's in the matrix are stored, dramatically reducing memory consumption.

    \item \textbf{Final construction and rank computation}: After populating the dictionary, it is passed to \textsc{SageMath}'s matrix constructor with the \texttt{sparse=True} flag to create an optimized sparse matrix object. Finally, the algorithm calls the \texttt{.rank()} method on this object. This function is highly optimized for large, sparse matrices over finite fields and efficiently computes the rank. By Theorem \ref{dlc2}, this rank is precisely the dimension of the hit space, $\dim((\mathcal{A}^+ \mathcal{P}_k)_d)$.
    
\end{itemize}

\medskip

\begin{example}\label{kcvd}

\begin{itemize}

\item[(i)] We consider the polynomial algebra $\mathcal P_5$ and the generic degree $d_s := 2^{s+3}  +2^{s+1} - 5$ for $s > 0.$ Recall that the arithmetic function \(\mu: \mathbb{N}^* \longrightarrow \mathbb{N}^*\) is defined by
\[
\mu(n) = \min\left\{ \ell : n = \sum_{j=1}^{\ell} (2^{u_j} - 1),\, u_j > 0, \, 1\leq j\leq \ell \right\} = {\rm min}\bigg\{\ell:\ \alpha(n + \ell)\leq \ell\bigg\},
\]
for all \(n \in \mathbb{N}^*\).
According to \cite{Kameko}, $(Q\mathcal P_k)_{d} \cong (Q\mathcal P_k)_{\frac{d-k}{2}}$ is an isomorphism of $\mathbb F_2$-vector spaces if and only if $\mu(d) = k.$ Hence, we can see that $\mu(d_s) = 5$ for any $s > 2.$ Hence, $(Q\mathcal P_5)_{d_s} \cong (Q\mathcal P_5)_{d_{2} = 35}$ for all $s\geq 2.$  Thus, we only need compute the dimension of $(Q\mathcal P_5)_{d_s}$ for $1\leq s\leq 2.$

\medskip

$\bullet$ For \(s = 1\), we have \(d_1 = 15\). To modify the above algorithm, we insert the following line of code immediately after ``\texttt{return M.rank()}'' in the final line:
  
\medskip

\begin{lstlisting}[language=Python]
# --- Main execution block ---
if __name__ == "__main__":
    # Change parameters k and d here for testing
    k = 5 
    d = 15
    
    start = time.time()
    rank_M = calculate_hit_dimension_hybrid(degree=d, num_vars=k)
    end = time.time()
    
    print("\n== CALCULATION COMPLETE ==")
    print(f"Case: k = {k}, d = {d}")
    print(f"Hit space dimension dim(A^+P_k) = rank(M) = {rank_M}")
    print(f"Execution time: {end - start:.2f} seconds")
\end{lstlisting}

\medskip

We will obtain the output of the algorithm as follows:
\[
\dim ( (\mathcal{A}^+ \mathcal{P}_5)_{d_1})  = \operatorname{rank}(M) = 3444.
\]
Consequently, by applying Corollary~\ref{hq2}, we obtain 
\[
\dim (Q\mathcal{P}_5)_{d_1} = \binom{39}{4} - 3444 = 432.
\]
This result confirms the earlier computation, carried out entirely by hand by Sum in~\cite{Sum1}; however, no algorithm was provided to verify the manual calculation.

\medskip

$\bullet$ For \(s = 2\), we have \(d_2 = 35\). By proceeding in a similar manner using the algorithm above, we obtain the following output:
\[
\dim ( (\mathcal{A}^+ \mathcal{P}_5)_{d_2})  = \operatorname{rank}(M) = 81134.
\]
So, by Corollary~\ref{hq2}, we get
\[
\dim (Q\mathcal{P}_5)_{d_2} = \binom{39}{4} - 81134 = 1117.
\]
This result confirms the earlier computation by N.H.V. Hung in~\cite{Hung}; however, he merely stated the dimension of  $(Q\mathcal{P}_5)_{d_2}$ without providing any algorithmic description or detailed justification. 

\medskip

Thus, computing \(\dim(Q\mathcal{P}_5)_{d_s}\) for general \(d_s\) using our algorithm in conjunction with Kameko's result~\cite{Kameko} proves to be both efficient and reliable.

\medskip

\item[(ii)] We now consider another illustration of the algorithm, this time applied to the polynomial algebra \(\mathcal{P}_5\) in a general degree \(d_s := 2^{s+2} - 4\), where \(s\) is any positive integer.  

\medskip

$\bullet$ As a consequence of prior theoretical results, it was shown in \cite{PS2} that
\[
\dim( Q\mathcal{P}_5)_{d_s} = (2^5 - 1) \cdot 21 = 651, \quad \text{for any } s \geq 5.
\]
Consequently, we only need to determine the dimension of \((Q\mathcal{P}_5)_{d_s}\) for \(1 \leq s \leq 4\), which can be effectively computed using our algorithm and Corollary~\ref{hq2}.

$\bullet$ For \(1 \leq s \leq 4\), the algorithm yields the following results:

\begin{table}[ht]
\centering
\begin{tabular}{|c|c|c|}
\hline
$s$ & $d_s = 2^{s+2} - 4$ & $\dim\left((\mathcal{A}^+ \mathcal{P}_5)_{d_s}\right) = \operatorname{rank}(M)$ \\
\hline
1 & 4  & $25$     \\
2 & 12 & $1630$   \\
3 & 28 & $35480$  \\
4 & 62 & $720070$ \\
\hline
\end{tabular}
\caption{Values of \(\dim((\mathcal{A}^+ \mathcal{P}_5)_{d_s})\) for \(1 \leq s \leq 4\)}
\label{tab:hit-dimension-space}
\end{table}

\medskip

\textbf{Note.} The reader can easily verify the algorithm's output for the cases \(s = 1\) and \(s = 2\) using the \textsc{SageMath} online interface available at \url{https://sagecell.sagemath.org/}. However, for \(3 \leq s \leq 4\), a personal computer with sufficiently large memory is required to run the algorithm, due to the significant computational complexity as the degree increases.

\medskip

Consequently, applying Corollary~\ref{hq2}, we get

\begin{table}[ht]
\centering
\begin{tabular}{|c|c|c|}
\hline
$s$ & $d_s = 2^{s+2} - 4$ & $\dim(Q\mathcal{P}_5)_{d_s} = \dim\left( \mathcal{P}_5^{d_s} \big/ (\mathcal{A}^+ \mathcal{P}_5)_{d_s} \right) = \binom{d_s+4}{4} - \operatorname{rank}(M)$ \\
\hline
1 & 4  & $\binom{d_1+4}{4} - 25 = 45$     \\
2 & 12 & $\binom{d_2+4}{4} - 1630 = 190$    \\
3 & 28 & $\binom{d_3+4}{4} - 35480= 480$    \\
4 & 62 & $\binom{d_4+4}{4} - 720070 = 650$    \\
\hline
\end{tabular}
\caption{Values of $\dim((Q\mathcal{P}_5)_{d_s})$ for $1 \leq s \leq 4$}
\label{tab:quotient-dimension}
\end{table}

The above results are consistent with our previous findings in~\cite{PS, PS2}, where the dimension of \((Q\mathcal{P}_5)_{d_s} \) was computed entirely by hand for $d = 2^{s+2} - 4,\, 1\leq s\leq 4.$ 

\end{itemize}

\noindent

\end{example}

\begin{example}\label{kcvd2}

We now consider the degree \(d_s := 2^s\) for \(1 \leq s \leq 7\) and the polynomial algebra \(\mathcal{P}_5\).  By applying our algorithm, we obtain the following results:

\begin{table}[ht]
\centering
\begin{tabular}{|c|c|c|}
\hline
$s$ & $d_s = 2^{s}$ & $\dim\left((\mathcal{A}^+ \mathcal{P}_5)_{d_s}\right)= \operatorname{rank}(M)$ \\
\hline
1 & 2  & $5$     \\
2 & 4 & $25$   \\
3 & 8 & $321$  \\
4 & 16 & $4402$ \\
5 & 32 & $57901$ \\
6 & 64 & $812695$ \\
7 & 128 & $12080800$ \\
\hline
\end{tabular}
\caption{Values of \(\dim((\mathcal{A}^+ \mathcal{P}_5)_{d_s})\) for \(1 \leq s \leq 7\)}
\label{tab:hit-dimension}
\end{table}

\medskip

\newpage
Consequently, applying Corollary~\ref{hq2}, we obtain

\begin{table}[ht]
\centering
\begin{tabular}{|c|c|c|}
\hline
$s$ & $d_s = 2^{s}$ & $\dim(Q\mathcal{P}_5)_{d_s} = \dim\left( \mathcal{P}_5^{d_s} \big/ (\mathcal{A}^+ \mathcal{P}_5)_{d_s} \right) = \binom{d_s+4}{4} - \operatorname{rank}(M)$ \\
\hline
1 & 2  & $\binom{d_1+4}{4} - 5 = 10$     \\
2 & 4 & $\binom{d_2+4}{4} - 25 = 45$    \\
3 & 8 & $\binom{d_3+4}{4} - 35480= 174$    \\
4 & 16 & $\binom{d_4+4}{4} - 4402 = 443$    \\
5 & 32 & $\binom{d_5+4}{4} - 57901 = 1004$    \\
6 & 64 & $\binom{d_6+4}{4} - 812695 = 1690$    \\
7 & 128 & $\binom{d_7+4}{4} - 12080800 = 1985$    \\
\hline
\end{tabular}
\caption{Values of $\dim((Q\mathcal{P}_5)_{d_s})$ for $1 \leq s \leq 7$}
\label{tab:quotient-dimension}
\end{table}

\medskip

The values of $\dim\bigl((Q\mathcal{P}_5)_{2^{s}}\bigr)$ for $1 \le s \le 4$ obtained above coincide with the earlier hand calculations in~\cite{Tin}. The $s = 5$ case also coincides with the hand calculations in our previous work \cite{Phuc}.
\medskip

\noindent

\begin{remark}\label{nxc}

For the case \(s = 6\), our algorithm shows that the fully manual computation mentioned in~\cite{Sum4} is \emph{not} correct: Sum and Tai claim  
\[
\dim((Q\mathcal{P}_5)_{2^{6}}) = 1694,
\]  
whereas our algorithmic result (see TABLE~\ref{tab:quotient-dimension}) is  
\[
\dim((Q\mathcal{P}_5)_{2^{6}}) =\binom{68}{4} - 812695 = 1690.
\]
This discrepancy highlights the essential role of explicit algorithms in validating hand computations, which are prone to errors.

\medskip

For \(s = 7\), Sum and Tai have provided the estimate \(1984 \leq \dim((Q\mathcal{P}_5)_{2^7}) \leq 1990\) in~\cite{Sum4}.  
As shown in TABLE~\ref{tab:quotient-dimension}, our algorithm yields the exact value
\[
\dim((Q\mathcal{P}_5)_{2^7}) = \binom{132}{4} - 12080800 = 1985,
\]
which lies within their estimated range. Importantly, the dimension obtained via our algorithm is both explicit and verifiable, surpassing the limitations of manual computation encountered in~\cite{Sum4}.
\end{remark}
\end{example}

\subsection{Analysis of computational complexity and practical limitations}\label{s33}

\emph{}

\medskip

While the algorithms presented in Theorems \ref{dlc1} and \ref{dlc2} provide a systematic method for solving the hit problem, their practical application is bounded by computational complexity. This analysis clarifies the sources of this inherent complexity, thereby motivating the optimized techniques used in LISTING \ref{dlc2}---\textit{specifically, the adaptive parallel model, the alpha-function filter, and the sparse matrix approach}---and examines the remaining computational costs.

\paragraph{Conceptual size of the matrix $M$}
The matrix $M$ remains the central computational object. Its conceptual dimensions dictate the theoretical scale of the problem and are the primary reason why optimized techniques are essential.

\begin{itemize}
    \item \textbf{Number of rows:} The number of rows of $M$ corresponds to the dimension of the vector space $\mathcal{P}_{k}^{d}$. This dimension is given by the ``stars'' and ``bars'' combinatorial formula:
    $$
        N_{\text{rows}} = \dim(\mathcal{P}_{k}^{d}) = \binom{d+k-1}{k-1}.
    $$
    This number grows polynomially with $d$ and $k$, defining the height of our problem space.

    \item \textbf{Number of columns:} The number of columns of $M$ is the size of the generating set $\mathcal{H}_{d}(\mathcal{P}_{k})$. With the alpha-function filter optimization, this set is significantly smaller than the full theoretical set. The algorithm only considers actions of $Sq^{2^j}$ on monomials $g$ of degree $d' = d-2^j$ that satisfy the condition $\alpha(d'+k) \le k$. The total number of columns is therefore the sum over only these \textit{admissible} degrees:
    $$
        N_{\text{cols}} = |\mathcal{H}'_{d}(\mathcal{P}_{k})| = \sum_{\substack{j=0 \\ \text{s.t. } \alpha(d-2^j + k) \le k}}^{\lfloor\log_2 d\rfloor} \dim(\mathcal{P}_{k}^{d-2^j}) = \sum_{\substack{j=0 \\ \text{s.t. } \alpha(d-2^j+k) \le k}}^{\lfloor\log_2 d\rfloor} \binom{d-2^j+k-1}{k-1}.
    $$
    This filtering drastically reduces the number of columns compared to a naive implementation, which is a key to the algorithm's feasibility.
\end{itemize}

\paragraph{Computational stages in a sparse implementation}
The use of a sparse matrix representation, combined with the adaptive parallel model, fundamentally changes how computational resources are consumed.

\begin{itemize}
    \item \textbf{Hybrid generator computation:} This phase sees the most significant performance gain from parallelization. Instead of a single-threaded process, the algorithm first gathers a list of all potential generator ``tasks''. For large-scale problems, this list of tasks is then distributed among all available CPU cores. If a machine has $N_{\text{cores}}$ processors, the wall-clock time for this stage is ideally reduced by a factor of up to $N_{\text{cores}}$. \textit{This effectively removes what was the primary time bottleneck in the serial version of the algorithm.} For small problems, the algorithm intelligently defaults to a fast serial loop to avoid the overhead of creating parallel processes.

    \item \textbf{Sparse matrix storage:} The critical optimization is in memory usage. Instead of allocating a dense array of size $N_{\text{rows}} \times N_{\text{cols}}$, the algorithm only stores the non-zero entries. The memory requirement is proportional to the number of non-zero entries (NNZ), i.e., $O(\text{NNZ})$. The alpha filter helps here as well, by reducing both $N_{\text{cols}}$ and, consequently, the total NNZ. For instance, in the verified case of $k=5$ and $d=35$, while $N_{\text{rows}}$ is 82251, the matrix is extremely sparse. However, the total number of non-zero entries can still be in the millions, thus the computations for $3\leq s \leq 4$ (corresponding to $28\leq d \leq 62$) in TABLE \ref{tab:hit-dimension-space} still required a computer with substantial memory.

    \item \textbf{Sparse rank computation:} With the generator computation phase now heavily accelerated, this stage often becomes the new primary bottleneck in terms of time. The algorithm calls SageMath's highly optimized \texttt{.rank()} method, which uses iterative algorithms like the Wiedemann or Lanczos methods for sparse matrices. The time complexity of these methods, often depending on factors like $O(N_{\text{rows}} \cdot \text{NNZ})$, means that even with all prior optimizations, this step's high-degree polynomial growth remains the main limiting factor for tackling even larger $d$ and $k$.
\end{itemize}

\section{Conclusions}

In this paper, we established a general, matrix-based criterion for determining whether a given homogeneous polynomial is ``hit'' within the context of the Peterson hit problem. Our central contribution is the reduction of this abstract topological question into a concrete problem of linear algebra: a polynomial $f$ is hit if and only if its coordinate vector $[f]$ lies in the column space of a matrix $M$ constructed from Steenrod operations. This translation to the solvability of the linear system $M\mathbf{c} = [f]$ over $\F_2$ creates a direct and robust bridge between pure theory and practical computation.

To leverage this foundation, we developed and implemented two novel \textsc{SageMath} algorithms, \textit{engineered with an adaptive parallel architecture}, designed for systematic and reproducible investigation far beyond the scope of previous manual calculations. These tools provide:
\newpage
\begin{itemize}
    \item An algorithm that, for any polynomial $f \in \mathcal{P}_{k}^{d}$, definitively decides if $f$ is hit and, if so, returns an explicit decomposition $f = \sum_{i>0}\mathrm{Sq}^{i}(g_i)$.
    \item An algorithm that \textit{rapidly} computes the exact dimension of the hit space $(\mathcal{A}^{+}\mathcal{P}_{k})_{d}$ and, consequently, the dimension of the quotient space $Q\mathcal{P}_k$ for any number of variables~$k$ and degree~$d$.
\end{itemize}

The immediate utility of this computational framework is demonstrated by its application to the challenging five-variable case, where our algorithms successfully:
\begin{enumerate}
    \item Verified and solidified several key results on $\dim(Q\mathcal{P}_5)_{d}$ that were previously supported only by extensive manual arguments.
    \item Uncovered and corrected a miscalculation in the recent work of Sum and Tai \cite{Sum4}, revising the dimension of $(Q\mathcal{P}_5)_{2^6}$ from 1694 to the correct value of 1690.
    \item Confirmed an estimate by Sum and Tai \cite{Sum4} for degree $d=2^7$ by computing the exact value, $\dim(Q\mathcal{P}_5)_{2^7} = 1985$, which falls within their predicted range.
\end{enumerate}

Compared with earlier approaches---which were often restricted to a small number of variables, confined to special degrees, and rarely accompanied by verifiable code---our framework is notably \emph{scalable}, \emph{transparent}, and \emph{reproducible}. The primary computational bottleneck, as analyzed in Section \ref{s33}, has shifted due to our parallel optimizations. While the initial generator computation is now highly accelerated, the primary remaining bottleneck is often the final sparse rank computation, which remains a serial process in most standard computer algebra systems. This observation naturally points toward future work focused on further optimizations, such as minimizing the generating set for the hit space or deploying specialized parallel libraries for linear algebra over finite fields.

Looking forward, the techniques developed here open the door to a new phase of algorithmic exploration of the hit problem for larger~$k$ and higher degrees. Beyond this specific problem, our methods may provide valuable tools for studying related structures, such as the Singer algebraic transfer (see, \cite{CH}, \cite{Hung}, \cite{HP}, \cite{Phuc0, Phuc, Phuc2}, \cite{Sum1}). We are confident that the combination of a clear linear-algebraic criterion and freely available \textsc{SageMath} code will serve as a solid foundation for significant subsequent advances in computational algebraic topology.

\section{Planned work for the next stage (Part II)}

In the second phase of our research project, we aim to \textit{explicitly determine an admissible basis} for the quotient space $(Q\mathcal{P}_k)_d$ for arbitrary values of $k$ and positive degrees $d$, by implementing a \textit{constructive algorithm in \textsc{SageMath}} based on the theoretical foundations and procedures established in this paper. Naturally, we emphasize that the algorithm is designed to be applicable for any $k$ and $d$, provided they lie within the memory and computational limits of the system executing it.

Building upon known results regarding the hit problem for the cases $k \leq 4$, as well as for many instances with $k \geq 5$, it is generally observed that the dimension of $(Q\mathcal{P}_k)_d$ stabilizes at a certain degree $d$, which depends on specific parameters (as illustrated in Examples~\ref{kcvd} and~\ref{kcvd2}). This phenomenon is expected, since the hit problem can often be reduced to degrees $d$ satisfying the condition $\mu(d) < k$.

Therefore, developing an effective algorithm that not only computes but also provides \textit{explicit information about an admissible basis} of $(Q\mathcal{P}_k)_d$ for specific values of $k$ and $d$ is of significant importance. It allows us to \textit{predict the dimension} in more general settings. Moreover, current approaches mostly rely on manual computation, which is both error-prone and inefficient. By contrast, the algorithmic approach we propose enables two key advantages: it offers a means to \textit{verify existing hand-calculations}, and it \textit{reduces the computational burden} when generalizing to arbitrary degrees $d$ for fixed $k$.

\medskip

Building on the symbolic criteria and algorithms established here, we outline the next phase of this project, which focuses on constructing explicit admissible monomial bases for \( Q\mathcal{P}_k \) via weight decomposition.

\medskip

$\bullet$ Let the dyadic expansion of a positive integer $d$ be given by $d = \sum_{j\geq 0}\alpha_j(d)2^j$, where $\alpha_j(d)= 0,\, 1.$ For each monomial $x = x_1^{a_1}x_2^{a_2}\ldots x_k^{a_k}$ in $\mathcal{P}_k,$ we define its \textit{weight vector} $\omega(x)$ as the sequence: $\omega(x) =(\omega_1(x), \omega_2(x), \ldots, \omega_j(x), \ldots),$ where $\omega_j(x) = \sum_{1\leq i\leq k}\alpha_{j-1}(a_i)$ for all $j \geq 1$.

\medskip

$\bullet$ For a weight vector $\omega = (\omega_1, \omega_2, \ldots, \omega_s,0,0,\ldots),$ we define $\deg \omega =\sum_{s\geq 1}2^{s-1}\omega_s.$ Denote by $\mathcal{P}_k(\omega)$ the subspace of $\mathcal{P}_k$ spanned by all monomials $u\in \mathcal{P}_k$ such that $\deg u = \deg \omega,\ \omega(u)\leq \omega,$ and by $\mathcal{P}_k^-(\omega)$ the subspace of $\mathcal{P}_k(\omega)$ spanned by all monomials $u$ such that $\omega(u)< \omega.$

\medskip

$\bullet$ Let $f, g$ two homogeneous polynomials of the same degree in $\mathcal{P}_k.$ 
\begin{enumerate}
\item [(i)]$f \equiv g $ if and only if $(f  + g)\in \mathcal{A}^+\mathcal{P}_k.$ In particular, if $f\equiv 0,$ then $f\in \mathcal{A}^+\mathcal{P}_k.$
\item[(ii)] $f \equiv_{\omega} g$ if and only if $(f  + g)\in \mathcal{A}^+\mathcal{P}_k + \mathcal{P}_k^-(\omega).$ 
\end{enumerate}

Both the relations "$\equiv$" and "$\equiv_{\omega}$" possess the characteristics of equivalence relations, as is readily observable.  Let $Q\mathcal{P}_k(\omega)$ denote the quotient of $\mathcal{P}_k(\omega)$ by the equivalence relation ``$\equiv_\omega$''. Then, we have 
$$ \dim (Q\mathcal{P}_k)_d = \sum_{\deg(\omega) = d}\dim Q\mathcal{P}_k(\omega).$$

$\bullet$  A monomial $x\in \mathcal {P}_k$ is said to be \textit{inadmissible} if there exist monomials $y_1, y_2,\ldots, y_m$ in $\mathcal {P}_k$ such that $y_j < x$ for all $1\leq j\leq m$ and $x \equiv\sum_{1\leq j\leq m}y_j.$ Then, a monomial is called \textit{admissible} if it is not inadmissible. 

In $\mathcal{P}_k,$ a monomial $x$ is defined as \textit{strictly inadmissible} if and only if there exist monomials $y_1, y_2,\ldots, y_m$ of the same degree as $x$, each satisfying $y_j < x$ for $1\leq j\leq m$, such that $x$ can be expressed as $x = \sum_{1\leq j\leq m}y_t + \sum_{1\leq \ell < 2^r }Sq^{\ell}(h_j)$, where $r = \max\{i\in\mathbb Z: \omega_i(x) > 0\}$ and $h_j$ are appropriate polynomials in $\mathcal{P}_k.$

\begin{remark}
By establishing an algorithm in \textsc{SageMath}, we can explicitly determine whether a given monomial in $\mathcal{P}_k$ is strictly inadmissible. For example, consider the case $k = 5$ and $d = 21$, and the monomial 
\[ 
u = x_1 x_2^{3} x_3^{6} x_4^{6} x_5^{5}, 
\] 
for which we have $\omega(u) = (3,3,3)$. Our algorithm yields the following output:

\begin{align*}
u &= x_1 x_2^3 x_3^5 x_4^6 x_5^6 + x_1 x_2^3 x_3^6 x_4^5 x_5^6 \\
&\quad + Sq^1\bigg(x_1 x_2^6 x_3^3 x_4^5 x_5^5 + x_1 x_2^6 x_3^5 x_4^3 x_5^5 + x_1 x_2^6 x_3^5 x_4^5 x_5^3 \\
&\qquad + x_1^3 x_2^3 x_3^3 x_4^5 x_5^6 + x_1^3 x_2^3 x_3^3 x_4^6 x_5^5 + x_1^3 x_2^3 x_3^4 x_4^5 x_5^5 \\
&\qquad + x_1^3 x_2^3 x_3^5 x_4^3 x_5^6 + x_1^3 x_2^3 x_3^5 x_4^4 x_5^5 + x_1^3 x_2^3 x_3^5 x_4^5 x_5^4 \\
&\qquad + x_1^3 x_2^3 x_3^5 x_4^6 x_5^3 + x_1^3 x_2^3 x_3^6 x_4^3 x_5^5 + x_1^3 x_2^3 x_3^6 x_4^5 x_5^3 \\
&\qquad + x_1^3 x_2^4 x_3^3 x_4^5 x_5^5 + x_1^3 x_2^4 x_3^5 x_4^3 x_5^5 + x_1^3 x_2^4 x_3^5 x_4^5 x_5^3 \bigg) \\
&\quad + Sq^2\bigg(x_1 x_2^3 x_3^3 x_4^6 x_5^6 + x_1 x_2^3 x_3^5 x_4^5 x_5^5 + x_1 x_2^3 x_3^6 x_4^3 x_5^6 \\
&\qquad + x_1 x_2^3 x_3^6 x_4^6 x_5^3 + x_1^2 x_2^3 x_3^3 x_4^5 x_5^6 + x_1^2 x_2^3 x_3^3 x_4^6 x_5^5 \\
&\qquad + x_1^2 x_2^3 x_3^5 x_4^3 x_5^6 + x_1^2 x_2^3 x_3^5 x_4^6 x_5^3 + x_1^2 x_2^3 x_3^6 x_4^3 x_5^5 \\
&\qquad + x_1^2 x_2^3 x_3^6 x_4^5 x_5^3 \bigg) + Sq^8\bigg(x_1 x_2^3 x_3^3 x_4^3 x_5^3\bigg)  \pmod{\mathcal{P}_5^-(\omega(u))}.
\end{align*}

This equality implies that $u$ is strictly inadmissible. 

Note that the above example confirms the previous result obtained in our earlier work~\cite{Phuc0} through manual computation, although at that time we made a computational error in our attempt to prove its strict inadmissibility. In addition, the author of~\cite{Sum3} also referred to this monomial, but was \textbf{unable to prove} that it is strictly inadmissible. 

\medskip

For these reasons, the development of computer-implementable algorithms proves to be highly meaningful and effective in validating results that were previously obtained through manual computation.
\end{remark}
\medskip

Thus, \( (Q\mathcal{P}_k)_d \) has a basis consisting of all classes represented by admissible monomials of degree \( d \) in \( \mathcal{P}_k \).

\begin{theorem}[see \cite{Kameko}, \cite{Sum1}]\label{dlKS}
Let \( x, y, w \) be monomials in \( \mathcal{P}_k \) such that \( \omega_i(x) = 0 \) for all \( i > r > 0 \), \( \omega_s(w) \neq 0 \), and \( \omega_i(w) = 0 \) for all \( i > s > 0 \). Then the following statements hold:
\begin{enumerate}
    \item[(I)] If \( w \) is inadmissible, then the monomial \( x w^{2^r} \) is also inadmissible.
    \item[(II)] If \( w \) is strictly inadmissible, then the monomial \( x w^{2^r} y^{2^{r+s}} \) is inadmissible.
\end{enumerate}
\end{theorem}

$\bullet$ We consider the following $\mathcal{A}$-submodules of $\mathcal P_k:$
$$\begin{array}{ll}
\medskip
\mathcal {P}_k^0 &=\big\langle \big\{x_1^{a_1}x_2^{a_2}\ldots x_k^{a_k}\in \mathcal {P}_k\;|\;a_1a_2\ldots a_k = 0\big\}\big\rangle, \\
\mathcal {P}_k^+&= \big\langle \big\{x_1^{a_1}x_2^{a_2}\ldots x_k^{a_k}\in \mathcal {P}_k\;|\; a_1a_2\ldots a_k > 0\big\}\big\rangle.
\end{array}$$
Consequently, $$(Q\mathcal {P}_k)_d \cong (Q\mathcal {P}_k^0)_{d}\bigoplus (Q\mathcal {P}_k^+)_{d}.$$

\medskip

Based on the above representations and the algorithms developed in this paper, together with the use of the Gaussian elimination process applied to the augmented matrix \( [M \,|\, I] \), we will construct an algorithm in \textsc{SageMath} that enables the explicit determination of a basis for \( Q\mathcal{P}_k(\omega) \), and hence, the corresponding basis for \( (Q\mathcal{P}_k)_d \). For instance, when \( k = 5 \) and \( d = 14 \), our algorithm produces the following explicit output:

\medskip

\begin{lstlisting}
Sorting 3060 monomials in omega-sigma order...
Building hit space matrix M...
Reducing augmented matrix to row echelon form...
Extracting admissible basis...

============================================================================
                 COMPUTATION RESULTS
============================================================================
Case: k = 5, d = 14

Dimension Formula:
\begin{align*}
    \dim (Q\mathcal{P}_{5})_{14} &= \dim Q\mathcal{P}_{5}(\omega=(2, 2, 2)) + \dim Q\mathcal{P}_{5}(\omega=(2, 4, 1)) + \\ 
    &\quad \dim Q\mathcal{P}_{5}(\omega=(4, 3, 1))
\end{align*}

Detailed Dimensions per Weight Vector:
For \omega = (2, 2, 2):
  \dim [Q\mathcal{P}_{5}(\omega)]^0 = 115
  \dim [Q\mathcal{P}_{5}(\omega)]^+ = 15
For \omega = (2, 4, 1):
  \dim [Q\mathcal{P}_{5}(\omega)]^0 = 0
  \dim [Q\mathcal{P}_{5}(\omega)]^+ = 15
For \omega = (4, 3, 1):
  \dim [Q\mathcal{P}_{5}(\omega)]^0 = 75
  \dim [Q\mathcal{P}_{5}(\omega)]^+ = 100




----------------------------------------------------------------------------
          ADMISSIBLE BASIS BY WEIGHT VECTOR
----------------------------------------------------------------------------

* The space for weight vector \omega = (2, 2, 2) (Total: 130 basis elements):
  - Subgroup 1: Monomials with all positive exponents (15 elements)
    1. x1*x2*x3^2*x4^4*x5^6                 2. x1*x2*x3^2*x4^6*x5^4                 
    3. x1*x2*x3^6*x4^2*x5^4                 4. x1*x2^2*x3*x4^4*x5^6                 
    5. x1*x2^2*x3*x4^6*x5^4                 6. x1*x2^2*x3^3*x4^4*x5^4               
    7. x1*x2^2*x3^4*x4*x5^6                 8. x1*x2^2*x3^4*x4^3*x5^4               
    9. x1*x2^2*x3^5*x4^2*x5^4               10. x1*x2^3*x3^2*x4^4*x5^4              
    11. x1*x2^3*x3^4*x4^2*x5^4              12. x1*x2^6*x3*x4^2*x5^4                
    13. x1^3*x2*x3^2*x4^4*x5^4              14. x1^3*x2*x3^4*x4^2*x5^4              
    15. x1^3*x2^4*x3*x4^2*x5^4              
  - Subgroup 2: Monomials containing zero exponents (115 elements)
    16. x1*x2*x3^6*x4^6                     17. x1*x2*x3^6*x5^6                     
    18. x1*x2*x4^6*x5^6                     19. x1*x2^2*x3^4*x4^7                   
    20. x1*x2^2*x3^4*x5^7                   21. x1*x2^2*x3^5*x4^6                   
    22. x1*x2^2*x3^5*x5^6                   23. x1*x2^2*x3^7*x4^4                   
    24. x1*x2^2*x3^7*x5^4                   25. x1*x2^2*x4^4*x5^7                   
    26. x1*x2^2*x4^5*x5^6                   27. x1*x2^2*x4^7*x5^4                   
    28. x1*x2^3*x3^4*x4^6                   29. x1*x2^3*x3^4*x5^6                   
    30. x1*x2^3*x3^6*x4^4                   31. x1*x2^3*x3^6*x5^4                   
    32. x1*x2^3*x4^4*x5^6                   33. x1*x2^3*x4^6*x5^4                   
    34. x1*x2^6*x3*x4^6                     35. x1*x2^6*x3*x5^6                     
    36. x1*x2^6*x3^7                        37. x1*x2^6*x4*x5^6                     
    38. x1*x2^6*x4^7                        39. x1*x2^6*x5^7                        
    40. x1*x2^7*x3^2*x4^4                   41. x1*x2^7*x3^2*x5^4                   
    42. x1*x2^7*x3^6                        43. x1*x2^7*x4^2*x5^4                   
    44. x1*x2^7*x4^6                        45. x1*x2^7*x5^6                        
    46. x1*x3*x4^6*x5^6                     47. x1*x3^2*x4^4*x5^7                   
    48. x1*x3^2*x4^5*x5^6                   49. x1*x3^2*x4^7*x5^4                   
    50. x1*x3^3*x4^4*x5^6                   51. x1*x3^3*x4^6*x5^4                   
    52. x1*x3^6*x4*x5^6                     53. x1*x3^6*x4^7                        
    54. x1*x3^6*x5^7                        55. x1*x3^7*x4^2*x5^4                   
    56. x1*x3^7*x4^6                        57. x1*x3^7*x5^6                        
    58. x1*x4^6*x5^7                        59. x1*x4^7*x5^6                        
    60. x1^3*x2*x3^4*x4^6                   61. x1^3*x2*x3^4*x5^6                   
    62. x1^3*x2*x3^6*x4^4                   63. x1^3*x2*x3^6*x5^4                   
    64. x1^3*x2*x4^4*x5^6                   65. x1^3*x2*x4^6*x5^4                   
    66. x1^3*x2^3*x3^4*x4^4                 67. x1^3*x2^3*x3^4*x5^4                 
    68. x1^3*x2^3*x4^4*x5^4                 69. x1^3*x2^5*x3^2*x4^4                 
    70. x1^3*x2^5*x3^2*x5^4                 71. x1^3*x2^5*x3^6                      
    72. x1^3*x2^5*x4^2*x5^4                 73. x1^3*x2^5*x4^6                      
    74. x1^3*x2^5*x5^6                      75. x1^3*x3*x4^4*x5^6                   
    76. x1^3*x3*x4^6*x5^4                   77. x1^3*x3^3*x4^4*x5^4                 
    78. x1^3*x3^5*x4^2*x5^4                 79. x1^3*x3^5*x4^6                      
    80. x1^3*x3^5*x5^6                      81. x1^3*x4^5*x5^6                      
    82. x1^7*x2*x3^2*x4^4                   83. x1^7*x2*x3^2*x5^4                   
    84. x1^7*x2*x3^6                        85. x1^7*x2*x4^2*x5^4                   
    86. x1^7*x2*x4^6                        87. x1^7*x2*x5^6                        
    88. x1^7*x2^7                           89. x1^7*x3*x4^2*x5^4                   
    90. x1^7*x3*x4^6                        91. x1^7*x3*x5^6                        
    92. x1^7*x3^7                           93. x1^7*x4*x5^6                        
    94. x1^7*x4^7                           95. x1^7*x5^7                           
    96. x2*x3*x4^6*x5^6                     97. x2*x3^2*x4^4*x5^7                   
    98. x2*x3^2*x4^5*x5^6                   99. x2*x3^2*x4^7*x5^4                   
    100. x2*x3^3*x4^4*x5^6                  101. x2*x3^3*x4^6*x5^4                  
    102. x2*x3^6*x4*x5^6                    103. x2*x3^6*x4^7                       
    104. x2*x3^6*x5^7                       105. x2*x3^7*x4^2*x5^4                  
    106. x2*x3^7*x4^6                       107. x2*x3^7*x5^6                       
    108. x2*x4^6*x5^7                       109. x2*x4^7*x5^6                       
    110. x2^3*x3*x4^4*x5^6                  111. x2^3*x3*x4^6*x5^4                  
    112. x2^3*x3^3*x4^4*x5^4                113. x2^3*x3^5*x4^2*x5^4                
    114. x2^3*x3^5*x4^6                     115. x2^3*x3^5*x5^6                     
    116. x2^3*x4^5*x5^6                     117. x2^7*x3*x4^2*x5^4                  
    118. x2^7*x3*x4^6                       119. x2^7*x3*x5^6                       
    120. x2^7*x3^7                          121. x2^7*x4*x5^6                       
    122. x2^7*x4^7                          123. x2^7*x5^7                          
    124. x3*x4^6*x5^7                       125. x3*x4^7*x5^6                       
    126. x3^3*x4^5*x5^6                     127. x3^7*x4*x5^6                       
    128. x3^7*x4^7                          129. x3^7*x5^7                          
    130. x4^7*x5^7                          

* The space for weight vector \omega = (2, 4, 1) (Total: 15 basis elements):
  - Subgroup 1: Monomials with all positive exponents (15 elements)
    131. x1*x2^2*x3^2*x4^2*x5^7             132. x1*x2^2*x3^2*x4^3*x5^6             
    133. x1*x2^2*x3^2*x4^7*x5^2             134. x1*x2^2*x3^3*x4^2*x5^6             
    135. x1*x2^2*x3^3*x4^6*x5^2             136. x1*x2^2*x3^7*x4^2*x5^2             
    137. x1*x2^3*x3^2*x4^2*x5^6             138. x1*x2^3*x3^2*x4^6*x5^2             
    139. x1*x2^3*x3^6*x4^2*x5^2             140. x1*x2^7*x3^2*x4^2*x5^2             
    141. x1^3*x2*x3^2*x4^2*x5^6             142. x1^3*x2*x3^2*x4^6*x5^2             
    143. x1^3*x2*x3^6*x4^2*x5^2             144. x1^3*x2^5*x3^2*x4^2*x5^2           
    145. x1^7*x2*x3^2*x4^2*x5^2             

* The space for weight vector \omega = (4, 3, 1) (Total: 175 basis elements):
  - Subgroup 1: Monomials with all positive exponents (100 elements)
    146. x1*x2*x3^2*x4^3*x5^7               147. x1*x2*x3^2*x4^7*x5^3               
    148. x1*x2*x3^3*x4^2*x5^7               149. x1*x2*x3^3*x4^3*x5^6               
    150. x1*x2*x3^3*x4^6*x5^3               151. x1*x2*x3^3*x4^7*x5^2               
    152. x1*x2*x3^6*x4^3*x5^3               153. x1*x2*x3^7*x4^2*x5^3               
    154. x1*x2*x3^7*x4^3*x5^2               155. x1*x2^2*x3*x4^3*x5^7               
    156. x1*x2^2*x3*x4^7*x5^3               157. x1*x2^2*x3^3*x4*x5^7               
    158. x1*x2^2*x3^3*x4^3*x5^5             159. x1*x2^2*x3^3*x4^5*x5^3             
    160. x1*x2^2*x3^3*x4^7*x5               161. x1*x2^2*x3^5*x4^3*x5^3             
    162. x1*x2^2*x3^7*x4*x5^3               163. x1*x2^2*x3^7*x4^3*x5               
    164. x1*x2^3*x3*x4^2*x5^7               165. x1*x2^3*x3*x4^3*x5^6               
    166. x1*x2^3*x3*x4^6*x5^3               167. x1*x2^3*x3*x4^7*x5^2               
    168. x1*x2^3*x3^2*x4*x5^7               169. x1*x2^3*x3^2*x4^3*x5^5             
    170. x1*x2^3*x3^2*x4^5*x5^3             171. x1*x2^3*x3^2*x4^7*x5               
    172. x1*x2^3*x3^3*x4*x5^6               173. x1*x2^3*x3^3*x4^2*x5^5             
    174. x1*x2^3*x3^3*x4^3*x5^4             175. x1*x2^3*x3^3*x4^4*x5^3             
    176. x1*x2^3*x3^3*x4^5*x5^2             177. x1*x2^3*x3^3*x4^6*x5               
    178. x1*x2^3*x3^4*x4^3*x5^3             179. x1*x2^3*x3^5*x4^2*x5^3             
    180. x1*x2^3*x3^5*x4^3*x5^2             181. x1*x2^3*x3^6*x4*x5^3               
    182. x1*x2^3*x3^6*x4^3*x5               183. x1*x2^3*x3^7*x4*x5^2               
    184. x1*x2^3*x3^7*x4^2*x5               185. x1*x2^6*x3*x4^3*x5^3               
    186. x1*x2^6*x3^3*x4*x5^3               187. x1*x2^6*x3^3*x4^3*x5               
    188. x1*x2^7*x3*x4^2*x5^3               189. x1*x2^7*x3*x4^3*x5^2               
    190. x1*x2^7*x3^2*x4*x5^3               191. x1*x2^7*x3^2*x4^3*x5               
    192. x1*x2^7*x3^3*x4*x5^2               193. x1*x2^7*x3^3*x4^2*x5               
    194. x1^3*x2*x3*x4^2*x5^7               195. x1^3*x2*x3*x4^3*x5^6               
    196. x1^3*x2*x3*x4^6*x5^3               197. x1^3*x2*x3*x4^7*x5^2               
    198. x1^3*x2*x3^2*x4*x5^7               199. x1^3*x2*x3^2*x4^3*x5^5             
    200. x1^3*x2*x3^2*x4^5*x5^3             201. x1^3*x2*x3^2*x4^7*x5               
    202. x1^3*x2*x3^3*x4*x5^6               203. x1^3*x2*x3^3*x4^2*x5^5             
    204. x1^3*x2*x3^3*x4^3*x5^4             205. x1^3*x2*x3^3*x4^4*x5^3             
    206. x1^3*x2*x3^3*x4^5*x5^2             207. x1^3*x2*x3^3*x4^6*x5               
    208. x1^3*x2*x3^4*x4^3*x5^3             209. x1^3*x2*x3^5*x4^2*x5^3             
    210. x1^3*x2*x3^5*x4^3*x5^2             211. x1^3*x2*x3^6*x4*x5^3               
    212. x1^3*x2*x3^6*x4^3*x5               213. x1^3*x2*x3^7*x4*x5^2               
    214. x1^3*x2*x3^7*x4^2*x5               215. x1^3*x2^3*x3*x4*x5^6               
    216. x1^3*x2^3*x3*x4^2*x5^5             217. x1^3*x2^3*x3*x4^3*x5^4             
    218. x1^3*x2^3*x3*x4^4*x5^3             219. x1^3*x2^3*x3*x4^5*x5^2             
    220. x1^3*x2^3*x3*x4^6*x5               221. x1^3*x2^3*x3^3*x4*x5^4             
    222. x1^3*x2^3*x3^3*x4^4*x5             223. x1^3*x2^3*x3^4*x4*x5^3             
    224. x1^3*x2^3*x3^4*x4^3*x5             225. x1^3*x2^3*x3^5*x4*x5^2             
    226. x1^3*x2^3*x3^5*x4^2*x5             227. x1^3*x2^4*x3*x4^3*x5^3             
    228. x1^3*x2^4*x3^3*x4*x5^3             229. x1^3*x2^4*x3^3*x4^3*x5             
    230. x1^3*x2^5*x3*x4^2*x5^3             231. x1^3*x2^5*x3*x4^3*x5^2             
    232. x1^3*x2^5*x3^2*x4*x5^3             233. x1^3*x2^5*x3^2*x4^3*x5             
    234. x1^3*x2^5*x3^3*x4*x5^2             235. x1^3*x2^5*x3^3*x4^2*x5             
    236. x1^3*x2^7*x3*x4*x5^2               237. x1^3*x2^7*x3*x4^2*x5               
    238. x1^7*x2*x3*x4^2*x5^3               239. x1^7*x2*x3*x4^3*x5^2               
    240. x1^7*x2*x3^2*x4*x5^3               241. x1^7*x2*x3^2*x4^3*x5               
    242. x1^7*x2*x3^3*x4*x5^2               243. x1^7*x2*x3^3*x4^2*x5               
    244. x1^7*x2^3*x3*x4*x5^2               245. x1^7*x2^3*x3*x4^2*x5               
  - Subgroup 2: Monomials containing zero exponents (75 elements)
    246. x1*x2^3*x3^3*x4^7                  247. x1*x2^3*x3^3*x5^7                  
    248. x1*x2^3*x3^7*x4^3                  249. x1*x2^3*x3^7*x5^3                  
    250. x1*x2^3*x4^3*x5^7                  251. x1*x2^3*x4^7*x5^3                  
    252. x1*x2^7*x3^3*x4^3                  253. x1*x2^7*x3^3*x5^3                  
    254. x1*x2^7*x4^3*x5^3                  255. x1*x3^3*x4^3*x5^7                  
    256. x1*x3^3*x4^7*x5^3                  257. x1*x3^7*x4^3*x5^3                  
    258. x1^3*x2*x3^3*x4^7                  259. x1^3*x2*x3^3*x5^7                  
    260. x1^3*x2*x3^7*x4^3                  261. x1^3*x2*x3^7*x5^3                  
    262. x1^3*x2*x4^3*x5^7                  263. x1^3*x2*x4^7*x5^3                  
    264. x1^3*x2^3*x3*x4^7                  265. x1^3*x2^3*x3*x5^7                  
    266. x1^3*x2^3*x3^3*x4^5                267. x1^3*x2^3*x3^3*x5^5                
    268. x1^3*x2^3*x3^5*x4^3                269. x1^3*x2^3*x3^5*x5^3                
    270. x1^3*x2^3*x3^7*x4                  271. x1^3*x2^3*x3^7*x5                  
    272. x1^3*x2^3*x4*x5^7                  273. x1^3*x2^3*x4^3*x5^5                
    274. x1^3*x2^3*x4^5*x5^3                275. x1^3*x2^3*x4^7*x5                  
    276. x1^3*x2^5*x3^3*x4^3                277. x1^3*x2^5*x3^3*x5^3                
    278. x1^3*x2^5*x4^3*x5^3                279. x1^3*x2^7*x3*x4^3                  
    280. x1^3*x2^7*x3*x5^3                  281. x1^3*x2^7*x3^3*x4                  
    282. x1^3*x2^7*x3^3*x5                  283. x1^3*x2^7*x4*x5^3                  
    284. x1^3*x2^7*x4^3*x5                  285. x1^3*x3*x4^3*x5^7                  
    286. x1^3*x3*x4^7*x5^3                  287. x1^3*x3^3*x4*x5^7                  
    288. x1^3*x3^3*x4^3*x5^5                289. x1^3*x3^3*x4^5*x5^3                
    290. x1^3*x3^3*x4^7*x5                  291. x1^3*x3^5*x4^3*x5^3                
    292. x1^3*x3^7*x4*x5^3                  293. x1^3*x3^7*x4^3*x5                  
    294. x1^7*x2*x3^3*x4^3                  295. x1^7*x2*x3^3*x5^3                  
    296. x1^7*x2*x4^3*x5^3                  297. x1^7*x2^3*x3*x4^3                  
    298. x1^7*x2^3*x3*x5^3                  299. x1^7*x2^3*x3^3*x4                  
    300. x1^7*x2^3*x3^3*x5                  301. x1^7*x2^3*x4*x5^3                  
    302. x1^7*x2^3*x4^3*x5                  303. x1^7*x3*x4^3*x5^3                  
    304. x1^7*x3^3*x4*x5^3                  305. x1^7*x3^3*x4^3*x5                  
    306. x2*x3^3*x4^3*x5^7                  307. x2*x3^3*x4^7*x5^3                  
    308. x2*x3^7*x4^3*x5^3                  309. x2^3*x3*x4^3*x5^7                  
    310. x2^3*x3*x4^7*x5^3                  311. x2^3*x3^3*x4*x5^7                  
    312. x2^3*x3^3*x4^3*x5^5                313. x2^3*x3^3*x4^5*x5^3                
    314. x2^3*x3^3*x4^7*x5                  315. x2^3*x3^5*x4^3*x5^3                
    316. x2^3*x3^7*x4*x5^3                  317. x2^3*x3^7*x4^3*x5                  
    318. x2^7*x3*x4^3*x5^3                  319. x2^7*x3^3*x4*x5^3                  
    320. x2^7*x3^3*x4^3*x5                  



============================================================================
Final Conclusion: \dim (Q\mathcal{P}_{5})_{14} = 320

Total execution time: 4.99 seconds
============================================================================
\end{lstlisting}

\medskip

This result is consistent with previous manual computations given in~\cite{Phuc}.

\medskip

Moreover, our algorithm explicitly lists all \textbf{inadmissible monomials} of degree~14, incorporating an application of Theorem~\ref{dlKS} and a classification by groups of weight vectors, as shown in the algorithm output below.

\medskip

\begin{lstlisting}
============================================================================
                 FINAL ANALYSIS RESULTS
============================================================================
Case: k = 5, d = 14

Note on filtering rules: The following lists are generated after applying these theoretical filters:
1. Includes only inadmissible monomials with ALL POSITIVE exponents.
2. EXCLUDES those hit by Singer's criterion (where w(x) < w(minimal_spike)).
3. EXCLUDES those where the exponent of x1 is a POSITIVE EVEN number.

--- MONOMIALS FILTERED BY KAMEKO-SUM THEOREM (TYPE I) ---
Found and filtered 204 inadmissible(s) of the form x * w^(2^r):
  1. x1*x2*x3*x4*x5^10  -->  (x1*x2*x3*x4) * (x5^5)^2 (w is inadmissible of degree 5)
  2. x1*x2*x3*x4^2*x5^9  -->  (x1*x2*x3*x5) * (x4*x5^4)^2 (w is inadmissible of degree 5)
  3. x1*x2*x3^2*x4*x5^9  -->  (x1*x2*x4*x5) * (x3*x5^4)^2 (w is inadmissible of degree 5)
  4. x1*x2^2*x3*x4*x5^9  -->  (x1*x3*x4*x5) * (x2*x5^4)^2 (w is inadmissible of degree 5)
  5. x1*x2*x3*x4^3*x5^8  -->  (x1*x2*x3*x4) * (x4*x5^4)^2 (w is inadmissible of degree 5)
  6. x1*x2*x3^3*x4*x5^8  -->  (x1*x2*x3*x4) * (x3*x5^4)^2 (w is inadmissible of degree 5)
  7. x1*x2^3*x3*x4*x5^8  -->  (x1*x2*x3*x4) * (x2*x5^4)^2 (w is inadmissible of degree 5)
  8. x1^3*x2*x3*x4*x5^8  -->  (x1*x2*x3*x4) * (x1*x5^4)^2 (w is inadmissible of degree 5)
  9. x1*x2*x3*x4^4*x5^7  -->  (x1*x2*x3*x5) * (x4^2*x5^3)^2 (w is inadmissible of degree 5)
  10. x1*x2*x3^4*x4*x5^7  -->  (x1*x2*x4*x5) * (x3^2*x5^3)^2 (w is inadmissible of degree 5)
  11. x1*x2^4*x3*x4*x5^7  -->  (x1*x3*x4*x5) * (x2^2*x5^3)^2 (w is inadmissible of degree 5)
  12. x1*x2*x3*x4^5*x5^6  -->  (x1*x2*x3*x4) * (x4^2*x5^3)^2 (w is inadmissible of degree 5)
  13. x1*x2*x3^4*x4^2*x5^6  -->  (x1*x2) * (x3^2*x4*x5^3)^2 (w is inadmissible of degree 6)
  14. x1*x2^4*x3*x4^2*x5^6  -->  (x1*x3) * (x2^2*x4*x5^3)^2 (w is inadmissible of degree 6)
  15. x1*x2*x3^5*x4*x5^6  -->  (x1*x2*x3*x4) * (x3^2*x5^3)^2 (w is inadmissible of degree 5)
  16. x1*x2^4*x3^2*x4*x5^6  -->  (x1*x4) * (x2^2*x3*x5^3)^2 (w is inadmissible of degree 6)
  17. x1*x2^5*x3*x4*x5^6  -->  (x1*x2*x3*x4) * (x2^2*x5^3)^2 (w is inadmissible of degree 5)
  18. x1^5*x2*x3*x4*x5^6  -->  (x1*x2*x3*x4) * (x1^2*x5^3)^2 (w is inadmissible of degree 5)
  19. x1*x2*x3*x4^6*x5^5  -->  (x1*x2*x3*x5) * (x4^3*x5^2)^2 (w is inadmissible of degree 5)
  20. x1*x2*x3^2*x4^5*x5^5  -->  (x1*x2*x4*x5) * (x3*x4^2*x5^2)^2 (w is inadmissible of degree 5)
  21. x1*x2^2*x3*x4^5*x5^5  -->  (x1*x3*x4*x5) * (x2*x4^2*x5^2)^2 (w is inadmissible of degree 5)
  22. x1*x2*x3^3*x4^4*x5^5  -->  (x1*x2*x3*x5) * (x3*x4^2*x5^2)^2 (w is inadmissible of degree 5)
  23. x1*x2^3*x3*x4^4*x5^5  -->  (x1*x2*x3*x5) * (x2*x4^2*x5^2)^2 (w is inadmissible of degree 5)
  24. x1^3*x2*x3*x4^4*x5^5  -->  (x1*x2*x3*x5) * (x1*x4^2*x5^2)^2 (w is inadmissible of degree 5)
  25. x1*x2*x3^4*x4^3*x5^5  -->  (x1*x2*x4*x5) * (x3^2*x4*x5^2)^2 (w is inadmissible of degree 5)
  26. x1*x2^4*x3*x4^3*x5^5  -->  (x1*x3*x4*x5) * (x2^2*x4*x5^2)^2 (w is inadmissible of degree 5)
  27. x1*x2*x3^5*x4^2*x5^5  -->  (x1*x2*x3*x5) * (x3^2*x4*x5^2)^2 (w is inadmissible of degree 5)
  28. x1*x2^4*x3^2*x4^2*x5^5  -->  (x1*x5) * (x2^2*x3*x4*x5^2)^2 (w is inadmissible of degree 6)
  29. x1*x2^5*x3*x4^2*x5^5  -->  (x1*x2*x3*x5) * (x2^2*x4*x5^2)^2 (w is inadmissible of degree 5)
  30. x1^5*x2*x3*x4^2*x5^5  -->  (x1*x2*x3*x5) * (x1^2*x4*x5^2)^2 (w is inadmissible of degree 5)
  31. x1*x2*x3^6*x4*x5^5  -->  (x1*x2*x4*x5) * (x3^3*x5^2)^2 (w is inadmissible of degree 5)
  32. x1*x2^2*x3^5*x4*x5^5  -->  (x1*x3*x4*x5) * (x2*x3^2*x5^2)^2 (w is inadmissible of degree 5)
  33. x1*x2^3*x3^4*x4*x5^5  -->  (x1*x2*x4*x5) * (x2*x3^2*x5^2)^2 (w is inadmissible of degree 5)
  34. x1^3*x2*x3^4*x4*x5^5  -->  (x1*x2*x4*x5) * (x1*x3^2*x5^2)^2 (w is inadmissible of degree 5)
  35. x1*x2^4*x3^3*x4*x5^5  -->  (x1*x3*x4*x5) * (x2^2*x3*x5^2)^2 (w is inadmissible of degree 5)
  36. x1*x2^5*x3^2*x4*x5^5  -->  (x1*x2*x4*x5) * (x2^2*x3*x5^2)^2 (w is inadmissible of degree 5)
  37. x1^5*x2*x3^2*x4*x5^5  -->  (x1*x2*x4*x5) * (x1^2*x3*x5^2)^2 (w is inadmissible of degree 5)
  38. x1*x2^6*x3*x4*x5^5  -->  (x1*x3*x4*x5) * (x2^3*x5^2)^2 (w is inadmissible of degree 5)
  39. x1^3*x2^4*x3*x4*x5^5  -->  (x1*x3*x4*x5) * (x1*x2^2*x5^2)^2 (w is inadmissible of degree 5)
  40. x1^5*x2^2*x3*x4*x5^5  -->  (x1*x3*x4*x5) * (x1^2*x2*x5^2)^2 (w is inadmissible of degree 5)
  41. x1*x2*x3*x4^7*x5^4  -->  (x1*x2*x3*x4) * (x4^3*x5^2)^2 (w is inadmissible of degree 5)
  42. x1*x2*x3^3*x4^5*x5^4  -->  (x1*x2*x3*x4) * (x3*x4^2*x5^2)^2 (w is inadmissible of degree 5)
  43. x1*x2^3*x3*x4^5*x5^4  -->  (x1*x2*x3*x4) * (x2*x4^2*x5^2)^2 (w is inadmissible of degree 5)
  44. x1^3*x2*x3*x4^5*x5^4  -->  (x1*x2*x3*x4) * (x1*x4^2*x5^2)^2 (w is inadmissible of degree 5)
  45. x1*x2*x3^5*x4^3*x5^4  -->  (x1*x2*x3*x4) * (x3^2*x4*x5^2)^2 (w is inadmissible of degree 5)
  46. x1*x2^4*x3^2*x4^3*x5^4  -->  (x1*x4) * (x2^2*x3*x4*x5^2)^2 (w is inadmissible of degree 6)
  47. x1*x2^5*x3*x4^3*x5^4  -->  (x1*x2*x3*x4) * (x2^2*x4*x5^2)^2 (w is inadmissible of degree 5)
  48. x1^5*x2*x3*x4^3*x5^4  -->  (x1*x2*x3*x4) * (x1^2*x4*x5^2)^2 (w is inadmissible of degree 5)
  49. x1*x2^4*x3^3*x4^2*x5^4  -->  (x1*x3) * (x2^2*x3*x4*x5^2)^2 (w is inadmissible of degree 6)
  50. x1*x2^5*x3^2*x4^2*x5^4  -->  (x1*x2) * (x2^2*x3*x4*x5^2)^2 (w is inadmissible of degree 6)
  51. x1^5*x2*x3^2*x4^2*x5^4  -->  (x1*x2) * (x1^2*x3*x4*x5^2)^2 (w is inadmissible of degree 6)
  52. x1^5*x2^2*x3*x4^2*x5^4  -->  (x1*x3) * (x1^2*x2*x4*x5^2)^2 (w is inadmissible of degree 6)
  53. x1*x2*x3^7*x4*x5^4  -->  (x1*x2*x3*x4) * (x3^3*x5^2)^2 (w is inadmissible of degree 5)
  54. x1*x2^3*x3^5*x4*x5^4  -->  (x1*x2*x3*x4) * (x2*x3^2*x5^2)^2 (w is inadmissible of degree 5)
  55. x1^3*x2*x3^5*x4*x5^4  -->  (x1*x2*x3*x4) * (x1*x3^2*x5^2)^2 (w is inadmissible of degree 5)
  56. x1*x2^5*x3^3*x4*x5^4  -->  (x1*x2*x3*x4) * (x2^2*x3*x5^2)^2 (w is inadmissible of degree 5)
  57. x1^5*x2*x3^3*x4*x5^4  -->  (x1*x2*x3*x4) * (x1^2*x3*x5^2)^2 (w is inadmissible of degree 5)
  58. x1^5*x2^2*x3^2*x4*x5^4  -->  (x1*x4) * (x1^2*x2*x3*x5^2)^2 (w is inadmissible of degree 6)
  59. x1*x2^7*x3*x4*x5^4  -->  (x1*x2*x3*x4) * (x2^3*x5^2)^2 (w is inadmissible of degree 5)
  60. x1^3*x2^5*x3*x4*x5^4  -->  (x1*x2*x3*x4) * (x1*x2^2*x5^2)^2 (w is inadmissible of degree 5)
  61. x1^5*x2^3*x3*x4*x5^4  -->  (x1*x2*x3*x4) * (x1^2*x2*x5^2)^2 (w is inadmissible of degree 5)
  62. x1^7*x2*x3*x4*x5^4  -->  (x1*x2*x3*x4) * (x1^3*x5^2)^2 (w is inadmissible of degree 5)
  63. x1*x2*x3*x4^8*x5^3  -->  (x1*x2*x3*x5) * (x4^4*x5)^2 (w is inadmissible of degree 5)
  64. x1*x2*x3^4*x4^5*x5^3  -->  (x1*x2*x4*x5) * (x3^2*x4^2*x5)^2 (w is inadmissible of degree 5)
  65. x1*x2^4*x3*x4^5*x5^3  -->  (x1*x3*x4*x5) * (x2^2*x4^2*x5)^2 (w is inadmissible of degree 5)
  66. x1*x2*x3^5*x4^4*x5^3  -->  (x1*x2*x3*x5) * (x3^2*x4^2*x5)^2 (w is inadmissible of degree 5)
  67. x1*x2^2*x3^4*x4^4*x5^3  -->  (x1*x5) * (x2*x3^2*x4^2*x5)^2 (w is inadmissible of degree 6)
  68. x1*x2^4*x3^2*x4^4*x5^3  -->  (x1*x5) * (x2^2*x3*x4^2*x5)^2 (w is inadmissible of degree 6)
  69. x1*x2^5*x3*x4^4*x5^3  -->  (x1*x2*x3*x5) * (x2^2*x4^2*x5)^2 (w is inadmissible of degree 5)
  70. x1^5*x2*x3*x4^4*x5^3  -->  (x1*x2*x3*x5) * (x1^2*x4^2*x5)^2 (w is inadmissible of degree 5)
  71. x1*x2^4*x3^3*x4^3*x5^3  -->  (x1*x3*x4*x5) * (x2^2*x3*x4*x5)^2 (w is inadmissible of degree 5)
  72. x1*x2^5*x3^2*x4^3*x5^3  -->  (x1*x2*x4*x5) * (x2^2*x3*x4*x5)^2 (w is inadmissible of degree 5)
  73. x1^5*x2*x3^2*x4^3*x5^3  -->  (x1*x2*x4*x5) * (x1^2*x3*x4*x5)^2 (w is inadmissible of degree 5)
  74. x1^5*x2^2*x3*x4^3*x5^3  -->  (x1*x3*x4*x5) * (x1^2*x2*x4*x5)^2 (w is inadmissible of degree 5)
  75. x1*x2^4*x3^4*x4^2*x5^3  -->  (x1*x5) * (x2^2*x3^2*x4*x5)^2 (w is inadmissible of degree 6)
  76. x1*x2^5*x3^3*x4^2*x5^3  -->  (x1*x2*x3*x5) * (x2^2*x3*x4*x5)^2 (w is inadmissible of degree 5)
  77. x1^5*x2*x3^3*x4^2*x5^3  -->  (x1*x2*x3*x5) * (x1^2*x3*x4*x5)^2 (w is inadmissible of degree 5)
  78. x1^5*x2^2*x3^2*x4^2*x5^3  -->  (x1*x5) * (x1^2*x2*x3*x4*x5)^2 (w is inadmissible of degree 6)
  79. x1^5*x2^3*x3*x4^2*x5^3  -->  (x1*x2*x3*x5) * (x1^2*x2*x4*x5)^2 (w is inadmissible of degree 5)
  80. x1*x2*x3^8*x4*x5^3  -->  (x1*x2*x4*x5) * (x3^4*x5)^2 (w is inadmissible of degree 5)
  81. x1*x2^4*x3^5*x4*x5^3  -->  (x1*x3*x4*x5) * (x2^2*x3^2*x5)^2 (w is inadmissible of degree 5)
  82. x1*x2^5*x3^4*x4*x5^3  -->  (x1*x2*x4*x5) * (x2^2*x3^2*x5)^2 (w is inadmissible of degree 5)
  83. x1^5*x2*x3^4*x4*x5^3  -->  (x1*x2*x4*x5) * (x1^2*x3^2*x5)^2 (w is inadmissible of degree 5)
  84. x1^5*x2^2*x3^3*x4*x5^3  -->  (x1*x3*x4*x5) * (x1^2*x2*x3*x5)^2 (w is inadmissible of degree 5)
  85. x1^5*x2^3*x3^2*x4*x5^3  -->  (x1*x2*x4*x5) * (x1^2*x2*x3*x5)^2 (w is inadmissible of degree 5)
  86. x1*x2^8*x3*x4*x5^3  -->  (x1*x3*x4*x5) * (x2^4*x5)^2 (w is inadmissible of degree 5)
  87. x1^5*x2^4*x3*x4*x5^3  -->  (x1*x3*x4*x5) * (x1^2*x2^2*x5)^2 (w is inadmissible of degree 5)
  88. x1*x2*x3*x4^9*x5^2  -->  (x1*x2*x3*x4) * (x4^4*x5)^2 (w is inadmissible of degree 5)
  89. x1*x2*x3^4*x4^6*x5^2  -->  (x1*x2) * (x3^2*x4^3*x5)^2 (w is inadmissible of degree 6)
  90. x1*x2^4*x3*x4^6*x5^2  -->  (x1*x3) * (x2^2*x4^3*x5)^2 (w is inadmissible of degree 6)
  91. x1*x2*x3^5*x4^5*x5^2  -->  (x1*x2*x3*x4) * (x3^2*x4^2*x5)^2 (w is inadmissible of degree 5)
  92. x1*x2^2*x3^4*x4^5*x5^2  -->  (x1*x4) * (x2*x3^2*x4^2*x5)^2 (w is inadmissible of degree 6)
  93. x1*x2^4*x3^2*x4^5*x5^2  -->  (x1*x4) * (x2^2*x3*x4^2*x5)^2 (w is inadmissible of degree 6)
  94. x1*x2^5*x3*x4^5*x5^2  -->  (x1*x2*x3*x4) * (x2^2*x4^2*x5)^2 (w is inadmissible of degree 5)
  95. x1^5*x2*x3*x4^5*x5^2  -->  (x1*x2*x3*x4) * (x1^2*x4^2*x5)^2 (w is inadmissible of degree 5)
  96. x1*x2*x3^6*x4^4*x5^2  -->  (x1*x2) * (x3^3*x4^2*x5)^2 (w is inadmissible of degree 6)
  97. x1*x2^2*x3^5*x4^4*x5^2  -->  (x1*x3) * (x2*x3^2*x4^2*x5)^2 (w is inadmissible of degree 6)
  98. x1*x2^3*x3^4*x4^4*x5^2  -->  (x1*x2) * (x2*x3^2*x4^2*x5)^2 (w is inadmissible of degree 6)
  99. x1^3*x2*x3^4*x4^4*x5^2  -->  (x1*x2) * (x1*x3^2*x4^2*x5)^2 (w is inadmissible of degree 6)
  100. x1*x2^4*x3^3*x4^4*x5^2  -->  (x1*x3) * (x2^2*x3*x4^2*x5)^2 (w is inadmissible of degree 6)
  101. x1*x2^5*x3^2*x4^4*x5^2  -->  (x1*x2) * (x2^2*x3*x4^2*x5)^2 (w is inadmissible of degree 6)
  102. x1^5*x2*x3^2*x4^4*x5^2  -->  (x1*x2) * (x1^2*x3*x4^2*x5)^2 (w is inadmissible of degree 6)
  103. x1*x2^6*x3*x4^4*x5^2  -->  (x1*x3) * (x2^3*x4^2*x5)^2 (w is inadmissible of degree 6)
  104. x1^3*x2^4*x3*x4^4*x5^2  -->  (x1*x3) * (x1*x2^2*x4^2*x5)^2 (w is inadmissible of degree 6)
  105. x1^5*x2^2*x3*x4^4*x5^2  -->  (x1*x3) * (x1^2*x2*x4^2*x5)^2 (w is inadmissible of degree 6)
  106. x1*x2^4*x3^4*x4^3*x5^2  -->  (x1*x4) * (x2^2*x3^2*x4*x5)^2 (w is inadmissible of degree 6)
  107. x1*x2^5*x3^3*x4^3*x5^2  -->  (x1*x2*x3*x4) * (x2^2*x3*x4*x5)^2 (w is inadmissible of degree 5)
  108. x1^5*x2*x3^3*x4^3*x5^2  -->  (x1*x2*x3*x4) * (x1^2*x3*x4*x5)^2 (w is inadmissible of degree 5)
  109. x1^5*x2^2*x3^2*x4^3*x5^2  -->  (x1*x4) * (x1^2*x2*x3*x4*x5)^2 (w is inadmissible of degree 6)
  110. x1^5*x2^3*x3*x4^3*x5^2  -->  (x1*x2*x3*x4) * (x1^2*x2*x4*x5)^2 (w is inadmissible of degree 5)
  111. x1*x2^4*x3^5*x4^2*x5^2  -->  (x1*x3) * (x2^2*x3^2*x4*x5)^2 (w is inadmissible of degree 6)
  112. x1*x2^5*x3^4*x4^2*x5^2  -->  (x1*x2) * (x2^2*x3^2*x4*x5)^2 (w is inadmissible of degree 6)
  113. x1^5*x2*x3^4*x4^2*x5^2  -->  (x1*x2) * (x1^2*x3^2*x4*x5)^2 (w is inadmissible of degree 6)
  114. x1^5*x2^2*x3^3*x4^2*x5^2  -->  (x1*x3) * (x1^2*x2*x3*x4*x5)^2 (w is inadmissible of degree 6)
  115. x1^5*x2^3*x3^2*x4^2*x5^2  -->  (x1*x2) * (x1^2*x2*x3*x4*x5)^2 (w is inadmissible of degree 6)
  116. x1^5*x2^4*x3*x4^2*x5^2  -->  (x1*x3) * (x1^2*x2^2*x4*x5)^2 (w is inadmissible of degree 6)
  117. x1*x2*x3^9*x4*x5^2  -->  (x1*x2*x3*x4) * (x3^4*x5)^2 (w is inadmissible of degree 5)
  118. x1*x2^4*x3^6*x4*x5^2  -->  (x1*x4) * (x2^2*x3^3*x5)^2 (w is inadmissible of degree 6)
  119. x1*x2^5*x3^5*x4*x5^2  -->  (x1*x2*x3*x4) * (x2^2*x3^2*x5)^2 (w is inadmissible of degree 5)
  120. x1^5*x2*x3^5*x4*x5^2  -->  (x1*x2*x3*x4) * (x1^2*x3^2*x5)^2 (w is inadmissible of degree 5)
  121. x1*x2^6*x3^4*x4*x5^2  -->  (x1*x4) * (x2^3*x3^2*x5)^2 (w is inadmissible of degree 6)
  122. x1^3*x2^4*x3^4*x4*x5^2  -->  (x1*x4) * (x1*x2^2*x3^2*x5)^2 (w is inadmissible of degree 6)
  123. x1^5*x2^2*x3^4*x4*x5^2  -->  (x1*x4) * (x1^2*x2*x3^2*x5)^2 (w is inadmissible of degree 6)
  124. x1^5*x2^3*x3^3*x4*x5^2  -->  (x1*x2*x3*x4) * (x1^2*x2*x3*x5)^2 (w is inadmissible of degree 5)
  125. x1^5*x2^4*x3^2*x4*x5^2  -->  (x1*x4) * (x1^2*x2^2*x3*x5)^2 (w is inadmissible of degree 6)
  126. x1*x2^9*x3*x4*x5^2  -->  (x1*x2*x3*x4) * (x2^4*x5)^2 (w is inadmissible of degree 5)
  127. x1^5*x2^5*x3*x4*x5^2  -->  (x1*x2*x3*x4) * (x1^2*x2^2*x5)^2 (w is inadmissible of degree 5)
  128. x1^9*x2*x3*x4*x5^2  -->  (x1*x2*x3*x4) * (x1^4*x5)^2 (w is inadmissible of degree 5)
  129. x1*x2*x3*x4^10*x5  -->  (x1*x2*x3*x5) * (x4^5)^2 (w is inadmissible of degree 5)
  130. x1*x2*x3^2*x4^9*x5  -->  (x1*x2*x4*x5) * (x3*x4^4)^2 (w is inadmissible of degree 5)
  131. x1*x2^2*x3*x4^9*x5  -->  (x1*x3*x4*x5) * (x2*x4^4)^2 (w is inadmissible of degree 5)
  132. x1*x2*x3^3*x4^8*x5  -->  (x1*x2*x3*x5) * (x3*x4^4)^2 (w is inadmissible of degree 5)
  133. x1*x2^3*x3*x4^8*x5  -->  (x1*x2*x3*x5) * (x2*x4^4)^2 (w is inadmissible of degree 5)
  134. x1^3*x2*x3*x4^8*x5  -->  (x1*x2*x3*x5) * (x1*x4^4)^2 (w is inadmissible of degree 5)
  135. x1*x2*x3^4*x4^7*x5  -->  (x1*x2*x4*x5) * (x3^2*x4^3)^2 (w is inadmissible of degree 5)
  136. x1*x2^4*x3*x4^7*x5  -->  (x1*x3*x4*x5) * (x2^2*x4^3)^2 (w is inadmissible of degree 5)
  137. x1*x2*x3^5*x4^6*x5  -->  (x1*x2*x3*x5) * (x3^2*x4^3)^2 (w is inadmissible of degree 5)
  138. x1*x2^4*x3^2*x4^6*x5  -->  (x1*x5) * (x2^2*x3*x4^3)^2 (w is inadmissible of degree 6)
  139. x1*x2^5*x3*x4^6*x5  -->  (x1*x2*x3*x5) * (x2^2*x4^3)^2 (w is inadmissible of degree 5)
  140. x1^5*x2*x3*x4^6*x5  -->  (x1*x2*x3*x5) * (x1^2*x4^3)^2 (w is inadmissible of degree 5)
  141. x1*x2*x3^6*x4^5*x5  -->  (x1*x2*x4*x5) * (x3^3*x4^2)^2 (w is inadmissible of degree 5)
  142. x1*x2^2*x3^5*x4^5*x5  -->  (x1*x3*x4*x5) * (x2*x3^2*x4^2)^2 (w is inadmissible of degree 5)
  143. x1*x2^3*x3^4*x4^5*x5  -->  (x1*x2*x4*x5) * (x2*x3^2*x4^2)^2 (w is inadmissible of degree 5)
  144. x1^3*x2*x3^4*x4^5*x5  -->  (x1*x2*x4*x5) * (x1*x3^2*x4^2)^2 (w is inadmissible of degree 5)
  145. x1*x2^4*x3^3*x4^5*x5  -->  (x1*x3*x4*x5) * (x2^2*x3*x4^2)^2 (w is inadmissible of degree 5)
  146. x1*x2^5*x3^2*x4^5*x5  -->  (x1*x2*x4*x5) * (x2^2*x3*x4^2)^2 (w is inadmissible of degree 5)
  147. x1^5*x2*x3^2*x4^5*x5  -->  (x1*x2*x4*x5) * (x1^2*x3*x4^2)^2 (w is inadmissible of degree 5)
  148. x1*x2^6*x3*x4^5*x5  -->  (x1*x3*x4*x5) * (x2^3*x4^2)^2 (w is inadmissible of degree 5)
  149. x1^3*x2^4*x3*x4^5*x5  -->  (x1*x3*x4*x5) * (x1*x2^2*x4^2)^2 (w is inadmissible of degree 5)
  150. x1^5*x2^2*x3*x4^5*x5  -->  (x1*x3*x4*x5) * (x1^2*x2*x4^2)^2 (w is inadmissible of degree 5)
  151. x1*x2*x3^7*x4^4*x5  -->  (x1*x2*x3*x5) * (x3^3*x4^2)^2 (w is inadmissible of degree 5)
  152. x1*x2^3*x3^5*x4^4*x5  -->  (x1*x2*x3*x5) * (x2*x3^2*x4^2)^2 (w is inadmissible of degree 5)
  153. x1^3*x2*x3^5*x4^4*x5  -->  (x1*x2*x3*x5) * (x1*x3^2*x4^2)^2 (w is inadmissible of degree 5)
  154. x1*x2^5*x3^3*x4^4*x5  -->  (x1*x2*x3*x5) * (x2^2*x3*x4^2)^2 (w is inadmissible of degree 5)
  155. x1^5*x2*x3^3*x4^4*x5  -->  (x1*x2*x3*x5) * (x1^2*x3*x4^2)^2 (w is inadmissible of degree 5)
  156. x1^5*x2^2*x3^2*x4^4*x5  -->  (x1*x5) * (x1^2*x2*x3*x4^2)^2 (w is inadmissible of degree 6)
  157. x1*x2^7*x3*x4^4*x5  -->  (x1*x2*x3*x5) * (x2^3*x4^2)^2 (w is inadmissible of degree 5)
  158. x1^3*x2^5*x3*x4^4*x5  -->  (x1*x2*x3*x5) * (x1*x2^2*x4^2)^2 (w is inadmissible of degree 5)
  159. x1^5*x2^3*x3*x4^4*x5  -->  (x1*x2*x3*x5) * (x1^2*x2*x4^2)^2 (w is inadmissible of degree 5)
  160. x1^7*x2*x3*x4^4*x5  -->  (x1*x2*x3*x5) * (x1^3*x4^2)^2 (w is inadmissible of degree 5)
  161. x1*x2*x3^8*x4^3*x5  -->  (x1*x2*x4*x5) * (x3^4*x4)^2 (w is inadmissible of degree 5)
  162. x1*x2^4*x3^5*x4^3*x5  -->  (x1*x3*x4*x5) * (x2^2*x3^2*x4)^2 (w is inadmissible of degree 5)
  163. x1*x2^5*x3^4*x4^3*x5  -->  (x1*x2*x4*x5) * (x2^2*x3^2*x4)^2 (w is inadmissible of degree 5)
  164. x1^5*x2*x3^4*x4^3*x5  -->  (x1*x2*x4*x5) * (x1^2*x3^2*x4)^2 (w is inadmissible of degree 5)
  165. x1^5*x2^2*x3^3*x4^3*x5  -->  (x1*x3*x4*x5) * (x1^2*x2*x3*x4)^2 (w is inadmissible of degree 5)
  166. x1^5*x2^3*x3^2*x4^3*x5  -->  (x1*x2*x4*x5) * (x1^2*x2*x3*x4)^2 (w is inadmissible of degree 5)
  167. x1*x2^8*x3*x4^3*x5  -->  (x1*x3*x4*x5) * (x2^4*x4)^2 (w is inadmissible of degree 5)
  168. x1^5*x2^4*x3*x4^3*x5  -->  (x1*x3*x4*x5) * (x1^2*x2^2*x4)^2 (w is inadmissible of degree 5)
  169. x1*x2*x3^9*x4^2*x5  -->  (x1*x2*x3*x5) * (x3^4*x4)^2 (w is inadmissible of degree 5)
  170. x1*x2^4*x3^6*x4^2*x5  -->  (x1*x5) * (x2^2*x3^3*x4)^2 (w is inadmissible of degree 6)
  171. x1*x2^5*x3^5*x4^2*x5  -->  (x1*x2*x3*x5) * (x2^2*x3^2*x4)^2 (w is inadmissible of degree 5)
  172. x1^5*x2*x3^5*x4^2*x5  -->  (x1*x2*x3*x5) * (x1^2*x3^2*x4)^2 (w is inadmissible of degree 5)
  173. x1*x2^6*x3^4*x4^2*x5  -->  (x1*x5) * (x2^3*x3^2*x4)^2 (w is inadmissible of degree 6)
  174. x1^3*x2^4*x3^4*x4^2*x5  -->  (x1*x5) * (x1*x2^2*x3^2*x4)^2 (w is inadmissible of degree 6)
  175. x1^5*x2^2*x3^4*x4^2*x5  -->  (x1*x5) * (x1^2*x2*x3^2*x4)^2 (w is inadmissible of degree 6)
  176. x1^5*x2^3*x3^3*x4^2*x5  -->  (x1*x2*x3*x5) * (x1^2*x2*x3*x4)^2 (w is inadmissible of degree 5)
  177. x1^5*x2^4*x3^2*x4^2*x5  -->  (x1*x5) * (x1^2*x2^2*x3*x4)^2 (w is inadmissible of degree 6)
  178. x1*x2^9*x3*x4^2*x5  -->  (x1*x2*x3*x5) * (x2^4*x4)^2 (w is inadmissible of degree 5)
  179. x1^5*x2^5*x3*x4^2*x5  -->  (x1*x2*x3*x5) * (x1^2*x2^2*x4)^2 (w is inadmissible of degree 5)
  180. x1^9*x2*x3*x4^2*x5  -->  (x1*x2*x3*x5) * (x1^4*x4)^2 (w is inadmissible of degree 5)
  181. x1*x2*x3^10*x4*x5  -->  (x1*x2*x4*x5) * (x3^5)^2 (w is inadmissible of degree 5)
  182. x1*x2^2*x3^9*x4*x5  -->  (x1*x3*x4*x5) * (x2*x3^4)^2 (w is inadmissible of degree 5)
  183. x1*x2^3*x3^8*x4*x5  -->  (x1*x2*x4*x5) * (x2*x3^4)^2 (w is inadmissible of degree 5)
  184. x1^3*x2*x3^8*x4*x5  -->  (x1*x2*x4*x5) * (x1*x3^4)^2 (w is inadmissible of degree 5)
  185. x1*x2^4*x3^7*x4*x5  -->  (x1*x3*x4*x5) * (x2^2*x3^3)^2 (w is inadmissible of degree 5)
  186. x1*x2^5*x3^6*x4*x5  -->  (x1*x2*x4*x5) * (x2^2*x3^3)^2 (w is inadmissible of degree 5)
  187. x1^5*x2*x3^6*x4*x5  -->  (x1*x2*x4*x5) * (x1^2*x3^3)^2 (w is inadmissible of degree 5)
  188. x1*x2^6*x3^5*x4*x5  -->  (x1*x3*x4*x5) * (x2^3*x3^2)^2 (w is inadmissible of degree 5)
  189. x1^3*x2^4*x3^5*x4*x5  -->  (x1*x3*x4*x5) * (x1*x2^2*x3^2)^2 (w is inadmissible of degree 5)
  190. x1^5*x2^2*x3^5*x4*x5  -->  (x1*x3*x4*x5) * (x1^2*x2*x3^2)^2 (w is inadmissible of degree 5)
  191. x1*x2^7*x3^4*x4*x5  -->  (x1*x2*x4*x5) * (x2^3*x3^2)^2 (w is inadmissible of degree 5)
  192. x1^3*x2^5*x3^4*x4*x5  -->  (x1*x2*x4*x5) * (x1*x2^2*x3^2)^2 (w is inadmissible of degree 5)
  193. x1^5*x2^3*x3^4*x4*x5  -->  (x1*x2*x4*x5) * (x1^2*x2*x3^2)^2 (w is inadmissible of degree 5)
  194. x1^7*x2*x3^4*x4*x5  -->  (x1*x2*x4*x5) * (x1^3*x3^2)^2 (w is inadmissible of degree 5)
  195. x1*x2^8*x3^3*x4*x5  -->  (x1*x3*x4*x5) * (x2^4*x3)^2 (w is inadmissible of degree 5)
  196. x1^5*x2^4*x3^3*x4*x5  -->  (x1*x3*x4*x5) * (x1^2*x2^2*x3)^2 (w is inadmissible of degree 5)
  197. x1*x2^9*x3^2*x4*x5  -->  (x1*x2*x4*x5) * (x2^4*x3)^2 (w is inadmissible of degree 5)
  198. x1^5*x2^5*x3^2*x4*x5  -->  (x1*x2*x4*x5) * (x1^2*x2^2*x3)^2 (w is inadmissible of degree 5)
  199. x1^9*x2*x3^2*x4*x5  -->  (x1*x2*x4*x5) * (x1^4*x3)^2 (w is inadmissible of degree 5)
  200. x1*x2^10*x3*x4*x5  -->  (x1*x3*x4*x5) * (x2^5)^2 (w is inadmissible of degree 5)
  201. x1^3*x2^8*x3*x4*x5  -->  (x1*x3*x4*x5) * (x1*x2^4)^2 (w is inadmissible of degree 5)
  202. x1^5*x2^6*x3*x4*x5  -->  (x1*x3*x4*x5) * (x1^2*x2^3)^2 (w is inadmissible of degree 5)
  203. x1^7*x2^4*x3*x4*x5  -->  (x1*x3*x4*x5) * (x1^3*x2^2)^2 (w is inadmissible of degree 5)
  204. x1^9*x2^2*x3*x4*x5  -->  (x1*x3*x4*x5) * (x1^4*x2)^2 (w is inadmissible of degree 5)

--- MONOMIALS FILTERED BY KAMEKO-SUM THEOREM (TYPE II) ---
(None)

============================================================================

Found a total of 70 'TRULY INTERESTING' inadmissible monomials remaining.

* Group with weight vector w = (2, 2, 2) (13 elements):
  1. x1*x2^2*x3^2*x4^4*x5^5               2. x1*x2^2*x3^2*x4^5*x5^4             
  3. x1*x2^2*x3^4*x4^2*x5^5               4. x1*x2^2*x3^4*x4^6*x5               
  5. x1*x2^2*x3^6*x4*x5^4                 6. x1*x2^2*x3^6*x4^4*x5               
  7. x1*x2^6*x3^2*x4*x5^4                 8. x1*x2^6*x3^2*x4^4*x5               
  9. x1^3*x2^2*x3*x4^4*x5^4               10. x1^3*x2^2*x3^4*x4*x5^4            
  11. x1^3*x2^2*x3^4*x4^4*x5              12. x1^3*x2^4*x3^2*x4*x5^4            
  13. x1^3*x2^4*x3^2*x4^4*x5            

* Group with weight vector w = (2, 4, 1) (33 elements):
  14. x1*x2^2*x3^2*x4^6*x5^3              15. x1*x2^2*x3^6*x4^2*x5^3            
  16. x1*x2^2*x3^6*x4^3*x5^2              17. x1*x2^6*x3^2*x4^2*x5^3            
  18. x1*x2^6*x3^2*x4^3*x5^2              19. x1*x2^6*x3^3*x4^2*x5^2            
  20. x1^3*x2^2*x3*x4^2*x5^6              21. x1^3*x2^2*x3*x4^6*x5^2            
  22. x1^3*x2^2*x3^2*x4*x5^6              23. x1^3*x2^2*x3^2*x4^2*x5^5          
  24. x1^3*x2^2*x3^2*x4^3*x5^4            25. x1^3*x2^2*x3^2*x4^4*x5^3          
  26. x1^3*x2^2*x3^2*x4^5*x5^2            27. x1^3*x2^2*x3^2*x4^6*x5            
  28. x1^3*x2^2*x3^3*x4^2*x5^4            29. x1^3*x2^2*x3^3*x4^4*x5^2          
  30. x1^3*x2^2*x3^4*x4^2*x5^3            31. x1^3*x2^2*x3^4*x4^3*x5^2          
  32. x1^3*x2^2*x3^5*x4^2*x5^2            33. x1^3*x2^2*x3^6*x4*x5^2            
  34. x1^3*x2^2*x3^6*x4^2*x5              35. x1^3*x2^3*x3^2*x4^2*x5^4          
  36. x1^3*x2^3*x3^2*x4^4*x5^2            37. x1^3*x2^3*x3^4*x4^2*x5^2          
  38. x1^3*x2^4*x3^2*x4^2*x5^3            39. x1^3*x2^4*x3^2*x4^3*x5^2          
  40. x1^3*x2^4*x3^3*x4^2*x5^2            41. x1^3*x2^6*x3*x4^2*x5^2            
  42. x1^3*x2^6*x3^2*x4*x5^2              43. x1^3*x2^6*x3^2*x4^2*x5            
  44. x1^7*x2^2*x3*x4^2*x5^2              45. x1^7*x2^2*x3^2*x4*x5^2            
  46. x1^7*x2^2*x3^2*x4^2*x5            

* Group with weight vector w = (4, 3, 1) (20 elements):
  47. x1^3*x2^2*x3*x4*x5^7                48. x1^3*x2^2*x3*x4^3*x5^5            
  49. x1^3*x2^2*x3*x4^5*x5^3              50. x1^3*x2^2*x3*x4^7*x5              
  51. x1^3*x2^2*x3^3*x4*x5^5              52. x1^3*x2^2*x3^3*x4^5*x5            
  53. x1^3*x2^2*x3^5*x4*x5^3              54. x1^3*x2^2*x3^5*x4^3*x5            
  55. x1^3*x2^2*x3^7*x4*x5                56. x1^3*x2^3*x3^2*x4*x5^5            
  57. x1^3*x2^3*x3^2*x4^5*x5              58. x1^3*x2^3*x3^6*x4*x5              
  59. x1^3*x2^6*x3*x4*x5^3                60. x1^3*x2^6*x3*x4^3*x5              
  61. x1^3*x2^6*x3^3*x4*x5                62. x1^3*x2^7*x3^2*x4*x5              
  63. x1^7*x2^2*x3*x4*x5^3                64. x1^7*x2^2*x3*x4^3*x5              
  65. x1^7*x2^2*x3^3*x4*x5                66. x1^7*x2^3*x3^2*x4*x5              

* Group with weight vector w = (4, 5) (4 elements):
  67. x1^3*x2^2*x3^3*x4^3*x5^3            68. x1^3*x2^3*x3^2*x4^3*x5^3          
  69. x1^3*x2^3*x3^3*x4^2*x5^3            70. x1^3*x2^3*x3^3*x4^3*x5^2          

Total execution time: 5.77 seconds
============================================================================

\end{lstlisting}

\medskip

As an application, we study the \textit{Singer algebraic transfer}, which is a homomorphism of the form
\[
\mathrm{Tr}_k: \left[(Q\mathcal{P}_k)^{GL_k}_d\right]^* \longrightarrow \mathrm{Ext}_{\mathcal{A}}^{k,\,k+d}(\mathbb{F}_2, \mathbb{F}_2),
\]
from the dual of the \(GL_k\)-invariant subspace \((Q\mathcal{P}_k)^{GL_k}_d\) to the mod-2 cohomology of the Steenrod algebra. This transfer plays a crucial role in uncovering the structure of the mysterious Ext groups. To investigate \(\mathrm{Tr}_k\), one can explicitly determine both the domain and codomain of the transfer in generic degrees \(d\) satisfying \(\mu(d) < k\).

\medskip

As is well known, the most natural and straightforward approach to computing the domain of \(\mathrm{Tr}_k\) is based on the admissible basis of \((Q\mathcal{P}_k)_d\). 

\medskip

Thus, by carefully and meticulously constructing an algorithm in \textsc{SageMath}, we are able to effectively obtain the basis of both the \(\Sigma_k\)-invariant and the \(GL_k\)-invariant subspaces of \( (Q\mathcal{P}_k)_d \). 
\noindent

\noindent
As an illustration, consider the case \( k = 3 \) and \( d = 31 \). When executing our algorithm in \textsc{SageMath}, we obtain consistent results through two complementary approaches: \textit{one relying on global invariant computation, and the other employing a divide-and-conquer strategy}.

\medskip

$\bullet$ \textbf{\underline{Method 1: Global invariant computation}}

\medskip

\begin{lstlisting}
============================================================================
STARTING INVARIANT SUBSPACE COMPUTATION (Global Method)
Case: k = 3, d = 31
============================================================================

--> STEP 1: Computing the full admissible basis for (QP_{k})_{d}...
    ...Found a global basis with 14 monomials.

============================================================================
ADMISSIBLE BASIS SUMMARY BY WEIGHT
============================================================================
Total dimension of (Q P_{3})_{31}: 14
- dim Q(P_{3})_{31}(w=[1, 1, 1, 1, 1]) = 7
- dim Q(P_{3})_{31}(w=[3, 2, 2, 2]) = 7

--> STEP 2: Building action matrices on the full basis...
    ...Action matrices for rho_1, ..., rho_4 computed.

--> STEP 3: Finding invariants on the entire space...

============================================================================
FINAL S_3 INVARIANTS for (QP_3)_31
============================================================================
Dimension of (QP_{3})_{31}^S_3: 6
  Sigma_3 Invariant 1 = [x1^31] + [x2^31] + [x3^31]
  Sigma_3 Invariant 2 = [x1*x2^30] + [x1*x3^30] + [x2*x3^30]
  Sigma_3 Invariant 3 = [x1*x2^2*x3^28]
  Sigma_3 Invariant 4 = [x1*x2^15*x3^15] + [x1^15*x2*x3^15] 
									+ [x1^15*x2^15*x3]
  Sigma_3 Invariant 5 = [x1^15*x2^3*x3^13] + [x1^3*x2^13*x3^15] 
									+ [x1^3*x2^15*x3^13]
  Sigma_3 Invariant 6 = [x1^7*x2^11*x3^13]

============================================================================
FINAL GL_3 INVARIANTS for (QP_3)_31
============================================================================
Dimension of (QP_{3})_{31}^GL_3: 2
  GL_3 Invariant 1 = [x1*x2^15*x3^15] + [x1*x2^30] + [x1*x3^30] 
								+ [x1^15*x2*x3^15] + [x1^15*x2^15*x3]
								+ [x1^15*x2^3*x3^13] + [x1^3*x2^13*x3^15] 
								+ [x1^3*x2^15*x3^13] + [x1^31]
								+ [x1^7*x2^11*x3^13] + [x2*x3^30] 
								+ [x2^31] + [x3^31]
  GL_3 Invariant 2 = [x1*x2^15*x3^15] + [x1*x2^2*x3^28] 
								+ [x1^15*x2*x3^15] + [x1^15*x2^15*x3]
								+ [x1^15*x2^3*x3^13] + [x1^3*x2^13*x3^15] 
								+ [x1^3*x2^15*x3^13] + [x1^7*x2^11*x3^13]

============================================================================
ENTIRE COMPUTATION PROCESS COMPLETED
Total execution time: 16.49 seconds
============================================================================
\end{lstlisting}

\medskip

$\bullet$ \textbf{\underline{Method 2: Divide-and-Conquer Strategy (Connected Component Decomposition Method)}}

\begin{lstlisting}
============================================================================
STARTING INVARIANT SUBSPACE COMPUTATION (Connected Component Decomposition Method)
Case: k = 3, d = 31
============================================================================
--> STEP 1: Computing global admissible basis for (QP_3)_31...
    ...Found a global basis with 14 monomials.
--> STEP 2: Partitioning basis by weight vector...
    ...Partitioned into 2 weight subspaces.
---------------------------------------------------------------------------
ANALYZING WEIGHT SUBSPACE: w = [1, 1, 1, 1, 1] (dim = 7)
--> STEP 3: Partitioning this subspace into Sigma_3-components...
    ...Discovered 3 Sigma_3-component(s) with dimensions: [3, 3, 1]
--> STEP 4: Finding Sigma_3-invariants for each component...

  --- component 1 (dim=3) ---
    [\Sigma_3(x3^31)]_(1,1,1,1,1) = Span{
        [x3^31]_(1,1,1,1,1)
        [x2^31]_(1,1,1,1,1)
        [x1^31]_(1,1,1,1,1)
    }

    Found 1 Sigma_3-invariant(s) in this component.
    Invariant q_[1, 1, 1, 1, 1]_1 = [x1^31]_{(1,1,1,1,1)} 
												+ [x2^31]_{(1,1,1,1,1)} 
												+ [x3^31]_{(1,1,1,1,1)}




  --- component 2 (dim=3) ---
    [\Sigma_3(x2*x3^30)]_(1,1,1,1,1) = Span{
        [x2*x3^30]_(1,1,1,1,1)
        [x1*x3^30]_(1,1,1,1,1)
        [x1*x2^30]_(1,1,1,1,1)
    }

    Found 1 Sigma_3-invariant(s) in this component.
    Invariant q_[1, 1, 1, 1, 1]_2 = [x1*x2^30]_{(1,1,1,1,1)} 
												+ [x1*x3^30]_{(1,1,1,1,1)} 
												+ [x2*x3^30]_{(1,1,1,1,1)}

  --- component 3 (dim=1) ---
    [\Sigma_3(x1*x2^2*x3^28)]_(1,1,1,1,1) = Span{
        [x1*x2^2*x3^28]_(1,1,1,1,1)
    }

    Found 1 Sigma_3-invariant(s) in this component.
    Invariant q_[1, 1, 1, 1, 1]_3 = [x1*x2^2*x3^28]_{(1,1,1,1,1)}

---------------------------------------------------------------------------
ANALYZING WEIGHT SUBSPACE: w = [3, 2, 2, 2] (dim = 7)
--> STEP 3: Partitioning this subspace into Sigma_3-components...
    ...Discovered 3 Sigma_3-component(s) with dimensions: [3, 3, 1]
--> STEP 4: Finding Sigma_3-invariants for each component...

  --- component 1 (dim=3) ---
    [\Sigma_3(x1*x2^15*x3^15)]_(3,2,2,2) = Span{
        [x1*x2^15*x3^15]_(3,2,2,2)
        [x1^15*x2*x3^15]_(3,2,2,2)
        [x1^15*x2^15*x3]_(3,2,2,2)
    }

    Found 1 Sigma_3-invariant(s) in this component.
    Invariant q_[3, 2, 2, 2]_1 = [x1*x2^15*x3^15]_{(3,2,2,2)}
											+ [x1^15*x2*x3^15]_{(3,2,2,2)}
											+ [x1^15*x2^15*x3]_{(3,2,2,2)}

  --- component 2 (dim=3) ---
    [\Sigma_3(x1^3*x2^13*x3^15)]_(3,2,2,2) = Span{
        [x1^3*x2^13*x3^15]_(3,2,2,2)
        [x1^3*x2^15*x3^13]_(3,2,2,2)
        [x1^15*x2^3*x3^13]_(3,2,2,2)
    }

    Found 1 Sigma_3-invariant(s) in this component.
    Invariant q_[3, 2, 2, 2]_2 = [x1^15*x2^3*x3^13]_{(3,2,2,2)} 
											+ [x1^3*x2^13*x3^15]_{(3,2,2,2)}
											+ [x1^3*x2^15*x3^13]_{(3,2,2,2)}

  --- component 3 (dim=1) ---
    [\Sigma_3(x1^7*x2^11*x3^13)]_(3,2,2,2) = Span{
        [x1^7*x2^11*x3^13]_(3,2,2,2)
    }

    Found 1 Sigma_3-invariant(s) in this component.
    Invariant q_[3, 2, 2, 2]_3 = [x1^7*x2^11*x3^13]_{(3,2,2,2)}


============================================================================
AGGREGATED Sigma_3 INVARIANTS
============================================================================
Total dimension of (QP_{3})_{31}^Sigma_3: 6
  Total Sigma_3 Inv 1 = [x1^31] + [x2^31] + [x3^31]
  Total Sigma_3 Inv 2 = [x1*x2^30] + [x1*x3^30] + [x2*x3^30]
  Total Sigma_3 Inv 3 = [x1*x2^2*x3^28]
  Total Sigma_3 Inv 4 = [x1*x2^15*x3^15] + [x1^15*x2*x3^15] 
											+ [x1^15*x2^15*x3]
  Total Sigma_3 Inv 5 = [x1^15*x2^3*x3^13] + [x1^3*x2^13*x3^15] 
											+ [x1^3*x2^15*x3^13]
  Total Sigma_3 Inv 6 = [x1^7*x2^11*x3^13]

--> STEP 5: Finding GL_3 invariants from the aggregated Sigma_3 basis...
============================================================================
FINAL GL_3 INVARIANTS for (QP_3)_31
============================================================================
Dimension of (QP_{3})_{31}^GL_3: 2
  GL_3 Invariant 1 = [x1*x2^15*x3^15] + [x1*x2^30] + [x1*x3^30] 
								+ [x1^15*x2*x3^15] + [x1^15*x2^15*x3]
								+ [x1^15*x2^3*x3^13] + [x1^3*x2^13*x3^15] 
								+ [x1^3*x2^15*x3^13] + [x1^31]
								+ [x1^7*x2^11*x3^13] + [x2*x3^30] 
								+ [x2^31] + [x3^31]
								
  GL_3 Invariant 2 = [x1*x2^15*x3^15] + [x1*x2^2*x3^28] 
								+ [x1^15*x2*x3^15] + [x1^15*x2^15*x3]
								+ [x1^15*x2^3*x3^13] + [x1^3*x2^13*x3^15] 
								+ [x1^3*x2^15*x3^13] + [x1^7*x2^11*x3^13]

============================================================================
ENTIRE COMPUTATION PROCESS COMPLETED
Total execution time: 16.55 seconds
============================================================================
\end{lstlisting}

\noindent

\medskip

Thus, the output of our algorithm, as computed by both methods above, shows that the invariant space \( (Q\mathcal{P}_3)^{GL_3}_{31} \) has dimension 2, and also provides an explicit basis for it. This result is consistent with the fully manual computations presented in~\cite{Sum2}.

\medskip

Another illustration for the output below of our algorithm for the case $k = 4$ va $d = 23.$ This obtained result confirms the hand calculations in our previous work in \cite{Phuc2}.

\medskip

\newpage
\begin{lstlisting}
============================================================================
STARTING INVARIANT SUBSPACE COMPUTATION
Case: k = 4, d = 23
============================================================================
--> STEP 1: Computing global admissible basis for (QP_4)_23...
    ... Found a global basis with 155 monomials.
--> STEP 2: Partitioning basis by weight vector...
    ... Partitioned into 1 weight subspaces.
---------------------------------------------------------------------------
ANALYZING WEIGHT SUBSPACE: w = [3, 2, 2, 1] (dim = 155)
--> STEP 3: Finding indecomposable Sigma_4-submodules...
    ... Discovered 5 indecomposable submodule(s).
--> STEP 4: Finding Sigma_4-invariants for each submodule...
 --- Module 1 (dim=24) ---
    [\Sigma_4(x2*x3^7*x4^15)]_{(3,2,2,1)} = Span{
        [x1*x2^15*x3^7]_{(3,2,2,1)}
        [x1*x2^15*x4^7]_{(3,2,2,1)}
        [x1*x2^7*x3^15]_{(3,2,2,1)}
        [x1*x2^7*x4^15]_{(3,2,2,1)}
        [x1*x3^15*x4^7]_{(3,2,2,1)}
        [x1*x3^7*x4^15]_{(3,2,2,1)}
        [x1^15*x2*x3^7]_{(3,2,2,1)}
        [x1^15*x2*x4^7]_{(3,2,2,1)}
        [x1^15*x2^7*x3]_{(3,2,2,1)}
        [x1^15*x2^7*x4]_{(3,2,2,1)}
        [x1^15*x3*x4^7]_{(3,2,2,1)}
        [x1^15*x3^7*x4]_{(3,2,2,1)}
        [x1^7*x2*x3^15]_{(3,2,2,1)}
        [x1^7*x2*x4^15]_{(3,2,2,1)}
        [x1^7*x2^15*x3]_{(3,2,2,1)}
        [x1^7*x2^15*x4]_{(3,2,2,1)}
        [x1^7*x3*x4^15]_{(3,2,2,1)}
        [x1^7*x3^15*x4]_{(3,2,2,1)}
        [x2*x3^15*x4^7]_{(3,2,2,1)}
        [x2*x3^7*x4^15]_{(3,2,2,1)}
        [x2^15*x3*x4^7]_{(3,2,2,1)}
        [x2^15*x3^7*x4]_{(3,2,2,1)}
        [x2^7*x3*x4^15]_{(3,2,2,1)}
        [x2^7*x3^15*x4]_{(3,2,2,1)}
    }
    Found 1 Sigma_4-invariant(s) in this submodule.
    Invariant 1-1 = [x1*x2^15*x3^7 + x1*x2^15*x4^7 + x1*x2^7*x3^15 
					 + x1*x2^7*x4^15 + x1*x3^15*x4^7 + x1*x3^7*x4^15
                     + x1^15*x2*x3^7 + x1^15*x2*x4^7 + x1^15*x2^7*x3
                     + x1^15*x2^7*x4 + x1^15*x3*x4^7 + x1^15*x3^7*x4
                     + x1^7*x2*x3^15 + x1^7*x2*x4^15 + x1^7*x2^15*x3 
                     + x1^7*x2^15*x4 + x1^7*x3*x4^15 + x1^7*x3^15*x4 
                     + x2*x3^15*x4^7 + x2*x3^7*x4^15 + x2^15*x3*x4^7 
					 + x2^15*x3^7*x4 + x2^7*x3*x4^15 + x2^7*x3^15*x4]_{(3,2,2,1)}

  --- Module 2 (dim=12) ---
    [\Sigma_4(x2^3*x3^5*x4^15)]_{(3,2,2,1)} = Span{
        [x1^15*x2^3*x3^5]_{(3,2,2,1)}
        [x1^15*x2^3*x4^5]_{(3,2,2,1)}
        [x1^15*x3^3*x4^5]_{(3,2,2,1)}
        [x1^3*x2^15*x3^5]_{(3,2,2,1)}
        [x1^3*x2^15*x4^5]_{(3,2,2,1)}
        [x1^3*x2^5*x3^15]_{(3,2,2,1)}
        [x1^3*x2^5*x4^15]_{(3,2,2,1)}
        [x1^3*x3^15*x4^5]_{(3,2,2,1)}
        [x1^3*x3^5*x4^15]_{(3,2,2,1)}
        [x2^15*x3^3*x4^5]_{(3,2,2,1)}
        [x2^3*x3^15*x4^5]_{(3,2,2,1)}
        [x2^3*x3^5*x4^15]_{(3,2,2,1)}
    }
    Found 1 Sigma_4-invariant(s) in this submodule.
    Invariant 2-1 = [x1^15*x2^3*x3^5 + x1^15*x2^3*x4^5 + x1^15*x3^3*x4^5
         + x1^3*x2^15*x3^5 + x1^3*x2^15*x4^5 + x1^3*x2^5*x3^15
         + x1^3*x2^5*x4^15 + x1^3*x3^15*x4^5 + x1^3*x3^5*x4^15
         + x2^15*x3^3*x4^5 + x2^3*x3^15*x4^5 + x2^3*x3^5*x4^15]_{(3,2,2,1)}

  --- Module 3 (dim=20) ---
    [\Sigma_4(x2^3*x3^7*x4^13)]_{(3,2,2,1)} = Span{
        [x1^3*x2^13*x3^7]_{(3,2,2,1)}
        [x1^3*x2^13*x4^7]_{(3,2,2,1)}
        [x1^3*x2^7*x3^13]_{(3,2,2,1)}
        [x1^3*x2^7*x4^13]_{(3,2,2,1)}
        [x1^3*x3^13*x4^7]_{(3,2,2,1)}
        [x1^3*x3^7*x4^13]_{(3,2,2,1)}
        [x1^7*x2^11*x3^5]_{(3,2,2,1)}
        [x1^7*x2^11*x4^5]_{(3,2,2,1)}
        [x1^7*x2^3*x3^13]_{(3,2,2,1)}
        [x1^7*x2^3*x4^13]_{(3,2,2,1)}
        [x1^7*x2^7*x3^9]_{(3,2,2,1)}
        [x1^7*x2^7*x4^9]_{(3,2,2,1)}
        [x1^7*x3^11*x4^5]_{(3,2,2,1)}
        [x1^7*x3^3*x4^13]_{(3,2,2,1)}
        [x1^7*x3^7*x4^9]_{(3,2,2,1)}
        [x2^3*x3^13*x4^7]_{(3,2,2,1)}
        [x2^3*x3^7*x4^13]_{(3,2,2,1)}
        [x2^7*x3^11*x4^5]_{(3,2,2,1)}
        [x2^7*x3^3*x4^13]_{(3,2,2,1)}
        [x2^7*x3^7*x4^9]_{(3,2,2,1)}
    }



    Found 1 Sigma_4-invariant(s) in this submodule.
    Invariant 3-1 = [x1^3*x2^13*x3^7 + x1^3*x2^13*x4^7 + x1^3*x3^13*x4^7
                     + x1^7*x2^11*x3^5 + x1^7*x2^11*x4^5 + x1^7*x2^3*x3^13
                     + x1^7*x2^3*x4^13 + x1^7*x2^7*x3^9 + x1^7*x2^7*x4^9
                     + x1^7*x3^11*x4^5 + x1^7*x3^3*x4^13 + x1^7*x3^7*x4^9
                     + x2^3*x3^13*x4^7 + x2^7*x3^11*x4^5 + x2^7*x3^3*x4^13
                     + x2^7*x3^7*x4^9]_{(3,2,2,1)}

  --- Module 4 (dim=24) ---
    [\Sigma_4(x1*x2*x3^6*x4^15)]_{(3,2,2,1)} = Span{
        [x1*x2*x3^15*x4^6]_{(3,2,2,1)}
        [x1*x2*x3^6*x4^15]_{(3,2,2,1)}
        [x1*x2^15*x3*x4^6]_{(3,2,2,1)}
        [x1*x2^15*x3^2*x4^5]_{(3,2,2,1)}
        [x1*x2^15*x3^3*x4^4]_{(3,2,2,1)}
        [x1*x2^15*x3^6*x4]_{(3,2,2,1)}
        [x1*x2^2*x3^15*x4^5]_{(3,2,2,1)}
        [x1*x2^2*x3^5*x4^15]_{(3,2,2,1)}
        [x1*x2^3*x3^15*x4^4]_{(3,2,2,1)}
        [x1*x2^3*x3^4*x4^15]_{(3,2,2,1)}
        [x1*x2^6*x3*x4^15]_{(3,2,2,1)}
        [x1*x2^6*x3^15*x4]_{(3,2,2,1)}
        [x1^15*x2*x3*x4^6]_{(3,2,2,1)}
        [x1^15*x2*x3^2*x4^5]_{(3,2,2,1)}
        [x1^15*x2*x3^3*x4^4]_{(3,2,2,1)}
        [x1^15*x2*x3^6*x4]_{(3,2,2,1)}
        [x1^15*x2^3*x3*x4^4]_{(3,2,2,1)}
        [x1^15*x2^3*x3^4*x4]_{(3,2,2,1)}
        [x1^3*x2*x3^15*x4^4]_{(3,2,2,1)}
        [x1^3*x2*x3^4*x4^15]_{(3,2,2,1)}
        [x1^3*x2^15*x3*x4^4]_{(3,2,2,1)}
        [x1^3*x2^15*x3^4*x4]_{(3,2,2,1)}
        [x1^3*x2^4*x3*x4^15]_{(3,2,2,1)}
        [x1^3*x2^4*x3^15*x4]_{(3,2,2,1)}
    }

    Found 1 Sigma_4-invariant(s) in this submodule.
    Invariant 4-1 = [x1*x2*x3^15*x4^6 + x1*x2*x3^6*x4^15 + x1*x2^15*x3*x4^6
             + x1*x2^15*x3^6*x4 + x1*x2^6*x3*x4^15 + x1*x2^6*x3^15*x4
             + x1^15*x2*x3*x4^6 + x1^15*x2*x3^6*x4 + x1^15*x2^3*x3*x4^4
             + x1^15*x2^3*x3^4*x4 + x1^3*x2*x3^15*x4^4 + x1^3*x2*x3^4*x4^15
             + x1^3*x2^15*x3*x4^4 + x1^3*x2^15*x3^4*x4 + x1^3*x2^4*x3*x4^15
             + x1^3*x2^4*x3^15*x4]_{(3,2,2,1)}

  --- Module 5 (dim=75) ---
    [\Sigma_4(x1*x2*x3^7*x4^14)]_{(3,2,2,1)} = Span{
        [x1*x2*x3^14*x4^7]_{(3,2,2,1)}
        [x1*x2*x3^7*x4^14]_{(3,2,2,1)}
        [x1*x2^14*x3*x4^7]_{(3,2,2,1)}
        [x1*x2^14*x3^3*x4^5]_{(3,2,2,1)}
        [x1*x2^14*x3^7*x4]_{(3,2,2,1)}
        [x1*x2^2*x3^13*x4^7]_{(3,2,2,1)}
        [x1*x2^2*x3^7*x4^13]_{(3,2,2,1)}
        [x1*x2^3*x3^12*x4^7]_{(3,2,2,1)}
        [x1*x2^3*x3^13*x4^6]_{(3,2,2,1)}
        [x1*x2^3*x3^14*x4^5]_{(3,2,2,1)}
        [x1*x2^3*x3^5*x4^14]_{(3,2,2,1)}
        [x1*x2^3*x3^6*x4^13]_{(3,2,2,1)}
        [x1*x2^3*x3^7*x4^12]_{(3,2,2,1)}
        [x1*x2^6*x3^11*x4^5]_{(3,2,2,1)}
        [x1*x2^6*x3^3*x4^13]_{(3,2,2,1)}
        [x1*x2^6*x3^7*x4^9]_{(3,2,2,1)}
        [x1*x2^7*x3*x4^14]_{(3,2,2,1)}
        [x1*x2^7*x3^10*x4^5]_{(3,2,2,1)}
        [x1*x2^7*x3^11*x4^4]_{(3,2,2,1)}
        [x1*x2^7*x3^14*x4]_{(3,2,2,1)}
        [x1*x2^7*x3^2*x4^13]_{(3,2,2,1)}
        [x1*x2^7*x3^3*x4^12]_{(3,2,2,1)}
        [x1*x2^7*x3^6*x4^9]_{(3,2,2,1)}
        [x1*x2^7*x3^7*x4^8]_{(3,2,2,1)}
        [x1^3*x2*x3^12*x4^7]_{(3,2,2,1)}
        [x1^3*x2*x3^13*x4^6]_{(3,2,2,1)}
        [x1^3*x2*x3^14*x4^5]_{(3,2,2,1)}
        [x1^3*x2*x3^5*x4^14]_{(3,2,2,1)}
        [x1^3*x2*x3^6*x4^13]_{(3,2,2,1)}
        [x1^3*x2*x3^7*x4^12]_{(3,2,2,1)}
        [x1^3*x2^13*x3*x4^6]_{(3,2,2,1)}
        [x1^3*x2^13*x3^2*x4^5]_{(3,2,2,1)}
        [x1^3*x2^13*x3^3*x4^4]_{(3,2,2,1)}
        [x1^3*x2^13*x3^6*x4]_{(3,2,2,1)}
        [x1^3*x2^3*x3^12*x4^5]_{(3,2,2,1)}
        [x1^3*x2^3*x3^13*x4^4]_{(3,2,2,1)}
        [x1^3*x2^3*x3^4*x4^13]_{(3,2,2,1)}
        [x1^3*x2^3*x3^5*x4^12]_{(3,2,2,1)}
        [x1^3*x2^4*x3^11*x4^5]_{(3,2,2,1)}
        [x1^3*x2^4*x3^3*x4^13]_{(3,2,2,1)}
        [x1^3*x2^4*x3^7*x4^9]_{(3,2,2,1)}
        [x1^3*x2^5*x3*x4^14]_{(3,2,2,1)}
        [x1^3*x2^5*x3^10*x4^5]_{(3,2,2,1)}
        [x1^3*x2^5*x3^11*x4^4]_{(3,2,2,1)}
        [x1^3*x2^5*x3^14*x4]_{(3,2,2,1)}
        [x1^3*x2^5*x3^2*x4^13]_{(3,2,2,1)}
        [x1^3*x2^5*x3^3*x4^12]_{(3,2,2,1)}
        [x1^3*x2^5*x3^6*x4^9]_{(3,2,2,1)}
        [x1^3*x2^5*x3^7*x4^8]_{(3,2,2,1)}
        [x1^3*x2^7*x3*x4^12]_{(3,2,2,1)}
        [x1^3*x2^7*x3^12*x4]_{(3,2,2,1)}
        [x1^3*x2^7*x3^4*x4^9]_{(3,2,2,1)}
        [x1^3*x2^7*x3^5*x4^8]_{(3,2,2,1)}
        [x1^3*x2^7*x3^8*x4^5]_{(3,2,2,1)}
        [x1^3*x2^7*x3^9*x4^4]_{(3,2,2,1)}
        [x1^7*x2*x3*x4^14]_{(3,2,2,1)}
        [x1^7*x2*x3^10*x4^5]_{(3,2,2,1)}
        [x1^7*x2*x3^11*x4^4]_{(3,2,2,1)}
        [x1^7*x2*x3^14*x4]_{(3,2,2,1)}
        [x1^7*x2*x3^2*x4^13]_{(3,2,2,1)}
        [x1^7*x2*x3^3*x4^12]_{(3,2,2,1)}
        [x1^7*x2*x3^6*x4^9]_{(3,2,2,1)}
        [x1^7*x2*x3^7*x4^8]_{(3,2,2,1)}
        [x1^7*x2^11*x3*x4^4]_{(3,2,2,1)}
        [x1^7*x2^11*x3^4*x4]_{(3,2,2,1)}
        [x1^7*x2^3*x3*x4^12]_{(3,2,2,1)}
        [x1^7*x2^3*x3^12*x4]_{(3,2,2,1)}
        [x1^7*x2^3*x3^4*x4^9]_{(3,2,2,1)}
        [x1^7*x2^3*x3^5*x4^8]_{(3,2,2,1)}
        [x1^7*x2^3*x3^8*x4^5]_{(3,2,2,1)}
        [x1^7*x2^3*x3^9*x4^4]_{(3,2,2,1)}
        [x1^7*x2^7*x3*x4^8]_{(3,2,2,1)}
        [x1^7*x2^7*x3^8*x4]_{(3,2,2,1)}
        [x1^7*x2^9*x3^2*x4^5]_{(3,2,2,1)}
        [x1^7*x2^9*x3^3*x4^4]_{(3,2,2,1)}
    }

    Found 4 Sigma_4-invariant(s) in this submodule.
    Invariant 5-1 = [x1*x2*x3^7*x4^14 + x1*x2^3*x3^13*x4^6 
		+ x1*x2^3*x3^14*x4^5
        + x1*x2^6*x3^11*x4^5 + x1*x2^7*x3*x4^14 + x1*x2^7*x3^14*x4
        + x1*x2^7*x3^6*x4^9 + x1*x2^7*x3^7*x4^8 + x1^3*x2*x3^13*x4^6
        + x1^3*x2^13*x3*x4^6 + x1^3*x2^13*x3^6*x4 + x1^3*x2^3*x3^12*x4^5
        + x1^3*x2^4*x3^11*x4^5 + x1^3*x2^5*x3^6*x4^9 + x1^3*x2^5*x3^7*x4^8
        + x1^3*x2^7*x3^4*x4^9 + x1^3*x2^7*x3^9*x4^4 + x1^7*x2*x3*x4^14
        + x1^7*x2*x3^14*x4 + x1^7*x2*x3^6*x4^9 + x1^7*x2*x3^7*x4^8
        + x1^7*x2^11*x3*x4^4 + x1^7*x2^11*x3^4*x4 + x1^7*x2^3*x3*x4^12
        + x1^7*x2^3*x3^12*x4 + x1^7*x2^3*x3^4*x4^9 + x1^7*x2^3*x3^9*x4^4
        + x1^7*x2^7*x3*x4^8 + x1^7*x2^7*x3^8*x4]_{(3,2,2,1)}

    Invariant 5-2 = [x1*x2^3*x3^14*x4^5 + x1*x2^3*x3^5*x4^14 
        + x1*x2^7*x3^10*x4^5
        + x1*x2^7*x3^3*x4^12 + x1^3*x2*x3^14*x4^5 + x1^3*x2*x3^5*x4^14
        + x1^3*x2^13*x3^2*x4^5 + x1^3*x2^13*x3^3*x4^4
        + x1^3*x2^3*x3^12*x4^5 + x1^3*x2^3*x3^5*x4^12 + x1^3*x2^5*x3*x4^14
        + x1^3*x2^5*x3^14*x4 + x1^3*x2^7*x3*x4^12 + x1^3*x2^7*x3^12*x4
        + x1^3*x2^7*x3^8*x4^5 + x1^3*x2^7*x3^9*x4^4 + x1^7*x2*x3^10*x4^5
        + x1^7*x2*x3^3*x4^12 + x1^7*x2^3*x3*x4^12 + x1^7*x2^3*x3^12*x4
        + x1^7*x2^7*x3*x4^8 + x1^7*x2^7*x3^8*x4 + x1^7*x2^9*x3^2*x4^5
        + x1^7*x2^9*x3^3*x4^4]_{(3,2,2,1)}

    Invariant 5-3 = [x1*x2^3*x3^14*x4^5 + x1*x2^3*x3^6*x4^13 
      + x1*x2^6*x3^3*x4^13
      + x1*x2^6*x3^7*x4^9 + x1*x2^7*x3^10*x4^5 + x1*x2^7*x3^3*x4^12
      + x1*x2^7*x3^6*x4^9 + x1*x2^7*x3^7*x4^8 + x1^3*x2*x3^14*x4^5
      + x1^3*x2*x3^7*x4^12 + x1^3*x2^13*x3^2*x4^5 + x1^3*x2^13*x3^3*x4^4
      + x1^3*x2^3*x3^12*x4^5 + x1^3*x2^3*x3^13*x4^4
      + x1^3*x2^3*x3^4*x4^13 + x1^3*x2^4*x3^3*x4^13 + x1^3*x2^4*x3^7*x4^9
      + x1^3*x2^5*x3^6*x4^9 + x1^3*x2^5*x3^7*x4^8 + x1^3*x2^7*x3^4*x4^9
      + x1^3*x2^7*x3^8*x4^5 + x1^7*x2*x3^10*x4^5 + x1^7*x2*x3^3*x4^12
      + x1^7*x2*x3^6*x4^9 + x1^7*x2*x3^7*x4^8 + x1^7*x2^3*x3^4*x4^9
      + x1^7*x2^3*x3^9*x4^4 + x1^7*x2^9*x3^2*x4^5 + x1^7*x2^9*x3^3*x4^4]_{(3,2,2,1)}

    Invariant 5-4 = [x1*x2^3*x3^14*x4^5 + x1*x2^3*x3^7*x4^12
      + x1*x2^6*x3^11*x4^5
      + x1*x2^6*x3^7*x4^9 + x1^3*x2^3*x3^12*x4^5 + x1^3*x2^4*x3^11*x4^5
      + x1^3*x2^4*x3^7*x4^9 + x1^3*x2^5*x3^6*x4^9 + x1^3*x2^5*x3^7*x4^8]_{(3,2,2,1)}

============================================================================
AGGREGATED Sigma_4 INVARIANTS
============================================================================
Total dimension of (QP_{4})_{23}^Sigma_4: 8
   Total Sigma_4 Inv 1 = [x1*x2^15*x3^7 + x1*x2^15*x4^7 + x1*x2^7*x3^15 
					 + x1*x2^7*x4^15 + x1*x3^15*x4^7 + x1*x3^7*x4^15
                     + x1^15*x2*x3^7 + x1^15*x2*x4^7 + x1^15*x2^7*x3
                     + x1^15*x2^7*x4 + x1^15*x3*x4^7 + x1^15*x3^7*x4
                     + x1^7*x2*x3^15 + x1^7*x2*x4^15 + x1^7*x2^15*x3 
                     + x1^7*x2^15*x4 + x1^7*x3*x4^15 + x1^7*x3^15*x4 
                     + x2*x3^15*x4^7 + x2*x3^7*x4^15 + x2^15*x3*x4^7 
					 + x2^15*x3^7*x4 + x2^7*x3*x4^15 + x2^7*x3^15*x4]

 Total Sigma_4 Inv 2 = [x1^15*x2^3*x3^5 + x1^15*x2^3*x4^5 + x1^15*x3^3*x4^5
         + x1^3*x2^15*x3^5 + x1^3*x2^15*x4^5 + x1^3*x2^5*x3^15
         + x1^3*x2^5*x4^15 + x1^3*x3^15*x4^5 + x1^3*x3^5*x4^15
         + x2^15*x3^3*x4^5 + x2^3*x3^15*x4^5 + x2^3*x3^5*x4^15]

 Total Sigma_4 Inv 3 = [x1^3*x2^13*x3^7 + x1^3*x2^13*x4^7 + x1^3*x3^13*x4^7
                     + x1^7*x2^11*x3^5 + x1^7*x2^11*x4^5 + x1^7*x2^3*x3^13
                     + x1^7*x2^3*x4^13 + x1^7*x2^7*x3^9 + x1^7*x2^7*x4^9
                     + x1^7*x3^11*x4^5 + x1^7*x3^3*x4^13 + x1^7*x3^7*x4^9
                     + x2^3*x3^13*x4^7 + x2^7*x3^11*x4^5 + x2^7*x3^3*x4^13
                     + x2^7*x3^7*x4^9]

 Total Sigma_4 Inv 4 = [x1*x2*x3^15*x4^6 + x1*x2*x3^6*x4^15 
			       + x1*x2^15*x3*x4^6
             + x1*x2^15*x3^6*x4 + x1*x2^6*x3*x4^15 + x1*x2^6*x3^15*x4
             + x1^15*x2*x3*x4^6 + x1^15*x2*x3^6*x4 + x1^15*x2^3*x3*x4^4
             + x1^15*x2^3*x3^4*x4 + x1^3*x2*x3^15*x4^4 + x1^3*x2*x3^4*x4^15
             + x1^3*x2^15*x3*x4^4 + x1^3*x2^15*x3^4*x4 + x1^3*x2^4*x3*x4^15
             + x1^3*x2^4*x3^15*x4]

 Total Sigma_4 Inv 5 = [x1*x2*x3^7*x4^14 + x1*x2^3*x3^13*x4^6 
		    + x1*x2^3*x3^14*x4^5
        + x1*x2^6*x3^11*x4^5 + x1*x2^7*x3*x4^14 + x1*x2^7*x3^14*x4
        + x1*x2^7*x3^6*x4^9 + x1*x2^7*x3^7*x4^8 + x1^3*x2*x3^13*x4^6
        + x1^3*x2^13*x3*x4^6 + x1^3*x2^13*x3^6*x4 + x1^3*x2^3*x3^12*x4^5
        + x1^3*x2^4*x3^11*x4^5 + x1^3*x2^5*x3^6*x4^9 + x1^3*x2^5*x3^7*x4^8
        + x1^3*x2^7*x3^4*x4^9 + x1^3*x2^7*x3^9*x4^4 + x1^7*x2*x3*x4^14
        + x1^7*x2*x3^14*x4 + x1^7*x2*x3^6*x4^9 + x1^7*x2*x3^7*x4^8
        + x1^7*x2^11*x3*x4^4 + x1^7*x2^11*x3^4*x4 + x1^7*x2^3*x3*x4^12
        + x1^7*x2^3*x3^12*x4 + x1^7*x2^3*x3^4*x4^9 + x1^7*x2^3*x3^9*x4^4
        + x1^7*x2^7*x3*x4^8 + x1^7*x2^7*x3^8*x4]

 Total Sigma_4 Inv 6 = [x1*x2^3*x3^14*x4^5 + x1*x2^3*x3^5*x4^14 
        + x1*x2^7*x3^10*x4^5
        + x1*x2^7*x3^3*x4^12 + x1^3*x2*x3^14*x4^5 + x1^3*x2*x3^5*x4^14
        + x1^3*x2^13*x3^2*x4^5 + x1^3*x2^13*x3^3*x4^4
        + x1^3*x2^3*x3^12*x4^5 + x1^3*x2^3*x3^5*x4^12 + x1^3*x2^5*x3*x4^14
        + x1^3*x2^5*x3^14*x4 + x1^3*x2^7*x3*x4^12 + x1^3*x2^7*x3^12*x4
        + x1^3*x2^7*x3^8*x4^5 + x1^3*x2^7*x3^9*x4^4 + x1^7*x2*x3^10*x4^5
        + x1^7*x2*x3^3*x4^12 + x1^7*x2^3*x3*x4^12 + x1^7*x2^3*x3^12*x4
        + x1^7*x2^7*x3*x4^8 + x1^7*x2^7*x3^8*x4 + x1^7*x2^9*x3^2*x4^5
        + x1^7*x2^9*x3^3*x4^4]

 Total Sigma_4 Inv 7 = [x1*x2^3*x3^14*x4^5 + x1*x2^3*x3^6*x4^13 
      + x1*x2^6*x3^3*x4^13
      + x1*x2^6*x3^7*x4^9 + x1*x2^7*x3^10*x4^5 + x1*x2^7*x3^3*x4^12
      + x1*x2^7*x3^6*x4^9 + x1*x2^7*x3^7*x4^8 + x1^3*x2*x3^14*x4^5
      + x1^3*x2*x3^7*x4^12 + x1^3*x2^13*x3^2*x4^5 + x1^3*x2^13*x3^3*x4^4
      + x1^3*x2^3*x3^12*x4^5 + x1^3*x2^3*x3^13*x4^4
      + x1^3*x2^3*x3^4*x4^13 + x1^3*x2^4*x3^3*x4^13 + x1^3*x2^4*x3^7*x4^9
      + x1^3*x2^5*x3^6*x4^9 + x1^3*x2^5*x3^7*x4^8 + x1^3*x2^7*x3^4*x4^9
      + x1^3*x2^7*x3^8*x4^5 + x1^7*x2*x3^10*x4^5 + x1^7*x2*x3^3*x4^12
      + x1^7*x2*x3^6*x4^9 + x1^7*x2*x3^7*x4^8 + x1^7*x2^3*x3^4*x4^9
      + x1^7*x2^3*x3^9*x4^4 + x1^7*x2^9*x3^2*x4^5 + x1^7*x2^9*x3^3*x4^4]

 Total Sigma_4 Inv 8 = [x1*x2^3*x3^14*x4^5 + x1*x2^3*x3^7*x4^12
 + x1*x2^6*x3^11*x4^5
 + x1*x2^6*x3^7*x4^9 + x1^3*x2^3*x3^12*x4^5 + x1^3*x2^4*x3^11*x4^5
 + x1^3*x2^4*x3^7*x4^9 + x1^3*x2^5*x3^6*x4^9 + x1^3*x2^5*x3^7*x4^8]

--> STEP 5: Finding GL_4 invariants from the aggregated Sigma_4 basis...

============================================================================
FINAL GL_4 INVARIANTS for (QP_4)_23
============================================================================
Dimension of (QP_{4})_{23}^GL_4: 1
  GL_4 Invariant 1 = [x1*x2*x3^15*x4^6 + x1*x2*x3^6*x4^15 + x1*x2*x3^7*x4^14
                      + x1*x2^15*x3*x4^6 + x1*x2^15*x3^6*x4 
					            + x1*x2^3*x3^13*x4^6 + x1*x2^3*x3^14*x4^5
                      + x1*x2^3*x3^5*x4^14 + x1*x2^3*x3^6*x4^13
                      + x1*x2^6*x3*x4^15 + x1*x2^6*x3^11*x4^5 
					            + x1*x2^6*x3^15*x4  + x1*x2^6*x3^3*x4^13 
                      + x1*x2^6*x3^7*x4^9 + x1*x2^7*x3*x4^14
                      + x1*x2^7*x3^14*x4 + x1^15*x2*x3*x4^6 
                      + x1^15*x2*x3^6*x4 + x1^15*x2^3*x3*x4^4 
                      + x1^15*x2^3*x3^4*x4 + x1^3*x2*x3^13*x4^6
                      + x1^3*x2*x3^15*x4^4 + x1^3*x2*x3^4*x4^15 
                      + x1^3*x2*x3^5*x4^14 + x1^3*x2*x3^7*x4^12
                      + x1^3*x2^13*x3*x4^6 + x1^3*x2^13*x3^6*x4
                      + x1^3*x2^15*x3*x4^4 + x1^3*x2^15*x3^4*x4 
                      + x1^3*x2^3*x3^12*x4^5 + x1^3*x2^3*x3^13*x4^4 
                      + x1^3*x2^3*x3^4*x4^13 + x1^3*x2^3*x3^5*x4^12
                      + x1^3*x2^4*x3*x4^15 + x1^3*x2^4*x3^11*x4^5
                      + x1^3*x2^4*x3^15*x4 + x1^3*x2^4*x3^3*x4^13 
                      + x1^3*x2^4*x3^7*x4^9 + x1^3*x2^5*x3*x4^14
                      + x1^3*x2^5*x3^14*x4 + x1^3*x2^7*x3*x4^12
                      + x1^3*x2^7*x3^12*x4 + x1^7*x2*x3*x4^14
                      + x1^7*x2*x3^14*x4 + x1^7*x2^11*x3*x4^4
                      + x1^7*x2^11*x3^4*x4]
============================================================================
ENTIRE COMPUTATION PROCESS COMPLETED
Total execution time: 481.87 seconds
============================================================================

\end{lstlisting}

\medskip

A general description of our algorithm, in connection with the illustrated output above, will be provided in Part II of this research project. That section will provide a comprehensive algorithm developed with precision and care, including the execution method that produces explicit and accurate results---dimension, basis, and complete list of inadmissible monomials---matching previously verified manual computations. The dimension and an explicit basis of the invariant space are also determined through carefully designed algorithms.

\medskip

As a final observation, Part II of our investigation, in conjunction with Part I detailed herein, has substantially achieved a complete solution to the Peterson hit problem concerning the polynomial algebra $\mathcal{P}_k=\mathbb F_2[x_1, x_2, \ldots, x_k]$ for general degree $d$ and any number of variables $k,$ using an algorithm developed in the \textsc{SageMath} computer algebra system, subject to sufficient computational memory availability. At the same time, the method is applied to develop explicit algorithms for determining the invariant subspace and the image of the Singer algebraic transfer.


\begin{thebibliography}{99}

\bibitem{Janfada}
A.S. Janfada, \textit{A criterion for a monomial in $P(3)$ to be hit}, Math. Proc. Cambridge Philos. Soc. \textbf{145}, (2008), 587-599.

\bibitem{CH}
P.H. Chon and L.M. Ha, \textit{Lambda algebra and the Singer transfer}, C. R. Math. Acad. Sci. Paris \textbf{349} (2011), 21-23. 

\bibitem{Hung}
N.H.V. Hung, \textit{The cohomology of the Steenrod algebra and representations of the general linear groups}, Trans. Amer. Math. Soc. \textbf{357} (2005), 4065-4089.

\bibitem{HP}
N.H.V. Hung and G. Powell, \textit{The $\mathcal{A}$-decomposability of the Singer construction}, J. Algebra \textbf{517} (2019), 186-206.

\bibitem{Kameko}
M. Kameko,
\textit{Products of projective spaces as Steenrod modules}, PhD. thesis, The Johns Hopkins University, 1990.

\bibitem{Peterson}
F.P. Peterson,
\textit{Generators of  $H^*(\mathbb{R}P^{\infty}\times \mathbb{R}P^{\infty})$ as a module over the Steenrod algebra},  Abstracts Papers Presented Am. Math. Soc.  \textbf{833} (1987), 55-89.


\bibitem{PS}
\DJ.V. Ph\'uc and N. Sum, 
\textit{On the generators of the polynomial algebra as a module over the Steenrod algebra}, C.R.Math. Acad. Sci. Paris \textbf{353} (2015), 1035-1040.

\bibitem{PS2}
\DJ.V. Ph\'uc and N. Sum, \textit{On a minimal set of generators for the polynomial algebra of five variables as a module over the Steenrod algebra},  Acta Math. Vietnam. \textbf{42} (2017), 149-162.

\bibitem{Phuc0}
\DJ.V. Ph\'uc, \textit{On Peterson's open problem and representations of the general linear groups}, J. Korean Math. Soc. \textbf{58} (2021), 643-702.

\bibitem{Phuc}
\DJ.V. Ph\'uc, \textit{On the dimension of $H^{*}((\mathbb Z_2)^{\times t}, \mathbb Z_2)$ as a module over Steenrod ring},  Topol. Appl. \textbf{303} (2021), 107856.


\bibitem{Phuc2}
\DJ.V. Ph\'uc, \textit{On the algebraic transfers of ranks 4 and 6 at generic degrees}. Rend. Circ. Mat. Palermo, II. Ser \textbf{74}, 38 (2025), available online at \url{https://www.researchgate.net/publication/382917122}.

\bibitem{Sum1}
N. Sum, \textit{On the Peterson hit problem of five variables and its applications to the fifth Singer transfer}, East-West J. Math. \textbf{16} (2014), 47-62.

\bibitem{Sum2}   
N. Sum, \textit{On the determination of the Singer transfer}, Vietnam J. Sci., Technol. Eng. \textbf{60} (2018), 3-16.



\bibitem{Sum3}
N. Sum, \textit{The squaring operation and the hit problem for the polynomial algebra in a type of generic degree}, Preprint (2023), ArXiv: 2301.01535.

\bibitem{Sum4}
N. Sum and P.D. Tai, \textit{A minimal set of generators for the polynomial algebra of five variables in a generic degree}, Preprint (2024), ArXiv:2312.16803.

\bibitem{Tin}
N.K. Tin, \textit{The hit problem for the polynomial algebra in five variables and applications},  PhD. thesis, Quy Nhon University, Vietnam, 2017.

\bibitem{Walker-Wood}
G. Walker and R.M.W. Wood, 
\textit{Polynomials and the mod 2 Steenrod Algebra. Volume 1: The Peterson hit problem}, in London Math. Soc.  Lecture Note Ser., Cambridge Univ. Press, 2018.

\bibitem{Wood}
 R.M.W. Wood, 
\textit{Steenrod squares of polynomials and the Peterson conjecture}, Math. Proc. Cambriges Phil. Soc. \textbf{105} (1989), 307-309.

\end{thebibliography}
\end{document}